\documentclass[12pt]{article}
\usepackage{amsmath,amssymb,amsthm,eufrak}
\usepackage[english]{babel}
\usepackage[active]{srcltx}
\usepackage{cite}
\usepackage{bbding}
\newcommand{\TL}{\mbox{\bf{$\!\!$.}}}
\newcommand{\bc}{\begin{center}}
\newcommand{\ec}{\end{center}}

\theoremstyle{theorem}

\newtheorem{theor}{Theorem}[section]
\theoremstyle{theorem}
\newtheorem{propos}{Proposition}[section]
\newtheorem{con}{Condition}[section]
\newtheorem{cor}{Corollary}[section]

\newcommand{\bfn}{\begin{equation}}
\newcommand{\efn}{\end{equation}}

\newcommand{\ov}{\overline}

\newcommand{\fo}{\forall}

\newcommand{\La}{\Lambda}
\newcommand{\la}{\lambda}
\newcommand{\e}{\emptyset}

\newcommand{\Om}{\Omega}

\newcommand{\df}{\stackrel{\triangle}{=}}
\newcommand{\cl}{{\cal L}}
\newcommand{\cp}{{\cal P}}
\newcommand{\cb}{{\cal B}}
\newcommand{\cf}{{\cal F}}
\newcommand{\cu}{{\cal U}}
\newcommand{\ca}{{\cal A}}
\begin{document}

\title{Filters and Ultrafilters as Approximate Solutions in the
Attainability Problems with Constrains of Asymptotic Character}
\author{A.G.~Chentsov}
\date{Institute of Mathematic and Mechanics UrB RAS
chentsov@imm.uran.ru} \maketitle

\abstract{Abstract problems about attainability in topological
spaces are considered. Some nonsequential version of the Warga
approximate solutions is investigated: we use filters and
ultrafilters of measurable spaces. Attraction sets are constructed.

}

\section{Introduction}\setcounter{equation}{0}

This investigation is devoted to questions connected with
attainability under constraints; these constraints can be perturbed.
Under these perturbations, jumps of the attained quality can arise.
If perturbation is reduced to a weakening of the initial standard
constraints, then we obtain some payoff in a result. Therefore,
behavior limiting with respect to the validity of constraints can be
very interesting. But, the investigation of possibilities of the
above-mentioned behavior is difficult. The corresponding
``straight'' methods are connected with constructions of asymptotic
analysis. Very fruitful approach is connected with the extension of
the corresponding problem. For example, in theory of control can be
used different variants of generalized controls formalizable in the
corresponding class of measures very often. In this connection, we
note the known investigations of J.\,Warga (see~\cite{1}). We recall
the notions of precise, generalized, and approximate controls (see
\cite{1}). In connection with this approach, we recall the
investigations of R.V. Gamkrelidze \cite{2}. For problems of impulse
control, we note the original approach of N.N. Krasovskii (see
\cite{3}) connected with the employment of distributions. If is
useful to recall some asymptotic constructions in mathematical
programming (see \cite{4, 5}). We note remarks in \cite{4,5}
connected with the possible employment of nonsequential approximate
(in the Warga terminology) solutions-nets.

The above-mentioned (and many other) investigations concern extremal problems. But, very
important analogs are known for different quality problems. We recall the fundamental
theorem about an alternative in differential games established by N.N. Krasovskii and
A.I. Subbotin \cite{6}. In the corresponding constructions, elements of extensions are
used very active. Moreover, approximate motions were used. The concrete connection of
generalized and approximate elements of the corresponding constructions was realized by
the rule of the extremal displacement of N.N. Krasovskii.

In general, the problem of the combination of generalized and
approximate elements in problems with constraints is very important.
Namely, generalized elements (in particular, generalized controls)
can be used for the representation of objects arising by the limit
passage in the class of approximate elements (approximate
solutions). These limit objects can be consider as attraction
elements. Very often these elements suppose a sequential realization
(see \cite[ch.\,III, IV]{1}). But, in other cases attraction
elements should be defined by more general procedures.

So, we can consider variants of generalized representation of
asymptotic objects. This approach is developed by J. Warga in theory
of control.

Similar problems can arise in distinct sections of mathematics. For example, adherent
points of the filter base in topological space can be considered as attraction elements.
Of course, here nonsequential variants of the limit passage are required very often.

 In the
following, the attainability problem with constraints of asymptotic character is
considered.

Fix two nonempty sets $E$ and $\mathbf{H},$ and an operator
$\mathbf{h}$ from $E$ into $\mathbf{H}.$ Elements of $E$ are
considered as solutions (sometimes controls) and elements of
$\mathbf{H}$ play the role of estimates. We consider $\mathbf{h}$ as
the aim mapping. If we have the set $E_o,\,E_o\subset E,$ of
admissible (in traditional sense) solutions, then $\mathbf{h}^1(E_o)
= \{\mathbf{h}(x):\,x\in E_o\}$ play the role of an attainability
domain in the estimate space. But, we can use another constraints:
instead of $E_o,$ a nonempty family $\mathcal{E}$ of subsets of $E$
is given. In this case, we can use sequences $(x_i)_{i=1}^\infty$ in
$E$ with a special property in the capacity of approximate
solutions. Namely, we require that the sequence $(x_i)_{i=1}^\infty$
has the following property: for any $\widetilde{E}_o\in
\mathcal{E},$ the inclusion $x_j\in \widetilde{E}_o$ takes place
from a certain index (i. e. for $j \geqslant \widetilde{j} _o,$
where $\widetilde{j}_o$ is a fixed index depending on
$\widetilde{E}_o).$ For such solutions we obtain the sequences
$\bigl(\mathbf{h}(x_i)\bigl)_{i=1}^\infty$ in $\mathbf{H}.$ If
$\mathbf{H}$ is equipped with a topology $\mathbf{t},$ then we can
consider the limits of such sequences
$\bigl(\mathbf{h}(x_i)\bigl)_{i=1}^\infty$ as attraction elements
(AE) in $(\mathbf{H},\mathbf{t}).$ Of course, our AE are
``sequential'': we use the limit passage in the class of sequences.
This approach can be very limiting. The last statement is connected
both with our family $\mathcal{E}$ and with topology $\mathbf{t}.$
The corresponding examples are known: see \cite{30, 31}. In many
cases, the more general variants of the limit passage are required.
Of course, we can consider nets $(x_\alpha)$ in $E$ and, as a
corollary, the corresponding nets $\bigl(\mathbf{h}(x_\alpha)\bigl)$
in $\mathbf{H}.$ In addition, the basic requirement of admissibility
it should be preserved: for any $\widetilde{E}_o\in \mathcal{E},$
the inclusion $x_\alpha \in \widetilde{E}_o$ is valid starting from
a certain index. With the employment of such nets, we can realize
new AE; this effect takes place in many examples.

But, the representation of the ``totality'' of above-mentioned
$(\mathcal{E}$-admissible) nets as a set is connected with
difficulties. Really, any net in the set $E$ is defined by a mapping
from a nonempty directed set (DS) $\mathcal{D}$ into $E.$ The
concrete choice of $\mathcal{D}$ is arbitrary ($\mathcal{D}$ is a
nonempty set). Therefore we have the very large ``totality'' of nets
with the point of view of traditional Zermelo axiomatics. But, this
situation can be corrected by the employment of filters of $E:$ it
is possible to introduce the set of all $\mathcal{E}$-admissible
filters of the set $E.$ In addition, the $\mathcal{E}$-admissibility
of a filter $\mathcal{F}$ is defined by the requirement
$\mathcal{E}\subset \mathcal{F}.$ So, we can consider nonsequential
approximate solutions (analogs of sequential approximate solutions
of Warga) as filters $\mathcal{F}$ of $E$ with the property
$\mathcal{E}\subset \mathcal{F}.$ Moreover, we can be restricted to
the employment of only ultrafilters (maximal filters) with the
above-mentioned property. In two last cases, we obtain two variants
of the set of admissible nonsequential  approximate solutions
defined in correspondence with Zermelo axiomatics. In our
investigation, such point of view is postulated. And what is more,
we give the basic attention to the consideration of ultrafilters.
Here, the important property of a compactness arises. Namely, the
corresponding space of ultrafilters is equipped with a compact
topology. This permits to consider ultrafilters as generalized
elements (GE) too (we keep in mind the above-mentioned
classification of Warga).

The basic difficulty is connected with a realizability: the
existence of free ultrafilters (for which effects of an extension
are realized) is established only with the employment of axiom of
choice. Roughly speaking, free ultrafilters are ``invisible''. This
property is connected with ultrafilters of the family of all subsets
of the corresponding ``unit''. But, we can to consider ultrafilters
of measurable spaces with algebras and semialgebras of sets. We note
that some measurable spaces admitting the representation of all such
ultrafilters are known (see, for example, \cite[$\S$\,7.6]{32}; in
addition, the unessential transformation with the employment of
finitely additive (0,1)-measures is used).

\section{General notions and designations}

We use the standard set-theoretical symbolics including quantors and
propositional connectives; as usually $\exists !$ replaces the
expression ``there exists and unique'', $\df$ is the equality by
definition. In the following, for any two objects $x$ and $y,\, \{x;
y\}$ is the unordered pair of $x$ and $y$ (see \cite[ch.\,II]{33}).
Then, $\{x\} \df \{x; x\}$ is singleton containing an object $x.$ Of
course, for any objects $x$ and $y\ \ (x, y) \df \bigl\{\{x\}; \{x;
y\}\bigl\}$ is the ordered pair of objects $x$ and $y;$ here, we
follow to \cite[ch.\,II]{33}. By $\emptyset$ we denote the empty
set. By a family we call a set all elements of which are sets.

By $\mathcal{P}(X)$ we denote the family of all subsets of a set
$X;$ then, $\cp^\prime(X) \df\,\cp(X)\setminus \{\emptyset\}$ is the
family of all nonempty subsets of $X.$ Of course, for any set $A,$
in the form of $\mathcal{P}^\prime\bigl(\mathcal{P}(A)\bigl)$ and
$\mathcal{P}^\prime\bigl(\mathcal{P}^\prime(A)\bigl),$ we have the
family of all nonempty subfamilies of $\mathcal{P}(A)$ and
$\mathcal{P}^\prime(A)$ respectively.

If $X$ is a set, then we denote by $\mathrm{Fin}(X)$ the family of all finite sets of
$\mathcal{P}^\prime(X);$ then $(\mathrm{FIN})[X] \df
\mathrm{Fin}(X)\,\cup\,\{\emptyset\}$ is the family of all finite subsets of $X.$

For any sets $A$ and $B,$ we denote by $B^A$ the set of all mappings from $A$ into $B.$
If $A$ and $B$ are sets, $f\in B^A,$  and $C\in \mathcal{P}(A),$ then $f^1(C) \df
\{f(x):\,x\in C\}\in \mathcal{P}(B)$ (the image of $C$ under the operation $f)$ and $(f
|\,C)\in B^C$ is the usual $C$-restriction of $f:(f |\,C)(y) \df f(y)\ \ \forall y\in C.$
In the following, $\mathbb{N}\df \{1;2;\ldots\}$ and $\mathbb{R}$ is the real line;
$\mathbb{N}\subset \mathbb{R}.$ Of course, we use the natural order $\leqslant$ of
$\mathbb{R}.$ If $n\in \mathbb{N},$ then
$$\overline{1,n}\df \{i\in \mathbb{N} |\,i\leqslant n\}.$$

{\bf Transformations of families.} For any nonempty family $\mathcal{A}$ and a set $B,$
we suppose that
$$\mathcal{A}\bigl|_B \df \{A\,\cap\,B:\,A\in \mathcal{A}\}\in
\mathcal{P}^\prime\bigl(\mathcal{P}(B)\bigl).$$ If $X$ and $Y$ are
sets and $f\in Y^X,$  then we suppose that
\begin{equation}\label{2.1}
\begin{array}{c}\Bigl(f^1[\mathcal{X}] \df \{f^1(A):\,A\in
\mathcal{X}\}\ \ \forall \mathcal{X}\in
\mathcal{P}^\prime\bigl(\mathcal{P}(X)\bigl)\Bigl)\,\&\\ \&\,
\,\Bigl(f^{-1}[\mathcal{Y}] \df \{f^{-1}(B):\,B\in \mathcal{Y}\}\ \
\forall \mathcal{Y}\in
\mathcal{P}^\prime\bigl(\mathcal{P}(Y)\bigl)\Bigl);
\end{array}
\end{equation}
of course, in (\ref{2.1}) nonempty families are defined.

If $\mathcal{E}$ is a family, then we suppose that $$\{\cup\} (\mathcal{E}) \df
\Bigl\{\bigcup_{H\in\mathcal{H}}H:\,\mathcal{H}\in \mathcal{P}(\mathcal{E})\Bigl\}$$ (we
keep in mind that $\mathcal{P}(\mathcal{E})$ is a nonempty set $(\emptyset\in
\mathcal{P}(\mathcal{E}))$ and, for $\mathcal{R}\in
\mathcal{P}(\mathcal{E}),\,\mathcal{R}$ is a family) and $$\{\cap\} (\mathcal{E}) \df
\Bigl\{\bigcap_{H\in\mathcal{H}}H:\,\mathcal{H}\in
\mathcal{P}^\prime(\mathcal{E})\Bigl\}$$ (of course, for $\mathcal{H}\in
\mathcal{P}^\prime(\mathcal{E}),\,\mathcal{H}$ is a nonempty family); moreover
$$\{\cap\}_\mathbf{f} (\mathcal{E}) \df
\Bigl\{\bigcap_{H\in\mathcal{K}}H:\,\mathcal{K}\in \mathrm{Fin}(\mathcal{E})\Bigl\}.$$
So, for any nonempty family $\mathcal{E},$ we obtain that $$\biggl(\{\cup\} (\mathcal{E})
\in \mathcal{P}^\prime\Bigl(\mathcal{P}\bigl(\bigcup\limits_{E\in
\mathcal{E}}E\bigl)\Bigl):\,\mathcal{E}\subset \{\cup\}(\mathcal{E})\biggl)\,\&$$
$$\&\,\biggl(\{\cap\} (\mathcal{E}) \in
\mathcal{P}^\prime\Bigl(\mathcal{P}\bigl(\bigcup\limits_{E\in
\mathcal{E}}E\bigl)\Bigl):\,\mathcal{E}\subset \{\cap\}(\mathcal{E})\biggl)\,\&$$
$$\&\,\biggl(\{\cap\}_\mathbf{f} (\mathcal{E}) \in
\mathcal{P}^\prime\Bigl(\mathcal{P}\bigl(\bigcup\limits_{E\in
\mathcal{E}}E\bigl)\Bigl):\,\mathcal{E}\subset \{\cap\}_\mathbf{f}(\mathcal{E})\biggl);$$
of course, $\{\cap\}_\mathbf{f}(\mathcal{E}) \subset \{\cap\}(\mathcal{E}).$

{\bf Special families.} Let $I$ be a set.  Then, we suppose that
\begin{equation}\label{2.2}\begin{array}{c} \pi[I]
\df \{\cl\in
\mathcal{P}^\prime\bigl(\mathcal{P}(I)\bigl)\,\bigl|\,(\emptyset\in
\cl)\,\&\,(I\in \cl)\,\&\\ \&\,(A \cap B\in \cl\ \ \fo A\in \cl\ \
\fo B\in \cl)\};\end{array}
\end{equation} elements of (\ref{2.2}) are called $\pi$-systems with
``zero'' and ``unit''. Moreover,
\begin{equation}\label{2.3}\begin{array}{c}(\mathrm{LAT})[I] \df \bigl\{\cl\in
\mathcal{P}^\prime\bigl(\mathcal{P}(I)\bigl)\,\bigl|\,(\emptyset\in \cl)\,\&\\
\&\,\bigl(\fo A\in\cl\ \ \fo B\in \cl\ \ (A \cup B\in \cl)\,\&\,(A \cap B\in
\cl)\bigl)\bigl\};\end{array}
\end{equation} elements of (\ref{2.3}) are lattices of subsets of $I$ (with
``zero''). Finally,
 \bfn\label{2.4}(\mathrm{LAT})_o[I] \df \bigl\{\cl\in (\mathrm{LAT})[I]\,\bigl|\,I\in
 \cl\bigl\}\in \mathcal{P}(\pi[I]).\efn Of course, in (\ref{2.4}) lattices of sets with
 ``zero'' and ``unit'' are introduced. We note that $\cl \cup \{I\} \in (\mathrm{LAT})_o[I]\ \
 \fo \cl\in (\mathrm{LAT})[I].$ Of course, \begin{equation}\label{2.5}\begin{array}{c}
 (\mathrm{top})[I] \df \bigl\{\tau\in
 \pi[I]\,\bigl|\,\bigcup\limits_{G\in\,\mathcal{G}}G\in \tau\ \ \fo \mathcal{G}\in
 \mathcal{P}^\prime(\tau)\bigl\} = \\ = \bigl\{\tau\in \pi[I]\,\bigl|\,\bigcup\limits_{G\in\,
 \mathcal{G}}G\in \tau\ \ \fo \mathcal{G}\in \mathcal{P}(\tau)\bigl\}\end{array}
\end{equation} is the set of
 all topologies of $I.$ If $\tau\in (\mathrm{top})[I],$ then the pair $(I,\tau)$ is a
 topological space (TS);
 \begin{equation}\label{2.6}
\begin{array}{c}(\mathrm{clos})[I] \df \bigl\{\mathcal{F}\in
\mathcal{P}^\prime\bigl(\mathcal{P}(I)\bigl)\,|\,(\emptyset\in \mathcal{F})\,\&\\
\&\,(I\in
\mathcal{F})\,\&\,(A \cup B\in \mathcal{F}\ \fo A\in \mathcal{F}\ \fo B\in \mathcal{F})\,\&\\
\&\,\bigl(\bigcap\limits_{H\in \mathcal{H}}H\in \mathcal{F}\ \ \fo \mathcal{H}\in
\mathcal{P}^\prime(\mathcal{F})\bigl)\bigl\};\end{array}
\end{equation}
in (\ref{2.6}) we have families dual with respect to topologies. It is obvious that
\begin{equation}\label{2.6`}\begin{array}{c}\bigl((\mathrm{top})[I] \subset
(\mathrm{LAT})_o[I]\bigl)\,\&\,\bigl((\mathrm{clos})[I]\subset
(\mathrm{LAT})_o[I]\bigl).
\end{array}
\end{equation} We suppose that
$\mathbf{C}_I:\,\mathcal{P}^\prime\bigl(\mathcal{P}(I)\bigl) \rightarrow
\mathcal{P}^\prime\bigl(\mathcal{P}(I)\bigl)$ is the mapping for which
\begin{equation}\label{2.7}\begin{array}{c}\mathbf{C}_I(\mathcal{H}) \df \{I\setminus
H:\,H\in \mathcal{H}\}\ \ \fo \mathcal{H}\in
\mathcal{P}^\prime\bigl(\mathcal{P}(I)\bigl).\end{array}
\end{equation}
From (\ref{2.5}) -- (\ref{2.7}), we obtain the following properties:
\begin{equation}\label{2.8}\begin{array}{c}\Bigl(\mathbf{C}_I\bigl(\mathbf{C}_I
(\mathcal{H})\bigl) = \mathcal{H}\ \ \fo \mathcal{H}\in
\mathcal{P}^\prime\bigl(\mathcal{P}(I)\bigl)\Bigl)\,\&\\
\&\,\bigl(\mathbf{C}_I(\tau)\in
(\mathrm{clos})[I]\ \ \fo \tau\in(\mathrm{top})[I]\bigl)\,\&\\
\&\,\bigl(\mathbf{C}_I(\mathcal{F})\in (\mathrm{top})[I]\ \ \fo
\mathcal{F}\in (\mathrm{clos})[I]\bigl).
\end{array}
\end{equation}
We note that $\mathcal{P}(I) \in (\mathrm{top})[I] \cap (\mathrm{clos})[I];$ in addition,
$$\mathbf{C}_I\bigl(\mathcal{P}(I)\bigl) = \mathcal{P}(I).$$  Of course, in (\ref{2.8}), we have
(in particular) the natural duality used in general topology. Let $$(c - \mathrm{top})[I]
\df \bigl\{\tau\in (\mathrm{top})[I] |\,\fo \xi\in \mathcal{P}^\prime(\tau)\ \ \bigl(I =
\bigcup\limits_{G\in\,\xi}G\bigl)\Longrightarrow$$
$$\Longrightarrow \bigl(\exists \mathcal{K}\in \mathrm{Fin}(\xi):\,I =
\bigcup\limits_{G\in\,\mathcal{K}}G\bigl)\bigl\}.$$ (the set of all
compact topologies of $I).$ Now, we introduce in consideration
algebras of  sets. Namely,
\begin{equation}\label{2.9}\begin{array}{c}(\mathrm{alg})[I] \df \{\cl\in \pi[I]\,|\,
I\setminus L\in \cl\ \ \fo L\in  \cl\}\subset
(\mathrm{LAT})_o[I].\end{array}\end{equation}In connection with
(\ref{2.9}), we note that
$$\{L\in \cl\,|\,I\setminus L\in \cl\}\in (\mathrm{alg})[I]\ \ \fo \cl\in
(\mathrm{LAT})_o[I].$$ If $\cl\in (\mathrm{alg})[I],$ then $(I,\cl)$
is a measurable space with an algebra of sets.

If $\cl\in \pi[I],\,n\in \mathbb{N}$ and $A\in \mathcal{P}(I),$ then by $\Delta_n(A,\cl)$
we denote the set of all mappings $$(L_i)_{i\in
\overline{1,n}}:\,\overline{1,n}\longrightarrow \cl$$ for each of which: 1) $A =
\bigcup\limits_{i=1}^n L_i;$ 2) $L_{i_1} \cap L_{i_2} = \emptyset\ \ \fo i_1\in
\overline{1,n}\ \ \fo i_2\in \overline{1,n}\setminus \{i_1\}.$ Then
\begin{equation}\label{2.10}\begin{array}{c}\Pi[I] \df \{\cl\in \pi[I] |\,\fo L\in\cl\
\ \exists n\in \mathbb{N}:\,\Delta_n(I\setminus L,\cl)\neq \emptyset\}\end{array}
\end{equation}
is the set of all semialgebras of subsets of $I.$ Of course,
$$(\mathrm{alg})[I] = \{\cl\in\Pi [I]\,|\,I\setminus L\in \cl\ \ \fo L\in
\cl\};$$ see (\ref{2.9}). If we have a semialgebra of subsets of
$I,$ then algebra generated by the initial semialgebra is realized
very simply: for any $\cl\in\Pi[I],$ $$\mathbf{a}_I^o(\cl) \df
\{A\in \mathcal{P}(I)\,|\,\exists n\in \mathbb{N}:\,\Delta_n(A,\cl)
\neq \emptyset\}\in (\mathrm{alg})[I]$$ has the properties: 1)
$\cl\subset \mathbf{a}_I^o(\cl);$ 2) $\fo \mathcal{A}\in
(\mathrm{alg})[I]$ $$(\cl\subset \mathcal{A}) \Longrightarrow
\bigl(\mathbf{a}_I^o(\cl)\subset \mathcal{A}\bigl).$$ Now, we
introduce some notions important for constructions of general
topology. Namely, we consider topological bases of two types:
\begin{equation}\label{2.11}\begin{array}{c}
(\mathrm{op}-\mathrm{BAS})[I]\df \bigl\{\mathcal{B}\in
\mathcal{P}\bigl(\mathcal{P}(I)\bigl)\,|\,(I =
\bigcup\limits_{B\in\,\mathcal{B}}B)\,\&\,\bigl(\fo B_1\in
\mathcal{B}\\ \fo B_2\in \mathcal{B}\ \fo x \in B_1 \cap B_2\ \
\exists B_3\in \mathcal{B}:\,(x\in B_3)\,\&\,(B_3\subset B_1 \cap
B_2)\bigl)\bigl\},\end{array}\end{equation}
\begin{equation}\label{2.12}\begin{array}{c}
(\mathrm{cl}-\mathrm{BAS})[I]\,\df\,\Bigl\{\mathcal{B}\in
\mathcal{P}^\prime\bigl(\mathcal{P}(I)\bigl)\,|\,(I \in
\mathcal{B})\,\&\,(\bigcap\limits_{B\in\,\mathcal{B}}B =
\emptyset)\,\&\\ \&\,\bigl(\fo B_1\in \mathcal{B}\ \ \fo B_2\in
\mathcal{B}\ \ \fo x \in I\setminus (B_1 \cup B_2)\\ \exists B_3\in
\mathcal{B}:\,(B_1 \cup B_2 \subset B_3)\,\&\,(x\notin
B_3)\bigl)\Bigl\},\end{array}
\end{equation}
Of course, $(\mathrm{op}-\mathrm{BAS})[I] = \{\beta\in
\mathcal{P}\bigl(\mathcal{P}(I)\bigl)\,|\,\{\cup\}(\beta) \in
(\mathrm{top})[I]\}.$ In connection with (\ref{2.11}), we suppose
that $(\mathrm{op}-\mathrm{BAS})_\emptyset[I] \df \{\mathcal{B}\in
(\mathrm{op}-\mathrm{BAS})[I]\,|\,\emptyset \in \mathcal{B}\},\
(\mathrm{op}-\mathrm{BAS})_\emptyset[I]\subset
\mathcal{P}^\prime\bigl(\mathcal{P}(I)\bigl);$
$$\widetilde{\mathcal{B}} \cup \{\emptyset\}\in
(\mathrm{op}-\mathrm{BAS})_\emptyset[I]\ \ \fo
\widetilde{\mathcal{B}}\in (\mathrm{op}-\mathrm{BAS})[I].$$
Moreover, the following obvious property is valid:
$$\{\cup\}(\mathcal{B}) = \{\cup\}(\mathcal{B}\cup \{\emptyset\})\ \
\fo \mathcal{B}\in (\mathrm{op}-\mathrm{BAS})[I].$$ We note the
natural connection of open and closed bases:
\begin{equation}\label{2.13}\begin{array}{c}\bigl(\mathbf{C}_I(\mathcal{B})\in (\mathrm{op}-
\mathrm{BAS})_\emptyset[I]\ \ \fo \mathcal{B}\in
(\mathrm{cl}-\mathrm{BAS})[I]\bigl)\,\&\\
\&\,\bigl(\mathbf{C}_I(\beta)\in (\mathrm{cl}-\mathrm{BAS})[I]\ \
\fo \beta\in
(\mathrm{op}-\mathrm{BAS})_\emptyset[I]\bigl).\end{array}
\end{equation}
Along with (\ref{2.13}), we note the following important property:
\begin{equation}\label{2.14}\{\cap\}(\mathcal{B})\in (\mathrm{clos})[I]\ \
\fo \mathcal{B}\in (\mathrm{cl}-\mathrm{BAS})[I].
\end{equation}
From (\ref{2.8}) and (\ref{2.14}), we obtain the obvious statement:
\begin{equation}\label{2.15}\mathbf{C}_I\bigl(\{\cap\}(\mathcal{B})\bigl)\in
 (\mathrm{top})[I]\ \ \fo \mathcal{B}\in (\mathrm{cl}-\mathrm{BAS})[I].
\end{equation}
So, closed bases can be used (see (\ref{2.15})) for topologies constructing. We note the
following obvious property  (here we use (\ref{2.13}) and (\ref{2.15})):
\begin{equation}\label{2.16}\mathbf{C}_I\bigl(\{\cap\}(\mathcal{B})\bigl) =
\{\cup\}\bigl(\mathbf{C}_I(\mathcal{B})\bigl)\ \ \fo \mathcal{B}\in
(\mathrm{cl}-\mathrm{BAS} )[I].
\end{equation}
Of course, in (\ref{2.16}), we use the usual duality property connected with (\ref{2.13})
-- (\ref{2.15}).

{\bf Some additions.} In the following, we suppose that
\begin{equation}\label{2.17}(\mathcal{D}-\mathrm{top})[I] \df
\{\tau\in (\mathrm{top})[I]\,|\,\{x\}\in \mathbf{C}_I[\tau]\  \ \fo x\in I\};
\end{equation}
if $\tau\in (\mathcal{D}-\mathrm{top})[I],$ then TS $(I,\tau)$ is called $T_1$-space. We
use (\ref{2.17}) under investigation of properties of topologies on ultrafilter spaces.

Finally, we suppose that $(\mathrm{LAT})^o[I] \df \{\cl\in
(\mathrm{LAT})_o[I]\,|\,\{x\}\in \cl\ \ \fo x\in I\}.$ So, we
introduce ``continuous'' lattices.

\section{Nets and filters as approximate solutions under constraints of asymptotic character}
\setcounter{equation}{0}

In this section, we fix a nonempty set $\mathbb{E}$ considered (in particular) as the
space of usual solutions. We consider families $\mathcal{E}\in
\mathcal{P}^\prime\bigl(\cp(\mathbb{E})\bigl)$ as constraints of asymptotic character. Of
course, in this case, we use asymptotic version of solutions. The simplest variant is
realized by the employment of sequences in $\mathbb{E}:$ in the set
$\mathbb{E}^\mathbb{N},$ the set of $\mathcal{E}$-admissible sequences (see section 1) is
selected. It is logical to generalize this approach: we keep in mind the employment of
nets.  Later, we introduce some definitions connected with the Moore-Smith convergence.
But, before we consider the filter convergence.

We denote by $\beta[\mathbb{E}]$ (by $\beta_o[\mathbb{E}])$ the set of all families
$\mathcal{B}\in \mathcal{P}^\prime\bigl(\mathcal{P}(\mathbb{E})\bigl)$ (families
$\mathcal{B}\in \mathcal{P}^\prime\bigl(\mathcal{P}^\prime(\mathbb{E})\bigl))$ for which
$$\fo B_1\in \mathcal{B}\ \ \fo B_2\in \mathcal{B}\ \ \exists B_3\in \mathcal{B}:\,
B_3\subset B_1 \cap B_2;$$ $\beta_o[\mathbb{E}]\subset \beta[\mathbb{E}].$ Then,
$\beta_o[\mathbb{E}]$ is the set of all filter bases on $\mathbb{E}.$ By
$\mathfrak{F}[\mathbb{E}]$ we denote the set of all filters on $\mathbb{E}:$
\begin{equation}\label{3.1}\begin{array}{c}\mathfrak{F}[\mathbb{E}] \df \bigl\{\mathcal{F}\in
\mathcal{P}^\prime\bigl(\mathcal{P}^\prime(\mathbb{E})\bigl)\,|\,(A
\cap B \in \mathcal{F}\ \ \fo A\in \mathcal{F}\ \ \fo B\in
\mathcal{F})\,\&\\ \&,\bigl(\{H\in
\mathcal{P}(\mathbb{E})\,|\,F\subset H\}\subset \mathcal{F}\ \ \fo
F\in \mathcal{F}\bigl)\bigl\}.\end{array}
\end{equation}
Using (\ref{3.1}), we introduce the set
$\mathfrak{F}_\mathbf{u}[\mathbb{E}]$ of all ultrafilters on
$\mathbb{E}:$
\begin{equation}\label{3.2}\mathfrak{F}_\mathbf{u}[\mathbb{E}] \df
\bigl\{\mathcal{U}\in \mathfrak{F}[\mathbb{E}]\,|\,\fo
\mathcal{F}\in \mathfrak{F}[\mathbb{E}]\ \ \bigl((\mathcal{U}\subset
\mathcal{F}) \Longrightarrow (\mathcal{U} =
\mathcal{F})\bigl)\bigl\}.\end{equation} In connection with
(\ref{3.1}) and (\ref{3.2}), see in particular \cite[ch. I]{34}. In
addition, \begin{equation}\label{3.3}
(\mathbb{E}-\mathrm{\mathbf{fi}})[\mathcal{B}] \df \{H\in
\mathcal{P}(\mathbb{E})\,|\,\exists B\in \mathcal{B}:\,B\subset
H\}\in \mathfrak{F}[\mathbb{E}]\ \ \fo \mathcal{B}\in
\beta_o[\mathbb{E}].
\end{equation}
By (\ref{3.3}) we define the filter on $\mathbb{E}$ generated by a
base of $\beta_o[\mathbb{E}].$

If $\mathcal{E}\in \mathcal{P}^\prime\bigl(\cp(\mathbb{E})\bigl),$ then by
$\mathfrak{F}_o[\mathbb{E}|\,\mathcal{E}]$ (by
$\mathfrak{F}_\mathbf{u}^o[\mathbb{E}|\,\mathcal{E}])$ we denote the set of all filters
$\mathcal{F}\in \mathfrak{F}[\mathbb{E}]$ (ultrafilters $\mathcal{F}\in
\mathfrak{F}_\mathbf{u}[\mathbb{E}])$ such that $\mathcal{E}\subset \mathcal{F}.$ Then,
for any filter $\mathcal{F}_*\in \mathfrak{F}[\mathbb{E}],$ we have
$\mathfrak{F}_\mathbf{u}^o[\mathbb{E} |\,\mathcal{F}_*]\in
\cp^\prime(\mathfrak{F}_\mathbf{u}[E])$ and what is more $\mathcal{F}_*$ is the
intersection of all  ultrafilters $\mathcal{U} \in \mathfrak{F}_\mathbf{u}^o[\mathbb{E}
|\,\mathcal{F}_*];$ see \cite[ch. I]{34}.

If a family $\mathcal{E}\in \cp^\prime\bigl(\cp(E)\bigl)$ is
considered as the constraint of asymptotic character, then
ultrafilters $\mathcal{U}\in \mathfrak{F}_\mathbf{u}^o[E
|\,\mathcal{E}]$ are considered as (nonsequential) approximate
solutions; of course, filters $\mathcal{F}\in
\mathfrak{F}_o[\mathbb{E}\,|\,\mathcal{E}]$ can be considered in
this capacity also. But, ultrafilters have better properties;
therefore, now we are restricted to employment of ultrafilters as
approximate solutions.

{\bf The filter of neighborhoods.} If $\tau \in
(\mathrm{top})[\mathbb{E}]$ and $x\in \mathbb{E},$ then
$$N_\tau^o(x) \df \{G\in \tau\,|\,x\in G\}\in \beta_o[\mathbb{E}]$$
and $N_\tau(x) \df (\mathbb{E}-\mathrm{\mathbf{fi}})[N_\tau^o(x)];$
of course, $N_\tau(x) \in \mathfrak{F}[\mathbb{E}]$ in
correspondence with (\ref{3.3}). We were introduce the filter of
neighborhoods of $x$ in the sense of \cite[ch. I]{34}. In the
following,
$$\mathrm{cl}(A,\tau) \df \{x\in \mathbb{E}\,|\,A \cap H \neq \emptyset\ \ \fo H\in
N_\tau(x)\}\ \ \fo \tau\in (\mathrm{top})[\mathbb{E}]\ \ \fo A\in \cp(\mathbb{E}).$$  So,
we introduce the closure operation in a TS. Moreover, we suppose that
\begin{equation}\label{2.18}\begin{array}{c}(x-\mathrm{bas})[\tau]
\df \{\mathcal{B}\in \cp\bigl(N_\tau(x)\bigl)\,|\,\fo A\in
N_\tau(x)\ \ \exists B\in \mathcal{B}:\,B\subset A\}\\
\fo \tau\in (\mathrm{top})[\mathbb{E}]\ \ \fo x\in \mathbb{E}.\end
{array}\end{equation}

{\bf The filter convergence.} We follow to \cite[ch. I]{34}. Suppose
that $\fo \tau\in (\mathrm{top})[\mathbb{E}]\ \ \fo \mathcal{B}\in
\beta_o[\mathbb{E}]\ \ \fo x\in \mathbb{E}$
\bfn\label{3.4}(\mathcal{B}{\stackrel{\tau}\Longrightarrow} x)
{\stackrel{\mathrm{def}}\Longleftrightarrow} \bigl(N_\tau(x) \subset
(\mathbb{E}-\mathrm{\mathbf{fi}})[\mathcal{B}]\bigl).\efn In
addition, $\mathfrak{F}[\mathbb{E}]\subset \beta_o[\mathbb{E}];$ see
(\ref{3.1}). Therefore, we can use (\ref{3.4}) in the case of
$\mathcal{B} = \mathcal{F},$ where $\mathcal{F}\in
\mathfrak{F}[\mathbb{E}];$ we note that
$(\mathbb{E}-\mathrm{\mathbf{fi}})[\mathcal{F}] = \mathcal{F}.$
Then, by (\ref{3.4}) $\fo \tau\in (\mathrm{top})[\mathbb{E}]\ \ \fo
\mathcal{F}\in \mathfrak{F}[\mathbb{E}]\ \ \fo x\in E$
\bfn\label{3.5}(\mathcal{F}{\stackrel{\tau}\Longrightarrow} x)
\Longleftrightarrow \bigl( N_\tau(x) \subset \mathcal{F}\bigl).\efn
Of course, it is possible to use the variant of (\ref{3.5})
corresponding to the case $\mathcal{F} = \mathcal{U},$ where
$\mathcal{U}\in \mathfrak{F}_\mathbf{u}[E].$

{\bf Nets and the Moore-Smith convergence.} On the basis of
(\ref{3.5}), we can to introduce the standard Moore-Smith
convergence of nets. We call a net in the set $\mathbb{E}$ arbitrary
triplet $(D,\preceq,f),$ where $(D,\preceq)$ is a nonempty DS and
$f\in \mathbb{E}^D.$ If $(D,\preceq,f)$ is a net in the set
$\mathbb{E},$ then
\begin{equation}\label{3.6}\begin{array}{c}(\mathbb{E}-\mathrm{ass})[D;\preceq;f] \df
\bigl\{V\in \cp(\mathbb{E})\,|\,\exists d\in D\ \ \fo \delta\in D\\  \bigl((d \preceq
\delta) \Longrightarrow (f(\delta)\in V)\bigl)\bigl\}\in
\mathfrak{F}[\mathbb{E}];\end{array}
\end{equation} we obtain the filter of $\mathbb{E}$ associated with $(D,\preceq,f).$ Now,
for any topology $\tau\in (\mathrm{top})[\mathbb{E}],$ a net
$(D,\preceq,f)$ in the set $\mathbb{E},$ and $x\in\mathbb{E},$ we
suppose that
\bfn\label{3.7}\bigl((D,\preceq,f){\stackrel{\tau}\longrightarrow}\,
x\bigl){\stackrel{\mathrm{def}}\Longleftrightarrow}\bigl((\mathbb{E}-\mathrm{ass})[D;\preceq;f]
{\stackrel{\tau}\Longrightarrow}\, x\bigl).\efn From (\ref{3.5}) and
(\ref{3.6}), we obtain that (\ref{3.7}) is the ``usual'' Moore-Smith
convergence (see \cite[ch. 2]{35}). Of course, any sequence
$\mathbf{x}\df (x_i)_{i\in \mathbb{N}}\in \mathbb{E}^\mathbb{N}$
generates the net $(\mathbb{N},\le,\mathbf{x}),$ where $\le$ is the
usual order of $\mathbb{N}.$

If $\mathcal{E}\in \cp^\prime\bigl(\cp(E)\bigl),$ then a net $(D,\preceq,f)$ in
$\mathbb{E}$ is called $\mathbb{E}$-admissible if $\mathcal{E}\subset
(\mathbb{E}-\mathrm{ass})[D;\preceq;f].$ In this case, $\mathcal{E}$ can be considered as
a constraint of asymptotic character and $(D,\preceq,f)$ plays the role of nonsequential
(generally speaking) approximate solution.

In conclusion, we note that
\begin{equation}\label{3.8}(\mathbb{E}-\mathrm{ult})[x] \df \{H\in
\cp(\mathbb{E})\,|\,x\in H\}\in \mathfrak{F}_\mathbf{u}[\mathbb{E}]\ \ \fo x\in
\mathbb{E}.
\end{equation}
In (\ref{3.8}), trivial ultrafilters are defined.

\section{Attraction sets}
\setcounter{equation}{0}

In this section, we construct nonsequential (generally speaking)
attraction sets (AS) using  different variants of the representation
of approximate solutions. Since nets are similar to sequences very
essential, we begin our consideration with the representation (of
AS) using nets.

For a brevity, in this section we fix following two nonempty sets:
$X$ and $Y.$ In addition, under $f\in Y^X$ and $\mathcal{B}\in
\beta_o[X],$
\begin{equation}\label{4.1}f^1[\mathcal{B}]\in \beta_o[Y];
\end{equation}
of course, in (\ref{4.1}), we can use a filter or ultrafilter instead of $\mathcal{B}.$
In addition, the important property takes place: if $f\in Y^X$ and $\mathcal{B}\in
\beta_o[X],$ then
\begin{equation}\label{4.2}\bigl((X-\mathrm{\mathbf{fi}})[\mathcal{B}]\in
\mathfrak{F}_\mathbf{u}[X]\bigl) \Longrightarrow
\bigl((Y-\mathrm{\mathbf{fi}})\bigl[f^1[\mathcal{B}]\bigl]\in
\mathfrak{F}_\mathbf{u}[Y]\bigl).
\end{equation}
So, by (\ref{4.2}) image of an ultrafilter base is an ultrafilter
base. Of course, the image of an ultrafilter is an ultrafilter  base
also.

Introduce AS: if $f\in Y^X,\,\tau\in (\mathrm{top})[Y]$ and
$\mathcal{X}\in \cp^\prime\bigl(\cp(X)\bigl),$ then by
$(\mathbf{as})[X;Y;\tau;f;\mathcal{X}]$ we denote the set of all
$y\in Y$ for each of which there exists a net $(D,\preceq,g)$ in the
set $X$ such that
\begin{equation}\label{4.3}\bigl(\mathcal{X}\subset
(X-\mathrm{ass})[D;\preceq;g]\bigl)\,\&\,\bigl((D,\preceq,f\circ
g){\stackrel{\tau}{\longrightarrow}}\,y\bigl);
\end{equation}
we consider $(\mathbf{as})[X;Y;\tau;f;\mathcal{X}]$ as AS. In this definition, we use
nets. But, for any filter $\mathcal{F}\in \mathfrak{F}[X]$ there exists a net
$(D,\preceq,g)$ in the set $X$ for which $\mathcal{F} = (X-\mathrm{ass})[D;\preceq;g]$
(see \cite[$\S$\,1.6]{36}).

\begin{propos}\label{p4.1}{\TL} For any $f\in Y^X,\,\tau\in (\mathrm{top})[Y]$ and $\mathcal{X}\in
\cp^\prime\bigl(\cp(X)\bigl)$
\begin{equation}\label{4.4}\begin{array}{c}(\mathrm{\mathbf{as}})[X;Y;\tau;f;\mathcal{X}]
= \{y\in Y\,|\,\exists \mathcal{F}\in
\mathfrak{F}_o[X\,|\,\mathcal{X}]:\,f^1[\mathcal{F}]\,
{\stackrel{\tau}{\Longrightarrow}}\,y\}.\end{array}
\end{equation}
\end{propos}
\begin{proof} Fix $f\in Y^X, \tau\in(\mathrm{top})[Y],$ and
$\mathcal{X}\in
 \cp^\prime\bigl(\cp(X)\bigl).$ Suppose that $\mathbf{A}$ and $\mathbf{B}$ are the sets
 on the left and right sides of (\ref{4.4}) respectively. Let $y^* \in \mathbf{A}.$ Then
 $y^*\in Y$ and, for a net $(D,\preceq,g)$ in $X,$ the relation (\ref{4.3}) is
 valid under $y = y^*.$
 Then, by (\ref{4.3})  \begin{equation}\label{4.5}\mathcal{F}^* \df \,(X-\mathrm{ass})
 [D;\preceq;g]\in \mathfrak{F}_o[X\,|\,\mathcal{X}].
\end{equation}
Moreover, by (\ref{3.7}) and (\ref{4.3}) $(Y-\mathrm{ass})[D;\preceq;f\circ
g]{\stackrel{\tau}{\Longrightarrow}}\,y^*.$ So, by (\ref{3.5})
\begin{equation}\label{4.6}N_\tau(y^*)\subset (Y-\mathrm{ass})[D;\preceq;f\circ
g].\end{equation} Let $H^*\in N_\tau(y^*).$ Then by (\ref{3.6}) and (\ref{4.6}), for some
$d^*\in D,$ the following property is valid: $\fo d\in D$ \begin{equation}\label{4.7}(d^*
\preceq d) \Longrightarrow \Bigl(f\bigl(g(d)\bigl)\in H^*\Bigl).\end{equation} In
addition, $D^* \,\df\,\{d\in D\,|\,d^* \preceq d\}\in \cp^\prime(D)$ and by (\ref{3.6})
and (\ref{4.5})
$$g^1(D^*)\in \mathcal{F}^*.$$ As a corollary, $f^1\bigl(g^1(D^*)\bigl) = (f\circ
g)^1(D^*) \in f^1[\mathcal{F}^*].$ But,by (\ref{4.7}) $f^1\bigl(g^1(D^*)\bigl)\subset
H^*.$ By (\ref{3.3}) $H^*\in (Y-\mathrm{\mathbf{fi}})\bigl[f^1[\mathcal{F}^*]\bigl].$
Since the choice of $H^*$ was arbitrary, the inclusion $N_\tau(y^*) \subset
(Y-\mathrm{\mathbf{fi}})\bigl[f^1[\mathcal{F}^*]\bigl]$ is established. By (\ref{3.4})
\begin{equation}\label{4.8}f^1[\mathcal{F}^*]{\stackrel{\tau}{\Longrightarrow}}\,y^*.
\end{equation}
By (\ref{4.5}) and (\ref{4.8}) $y^* \in \mathbf{B}.$ The inclusion
$\mathbf{A}\subset \mathbf{B}$ is established.

Let $y^o\in \mathbf{B}.$ Then, for $y^o\in Y,$ we have a filter $\mathcal{F}^o\in
\mathfrak{F}_o[X\,|\,\mathcal{X}]$ such that
\begin{equation}\label{4.9}f^1[\mathcal{F}^o]{\stackrel{\tau}{\Longrightarrow}}\,y^o.
\end{equation}
Choose a net $(\mathbb{D},\sqsubseteq, \varphi)$ in $X$ for which
$\mathcal{F}^o = (X-\mathrm{ass})[\mathbb{D};\sqsubseteq;\varphi].$
By (\ref{4.1}) $f^1[\mathcal{F}^o]\in \beta_o[Y]$ and, as a
corollary, by (\ref{3.4}) and (\ref{4.9})
\bfn\label{4.10}N_\tau(y^o) \subset
(Y-\mathrm{\mathbf{fi}})\bigl[f^1[\mathcal{F}^o]\bigl].\efn Then by
(\ref{3.3}) and (\ref{4.10}) we obtain that $\fo H\in N_\tau(y^o)\
\exists B\in f^1[\mathcal{F}^o]:\ B\subset H.$ Using (\ref{2.1}) we
have the property: $\fo H\in N_\tau(y^o)\ \exists F\in
\mathcal{F}^o:\,f^1(F)\subset H.$ Choose arbitrary $H^o\in
N_\tau(y^o);$ then, for some $F^o\in \mathcal{F}^o$ the inclusion
$f^1(F^o)\subset H^o$ is valid. By (\ref{3.6}) and the choice of
$(\mathbb{D},\sqsubseteq,\varphi),$ for some $d^o\in \mathbb{D},$
the following property is realized: $\fo \delta\in \mathbb{D}$
$$(d^o\sqsubseteq \delta) \Longrightarrow \bigl(\varphi(\delta)\in F^o\bigl).$$ By the
choice of $F^o$ we obtain that $\fo \delta\in \mathbb{D}$ $$(d^o\sqsubseteq \delta)
\Longrightarrow \bigl((f\circ \varphi)(\delta)\in H^o\bigl).$$ Then, $H^o\in
(Y-\mathrm{ass})[\mathbb{D};\sqsubseteq;f\circ \varphi].$ So, the important inclusion
$$N_\tau(y^o)\subset (Y-\mathrm{ass})[\mathbb{D};\sqsubseteq;f \circ \varphi]$$ is valid.
Then $(Y-\mathrm{ass})[\mathbb{D};\sqsubseteq;f\circ
\varphi]{\stackrel{\tau}{\Longrightarrow}}\,y^o$ (see (\ref{3.5})). By (\ref{3.7})
\bfn\label{4.11}(\mathbb{D},\sqsubseteq,f\circ
\varphi){\stackrel{\tau}{\longrightarrow}}\,y^o.\efn  Moreover, by the choice of
$\mathcal{F}^o$ and $(\mathbb{D},\sqsubseteq,\varphi)$ the inclusion
$$\mathcal{X}\subset (X-\mathrm{ass})[\mathbb{D};\sqsubseteq;\varphi]$$is valid. From
(\ref{4.11}), we have the inclusion $y^o\in \mathbf{A}.$ So,
$\mathbf{B}\subset \mathbf{A}$ and, as a corollary, $\mathbf{A} =
\mathbf{B}.$ \end{proof}

\begin{propos}\label{p4.2}{\TL}  For any $f\in Y^X,\,\tau\in (\mathrm{top})[Y],$ and
$\mathcal{X}\in \cp^\prime\bigl(\cp(X)\bigl)$
\bfn\label{4.12}(\mathrm{\mathbf{as}})[X;Y;\tau;f;\mathcal{X}] = \{y\in Y\,|\,\ \exists
\mathcal{U}\in \mathfrak{F}_\mathbf{u}^o[X\,|\,\mathcal{X}]:\,f^1[\mathcal{U}]
{\stackrel{\tau}{\Longrightarrow}}\,y\}.\efn
\end{propos}
\begin{proof} We denote respectively by $\mathbf{F}$ an $\mathbf{U}$ the sets
on the left and right sides of (\ref{4.12}). Since
$\mathfrak{F}_\mathbf{u}^o[X\,|\,\mathcal{X}]\subset
\mathfrak{F}_o[X\,|\,\mathcal{X}],$ we have the obvious inclusion
$\mathbf{U}\subset \mathbf{F}$ (see Proposition~\ref{p4.1}). Let
$y_o\in \mathbf{F}.$ Then by Proposition~\ref{p4.1}
$f^1[\mathcal{F}]{\stackrel{\tau}{\Longrightarrow}}\,y_o$ for some
$\mathcal{F}\in \mathfrak{F}_o[X\,|\,\mathcal{X}].$ Then
$\mathcal{F}\in \mathfrak{F}[X]$ and $\mathcal{X}\subset
\mathcal{F}.$ We recall (see Section 3) that
$\mathfrak{F}_\mathbf{u}^o[X\,|\,\mathcal{F}]\in
\cp^\prime(\mathfrak{F}_\mathbf{u}[X]).$ Choose arbitrary
$\mathfrak{U}\in \mathfrak{F}_\mathbf{u}^o[X\,|\,\mathcal{F}].$ Then
$\mathfrak{U}\in \mathfrak{F}_\mathbf{u}[X]$ and $\mathcal{X}\subset
\mathcal{F}\subset \mathfrak{U}.$ Therefore, $\mathfrak{U}\in
\mathfrak{F}_\mathbf{u}^o[X\,|\,\mathcal{X}].$ Moreover, by
(\ref{2.1}) $f^1[\cf]\subset f^1[\mathfrak{U}]$ and, as a corollary,
\bfn\label{4.13}(Y-\mathrm{\mathbf{fi}})\bigl[f^1[\cf]\bigl]\subset
(Y-\mathrm{\mathbf{fi}})\bigl[f^1[\mathfrak{U}]\bigl]\efn (we recall
that by (\ref{4.1}) $f^1[\cf]\in \beta_o[Y]$ and
$f^1[\mathfrak{U}]\in \beta_o[Y]).$ By the choice of $\cf$ we have
the inclusion $$N_\tau(y_o) \subset
(Y-\mathrm{\mathbf{fi}})\bigl[f^1[\cf]\bigl]$$ (see (\ref{3.4})).
Then by (\ref{4.13}) $N_\tau(y_o) \subset
(Y-\mathrm{\mathbf{fi}})\bigl[f^1[\mathfrak{U}]\bigl]$ and, as a
corollary (see (\ref{3.4})),
$$f^1[\mathfrak{U}]{\stackrel{\tau}{\Longrightarrow}}\,y_o.$$  Then, $y_o\in \mathbf{U}.$
The inclusion $\mathbf{F}\subset \mathbf{U}$ is established.
\end{proof}

Recall that, for any family $\mathcal{X}\in \cp^\prime\bigl(\cp(X)\bigl),\ \
\{\cap\}_\mathbf{f}(\mathcal{X})\in \cp^\prime\bigl(\cp(X)\bigl)$ and $\mathcal{X}\subset
\{\cap\}_\mathbf{f}(\mathcal{X}).$ We note the following obvious

\begin{propos}\label{p4.3} For any $\mathcal{X}\in
\cp^\prime\bigl(\cp(X)\bigl),$ the equality $\mathfrak{F}_\mathbf{u}^o[X\,|\,\mathcal{X}]
= \mathfrak{F}_\mathbf{u}^o[X\,|\{\cap\}_\mathbf{f}(\mathcal{X})]$ is valid.
\end{propos}
\begin{proof}
Recall that $\mathcal{X}\subset
\{\cap\}_\mathbf{f}(\mathcal{X}).$ Therefore,
$\mathfrak{F}_\mathbf{u}^o[X\,|\,\{\cap\}_\mathbf{f}(\mathcal{X})]\subset
\mathfrak{F}_\mathbf{u}^o[X\,|\,\mathcal{X}].$ On the other hand,
from (\ref{3.1}), we obtain that $\cf = \{\cap\}_\mathbf{f}(\cf) \ \
\fo \cf\in \mathfrak{F}[X].$ Then, for an ultrafilter
$\mathcal{U}\in \mathfrak{F}_\mathbf{u}^o[X\,|\,\mathcal{X}],$
$$\{\cap\}_\mathbf{f}(\mathcal{X})\subset
\{\cap\}_\mathbf{f}(\mathcal{U}) = \mathcal{U}$$ and, as a
corollary, $\mathcal{U}\in
\mathfrak{F}_\mathbf{u}^o[X\,|\,\{\cap\}_\mathbf{f}(\mathcal{X})].$
So, since the choice of $\mathcal{U}$ was arbitrary,
$\mathfrak{F}_\mathbf{u}^o[X\,|\,\mathcal{X}]\subset
\mathfrak{F}_\mathbf{u}^o[X\,|\,\{\cap\}_\mathbf{f}(\mathcal{X})]$
and, as a corollary, $\mathfrak{F}_\mathbf{u}^o[X\,|\,\mathcal{X}] =
\mathfrak{F}_\mathbf{u}^o[X\,|\,\{\cap\}_\mathbf{f}(\mathcal{X})].$
\end{proof}

\begin{cor}\label{c4.1}{\TL}
 If $f\in Y^X,\,\tau\in (\mathrm{top})[Y],$ and
$\mathcal{X}\in \cp^\prime\bigl(\cp(X)\bigl),$ then
$$(\mathrm{\mathbf{as}})[X;Y;\tau;f;\mathcal{X}] =
(\mathrm{\mathbf{as}})[X;Y;\tau;f;\{\cap\}_\mathbf{f}(\mathcal{X})].$$
\end{cor}

The corresponding proof is realized by the immediate combination of
Propositions~\ref{p4.2} and \ref{p4.3}. We note that, by definitions
of Section 2 \bfn\label{4.14}\{\cap\}_\mathbf{f}(\mathcal{X})\in
\beta[X]\ \ \fo\mathcal{X}\in \cp^\prime\bigl(\cp(X)\bigl).\efn In
connection with (\ref{4.14}), we note the following general
property. Namely, $\fo f\in Y^X\ \ \fo \tau\in (\mathrm{top})[Y]\ \
\fo \mathcal{B}\in \beta[X]$
\bfn\label{4.15}(\mathrm{\mathbf{as}})[X;Y;\tau;f;\mathcal{B}] =
\bigcap\limits_{B\in
\mathcal{B}}\mathrm{cl}\bigl(f^1(B),\tau\bigl).\efn  Then, by
(\ref{4.14}), (\ref{4.15}), and Corollary~\ref{c4.1}
\begin{equation}\label{4.16}\begin{array}{c}
(\mathrm{\mathbf{as}})[X;Y;\tau;f;\mathcal{X}] =
\bigcap\limits_{B\in\{\cap\}_\mathbf{f}(\mathcal{X})}\mathrm{cl}\bigl(f^1(B),\tau\bigl)\\
\ \fo f\in Y^X\ \ \fo \tau\in (\mathrm{top})[Y]\ \ \fo \mathcal{X}\in
\cp^\prime\bigl(\cp(X)\bigl).\end{array}
\end{equation}
In connection with (\ref{4.16}), we note that $\fo \mathcal{X}\in
\cp^\prime\bigl(\cp(X)\bigl)$ \bfn\label{4.17}\bigl(\emptyset\in
\{\cap\}_\mathbf{f}(\mathcal{X})\bigl) \Longrightarrow
(\mathfrak{F}_\mathbf{u}^o[X\,|\,\mathcal{X}] = \emptyset).\efn

{\bf Remark 4.1}. By analogy with Proposition~\ref{p4.3} we have that
$$\mathfrak{F}_o[X\,|\,\mathcal{X}] =
\mathcal{F}_o[X\,|\,\{\cap\}_\mathbf{f}(\mathcal{X})]\ \ \fo
\mathcal{X}\in \cp^\prime\bigl(\cp(X)\bigl).$$ Really, fix
$\mathcal{X}\in \cp^\prime\bigl(\cp(X)\bigl).$ Then
$\mathcal{X}\subset \{\cap\}_\mathbf{f}(\mathcal{X}).$ Therefore,
$\mathfrak{F}_o[X\,|\,\{\cap\}_\mathbf{f}(\mathcal{X})]\subset
\mathfrak{F}_o[X\,|\,\mathcal{X}].$ Let $\cf\in
\mathfrak{F}_o[X\,|\,\mathcal{X}].$ Then, $\cf\in \mathfrak{F}[X]$
and $\mathcal{X}\subset \cf.$ But, from (\ref{3.1}), we have the
equality $\cf = \{\cap\}_\mathbf{f}(\cf),$ where by the choice of
$\cf\ \ \{\cap\}_\mathbf{f}(\mathcal{X}) \subset
\{\cap\}_\mathbf{f}(\cf).$ So,
$\{\cap\}_\mathbf{f}(\mathcal{X})\subset \cf$ and, as a corollary,
$\cf\in \mathfrak{F}_o[X\,|\,\{\cap\}_\mathbf{f}(\mathcal{X})].$ The
inclusion $\mathfrak{F}_o[X\,|\,\mathcal{X}]\subset
\mathfrak{F}_o[X\,|\,\{\cap\}_\mathbf{f}(\mathcal{X})]$ is
established. So, $\mathfrak{F}_o[X\,|\,\mathcal{X}] =
\mathfrak{F}_o[X\,|\,\{\cap\}_\mathbf{f}(\mathcal{X})]. $

Returning to (\ref{4.17}), we note that by Proposition~\ref{p4.2} $\fo f\in Y^X\ \ \fo
\mathcal{X}\in \cp^\prime\bigl(\cp(X)\bigl)$ \bfn\label{4.18}\bigl(\emptyset\in
\{\cap\}_\mathbf{f}(\mathcal{X})\bigl) \Longrightarrow
\bigl((\mathrm{\mathbf{as}})[X;Y;\tau;f;\mathcal{X}] = \emptyset\ \ \fo \tau\in
(\mathrm{top})[Y]\bigl).\efn

{\bf Remark 4.2}. We have that, for the case $\emptyset\notin
\{\cap\}_\mathbf{f}(\mathcal{X}),$ it is possible that $$\exists \tau\in
(\mathrm{top})[Y]:\,(\mathrm{\mathbf{as}})[X;Y;\tau;f;\mathcal{X}] = \emptyset.$$ Indeed,
consider the case $X= Y = \mathbb{R},$ $f(x) = x\ \ \fo x\in X,\, \tau = \tau_\mathbb{R}$
is the usual $|\cdot|$-topology of real line $\mathbb{R},$ and $$\mathcal{X}=
\{[\,c,\infty[\,:\,\ c\in \mathbb{R}\} .$$ Then, $\mathcal{X}\in \beta_o[X]$ and
$\emptyset \notin \{\cap\}_\mathbf{f}(\mathcal{X}).$ Therefore, by (\ref{4.15})
$$\hspace{3.0cm}(\mathrm{\mathbf{as}})[X;Y;\tau;f;\mathcal{X}] = \bigcap\limits_{c\in
\mathbb{R}}[\,c,\infty[\,= \emptyset.\hspace{4cm} \Box$$\smallskip

It is obvious the following

\begin{propos}\label{p4.4}{\TL}  If $f\in Y^X,\, \tau\in (c-\mathrm{top})[Y],$ and $\mathcal{B}\in
\beta_o[X],$ then $$(\mathrm{\mathbf{as}})[X;Y;\tau;f;\mathcal{B}]
\neq \emptyset.$$\end{propos}
\begin{proof}
The corresponding proof follows from known statements of general
topology (see \cite[ch.\,I]{34}). But, we consider this proof for a
completeness of the account. In our case, we have (\ref{4.15}). In
addition,
\bfn\label{4.19}\mathcal{T}\,\df\,\bigl\{\mathrm{cl}\bigl(f^1(B),\tau\bigl):\,B\in
\mathcal{B}\bigl\}\efn  is nonempty family of sets closed in the
compact topological space (TS) $(Y,\tau).$ Moreover, $\mathcal{T}\in
\beta_o[Y]$ (we use known properties of the closure operation and
the image operation). Since $\emptyset \notin \mathcal{B},$ we
obtain that $\emptyset\notin \mathcal{T}.$ In addition,
$\mathcal{T}\in \beta[Y].$ Therefore, by \cite[(3.3.16)]{32} we have
the following property: if $n\in \mathbb{N}$ and
$$(T_i)_{i\in\ov{1,n}}:\,\ov{1,n}\longrightarrow \mathcal{T},$$ then $\exists T\in
\mathcal{T}:\,T \subset \bigcap\limits_{i=1}^nT_i.$ As a corollary,
$\mathcal{T}$ is the nonempty centered system of closed sets in a
compact TS. Then, the intersection of all sets of $\mathcal{T}$ is
not empty. By (\ref{4.19}) $$\bigcap\limits_{B\in
\mathcal{B}}\mathrm{cl}\bigl(f^1(B),\tau\bigl) =
\bigcap\limits_{T\in \mathcal{T}} T \neq \emptyset.$$ Using
(\ref{4.15}), we obtain the required statement about the
nonemptyness of attraction set. \end{proof}

\begin{cor}\label{c4.2}{\TL} If $f\in Y^X$ and $\mathcal{X}\in \cp^\prime\bigl(\cp(X)\bigl),$
then
$$\bigl(\emptyset \notin \{\cap\}_\mathbf{f}(\mathcal{X})\bigl) \Longrightarrow
\bigl((\mathrm{\mathbf{as}})[X;Y;\tau;f;\mathcal{X}] \neq \emptyset\ \ \fo \tau\in
(c-\mathrm{top})[Y]\bigl).$$
\end{cor}

\begin{proof} Let $\emptyset \notin
\{\cap\}_\mathbf{f}(\mathcal{X}).$ Choose arbitrary topology
$\tau\in (c-\mathrm{top})[Y].$ By (\ref{4.14}) $
\{\cap\}_\mathbf{f}(\mathcal{X})\in \beta[X].$ Moreover,
$\mathcal{X}\subset \{\cap\}_\mathbf{f}(\mathcal{X}).$ Therefore,
$\emptyset\notin \mathcal{X}.$ Then, $
\{\cap\}_\mathbf{f}(\mathcal{X})\in \beta_o[X]$ and by
Proposition~\ref{p4.4}
$$(\mathrm{\mathbf{as}})[X;Y;\tau;f;\{\cap\}_\mathbf{f}(\mathcal{X})] \neq \emptyset.$$
Using Corollary~\ref{c4.1}, we obtain that
$(\mathrm{\mathbf{as}})[X;Y;\tau;f;\mathcal{X}] \neq \emptyset. $
\end{proof}

In the following, we use the continuity  notion. In this connection, suppose that
\begin{equation}\label{4.20}\begin{array}{c} C(X,\tau_1,Y,\tau_2) \,\df\,\bigl\{f\in
Y^X\,|\,f^{-1}(G) \in \tau_1\ \ \fo G\in \tau_2\bigl\}\\  \fo \tau_1\in
(\mathrm{top})[X]\ \ \fo \tau_2\in (\mathrm{top})[Y].\end{array}\end{equation} So,
continuous functions are defined. In the following, we use bijections, open and closed
mappings, and homeomorphisms. Let
\begin{equation}\label{4.21}\begin{array}{c}(\mathrm{bi})[X;Y]\,\df\,\biggl\{f\in
Y^X\,|\,\bigl(f^1(X) = Y\bigl)\,\&\,\biggl(\fo x_1\in X\ \ \fo x_2\in X\\
\Bigl(\bigl(f(x_1) = f(x_2)\bigl) \Longrightarrow (x_1 =
x_2)\Bigl)\biggl)\biggl\}.\end{array}\end{equation} In (\ref{4.21}),
the set of all bijections from $X$ onto $Y$ is defined. If
$\tau_1\in (\mathrm{top})[X]$ and $\tau_2\in(\mathrm{top})[Y],$ then
\begin{equation}\label{4.22}\begin{array}{c}C_{\mathrm{op}}(X,\tau_1,Y,\tau_2) = \bigl\{f\in
C(X,\tau_1,Y,\tau_2)\,|\,f^1(G)\in \tau_2\ \ \fo G\in
\tau_1\bigl\},\end{array}\end{equation}
\begin{equation}\label{4.23}\begin{array}{c}C_{\mathrm{cl}}(X,\tau_1,Y,\tau_2) = \bigl\{f\in
C(X,\tau_1,Y,\tau_2)\,|\,f^1(F)\in \mathbf{C}_Y[\tau_2]\\ \ \fo F\in
\mathbf{C}_X[\tau_1]\bigl\}.\end{array}\end{equation} In (\ref{4.22}) (in (\ref{4.23})),
we consider open (closed) mappings. In addition,
\begin{equation}\label{4.24}\begin{array}{c}(\mathrm{Hom})[X;\tau_1;Y;\tau_2]\,\df\,
C_\mathrm{op}(X,\tau_1,Y,\tau_2)\,\cap\,(\mathrm{bi})[X;Y] =\\
= C_\mathrm{cl}(X,\tau_1,Y,\tau_2)\,\cap\,(\mathrm{bi})[X;Y]\ \ \fo \tau_1\in
(\mathrm{top})[X]\ \ \fo \tau_2\in (\mathrm{top})[Y].\end{array}\end{equation} So, in
(\ref{4.24}), the set of homeomorphisms is defined.

\section{Some properties of ultrafilters of measurable spaces}
\setcounter{equation}{0}

In this section, we fix a nonempty set $\mathbf{E}.$ We consider the very general
measurable space $(\mathbf{E},\cl),$ where $\cl\in \pi[\mathbf{E}]$ is fixed also.
According to necessity, we will be supplement the corresponding suppositions with respect
to $\cl.$ We suppose that $\mathbb{F}^*(\cl)$ is the set of all families $\mathcal{F}\in
\cp^\prime(\cl)$ such that $$(\emptyset \notin \cf)\,\&\,\bigl(A \cap B\in \cf\ \ \fo
A\in \cf\ \ \fo B\in \cf)\,\&$$ $$\&\,(\fo F\in \cf\ \ \fo L\in \cl\ \ (F\subset L)
\Longrightarrow (L\in \cf)\bigl).$$ Elements of the set $\mathbb{F}^*(\cl)$ are filters
of $\cl.$ In addition,
\begin{equation}\label{5.0}\begin{array}{c}\mathbb{F}_o^*(\cl)\,\df\,\bigl\{\mathcal{U}\in
\mathbb{F}^*(\cl)\,|\,\fo \cf\in \mathbb{F}^*(\cl)\ \ \bigl((\mathcal{U}\subset
\cf)\Longrightarrow (\mathcal{U} = \cf)\bigl)\bigl\}\end{array}\end{equation} is the set
of all ultrafilters of $\cl.$ Recall that (see \cite[p.\,29]{39})
\begin{equation}\label{5.1}\fo \cf\in \mathbb{F}^*(\cl)\ \ \exists
\mathcal{U}\in \mathbb{F}_o^*(\cl):\,\cf\subset \mathcal{U}\end{equation} In the
following, (\ref{5.1}) plays the very important role.

We introduce the mapping $\Phi_\cl:\,\cl \rightarrow
\cp\bigl(\mathbb{F}_o^*(\cl)\bigl)$ by the following rule:
\bfn\label{5.2}\Phi_\cl(L)\,\df\,\{\mathcal{U}\in
\mathbb{F}_o^*(\cl)\,|\,L\in \mathcal{U}\}\ \ \fo L\in \cl.\efn We
note that $\{\mathbf{E}\}\in \mathbb{F}^*(\cl)$ and by (\ref{5.1})
$\mathbb{F}_o^*(\cl)\neq \emptyset.$ In addition, we recall that
(see Section 2)
\bfn\label{5.3}(\mathbb{UF})[\mathbf{E};\cl]\,\df\,\{\Phi_\cl(L):\,L\in\cl\}\in
\pi[\mathbb{F}_o^*(\cl)];\efn by (\ref{5.3}) the pair
$(\mathbb{F}_o^*(\cl),(\mathbb{UF})[\mathbf{E};\cl])$ is a nonempty
multiplicative space. We note some simplest general properties. We
obtain that
$$(\mathcal{E}-\mathrm{set})[\mathbf{E}]\,\df\,\{A\in \cp(\mathbf{E})\,|\,A \cap S\neq \emptyset
\, \ \fo S\in \mathcal{E}\}\in
\cp\bigl(\cp^\prime(\mathbf{E})\bigl)\ \ \fo \mathcal{E}\in
\cp^\prime\bigl(\cp(E)\bigl).$$ We note that $\mathcal{B}|_A\in
\beta_o[A]\ \ \fo \mathcal{B}\in \beta_o[\mathbf{E}]\ \ \fo A\in
(\mathcal{B}-\mathrm{set})[\mathbf{E}].$ In addition, for $A\in
\cp^\prime(\mathbf{E})$ the inclusion $\beta_o[A] \subset
\beta_o[\mathbf{E}]$ takes place. Therefore,
\bfn\label{5.4}\mathcal{B}|_A\in \beta_o[\mathbf{E}]\ \ \fo
\mathcal{B}\in \beta_o[\mathbf{E}]\ \ \fo A\in
(\mathcal{B}-\mathrm{set})[\mathbf{E}].\efn With the employment of
(\ref{5.4}), we obtain that, for any $\mathcal{B}\in
\beta_o[\mathbf{E}]$ and $A\in
(\mathcal{B}-\mathrm{set})[\mathbf{E}]$
\begin{equation}\label{5.5}\begin{array}{c}(\mathbf{E}-\mathrm{\mathbf{fi}})[\mathcal{B}|_A]\in
\mathfrak{F}[\mathbf{E}]:\,\bigl((\mathbf{E}-\mathrm{\mathbf{fi}})[\mathcal{B}]\subset
(\mathbf{E}-\mathrm{\mathbf{fi}})[\mathcal{B}|_A]\bigl)\,\&\\
\&\,\bigl(A\in
(\mathbf{E}-\mathrm{\mathbf{fi}})[\mathcal{B}|_A]\bigl).\end{array}\end{equation}
Now, we return to the space $(\mathbf{E},\cl).$ Suppose that
\begin{equation}\label{5.6}\begin{array}{c}\beta_\cl^o[\mathbf{E}]\,\df\,\bigl\{\mathcal{B}\in
\cp^\prime(\cl)\,|\,(\emptyset\notin \mathcal{B})\,\&\\ \&\,(\fo
B_1\in \mathcal{B}\ \ \fo B_2\in \mathcal{B}\ \ \exists B_3\in
\mathcal{B}:\,B_3\subset B_1 \cap B_2)\bigl\}
\end{array}\end{equation} (the set of filter bases of $\cl);$  $\mathbb{F}^*(\cl) \subset
\beta_\cl^o[\mathbf{E}]$ and $$\beta_\cl^o[\mathbf{E}] =
\beta_o[\mathbf{E}] \cap \cp(\cl) = \{\mathcal{B}\in
\beta_o[\mathbf{E}]\,|\,\mathcal{B}\subset \cl\}.$$ We note the
obvious property: $\cf \cap \cl \in \mathbb{F}^*(\cl)\ \ \fo \cf \in
\mathfrak{F}[\mathbf{E}].$ In addition,
$$(\mathbf{E}-\mathrm{\mathbf{fi}})[\mathcal{B}
|\,\cl]\,\df\,(\mathbf{E}-\mathrm{\mathbf{fi}})[\mathcal{B}] \cap
\cl = $$ $$= \{L\in \cl |\,\exists B\in \mathcal{B}:\,B\subset
L\}\in \mathbb{F}^*(\cl)\ \ \fo \mathcal{B}\in
\beta_\cl^o[\mathbf{E}].$$ Using (\ref{5.4}) and the obvious
inclusion $\beta_\cl^o[\mathbf{E}]\subset \beta_o[\mathbf{E}],$
under $\mathcal{B}\in \beta_\cl^o[\mathbf{E}]$ and $A\in
(\mathcal{B}-\mathrm{set})[\mathbf{E}] \cap \cl,$ we obtain, that
$\mathcal{B}|_A \in \beta_\cl^o[\mathbf{E}].$ We note that, under
$\mathcal{B}\in \beta_\cl^o[\mathbf{E}]$ and $A\in
(\mathcal{B}-\mathrm{set})[\mathbf{E}] \cap \cl,$ the filter
$$(\mathbf{E}-\mathrm{\mathbf{fi}})[\mathcal{B}|_A\,|\,\cl]\in \mathbb{F}^*(\cl)$$
has the following properties
\begin{equation}\label{5.6`}\begin{array}{c}\bigl((\mathbf{E}-\mathrm{\mathbf{fi}})
[\mathcal{B}|\,\cl]\subset
(\mathbf{E}-\mathrm{\mathbf{fi}})[\mathcal{B}|_A\,|\,\cl]\bigl)\,\&\,\bigl(A\in
(\mathbf{E}-\mathrm{\mathbf{fi}})[\mathcal{B}|_A\,|\,\cl]\bigl).\end{array}\end{equation}
Of course, $(\mathbf{E}-\mathrm{\mathbf{fi}})[\cf |\,\cl] = \cf\ \
\fo \cf\in \mathbb{F}^*(\cl).$ We can use this property in
(\ref{5.6`}): for any $\cf\in \mathbb{F}^*(\cl)$ and $A\in
(\cf-\mathrm{set})[\mathbf{E}] \cap \cl,$ the filter
$(\mathbf{E}-\mathrm{\mathbf{fi}})[\cf|_A\,|\,\cl]\in
\mathbb{F}^*(\cl)$ has the properties
\bfn\label{5.7}\bigl(\cf\subset
(\mathbf{E}-\mathrm{\mathbf{fi}})[\cf|_A\,|\,\cl]\bigl)\,\&\,\bigl(A\in
(\mathbf{E}-\mathrm{\mathbf{fi}})[\cf|_A\,|\,\cl]\bigl).\efn In
connection with (\ref{5.7}), we recall the very general property: if
$\mathcal{B}\in \beta_\cl^o[\mathbf{E}]$ and $A\in
(\mathcal{B}-\mathrm{set})[\mathbf{E}] \cap \cl,$ then
$\mathcal{B}|_A\in \beta_\cl^o[\mathbf{E}].$ Using the maximality
property, we obtain that $$(\mathcal{U}-\mathrm{set})[\mathbf{E}]
\cap \cl = \mathcal{U}\ \ \fo \mathcal{U}\in \mathbb{F}_o^*(\cl).$$
And what is more, $\mathbb{F}_o^*(\cl) = \{\cf\in
\mathbb{F}^*(\cl)\,|\,\cf = (\cf-\mathrm{set})[\mathbf{E}] \cap
\cl\}.$

Of course, the above-mentioned properties are valid for
\bfn\label{5.8}\cl\in (\mathrm{LAT})_o[\mathbf{E}].\efn The
following reasoning is similar to the construction of
\cite[$\S\,$3.6]{36} connected with Wallman extension; in addition,
later until the end of this section, we suppose that (\ref{5.8}) is
valid (so, we fix a lattice with ``zero'' and ``unit'').

So, if $\cu\in \mathbb{F}_o^*(\cl),\,A\in \cl,$ and $B\in \cl,$ then
(under condition (\ref{5.8})) \bfn\label{5.9}(A \cup B\in
\mathcal{U}) \Longrightarrow \bigl((A \in \mathcal{U}) \vee (B\in
\mathcal{U})\bigl)\efn The property (\ref{5.9}) is basic. As a
corollary, $\fo \mathcal{U}\in \mathbb{F}_o^*(\cl)\ \ \fo A\in \cl\
 \fo B\in \cl$ \bfn\label{5.10}(A \cup B = \mathbf{E})
\Longrightarrow \bigl((A\in \mathcal{U}) \vee (B\in
\mathcal{U})\bigl)\efn We note that by (\ref{5.9}) the following
property is valid: $$\Phi_\cl(A \cup B) = \Phi_\cl(A) \cup
\Phi_\cl(B)\ \ \fo A\in \cl\ \ \fo B\in \cl.$$ As a corollary, we
obtain the property \bfn\label{5.11}(\mathbb{UF})[\mathbf{E};\cl]\in
(\mathrm{LAT})_o[\mathbb{F}_o^*(\cl)]\efn (so, under (\ref{5.8}),
the statement (\ref{5.3}) is amplified). In (\ref{5.11}), we have
the lattice of subsets of $\mathbb{F}_o^*(\cl).$ This important fact
used below.

\section{Topological properties, 1}
\setcounter{equation}{0}

As in the previous section, now we fix a nonempty set $\mathbf{E}$
and a family $\cl\in \pi[\mathbf{E}].$ We note the following obvious
property: $$\{L\in \cl\,|\,\exists B\in \mathcal{B}:\,B\subset
L\}\in \mathbb{F}^*(\cl)\ \ \fo \mathcal{B}\in
\beta_o[\mathbf{E}].$$ From definitions of the previous section, the
following known property follows: $\fo \cu_1\in \mathbb{F}_o^*(\cl)\
\ \fo \cu_2\in \mathbb{F}_o^*(\cl)$ \bfn\label{7.0}(\cu_1\neq \cu_2)
\Longrightarrow (\exists A\in \cu_1\ \ \exists B\in \cu_2:\,A \cap B
= \emptyset).\efn Moreover, we note that $\fo \cu\in
\mathbb{F}_o^*(\cl)$ \bfn\label{7.1}\{\cu\} = \bigcap\limits_{L\in
\cu}\{\cf\in \mathbb{F}_o^*(\cl)\,|\,L\in \cf\} =
\bigcap\limits_{L\in \cu} \Phi_\cl(L).\efn Moreover, we note that
$\pi[\mathbb{F}_o^*(\cl)]\subset
(\mathrm{op}-\mathrm{BAS})[\mathbb{F}_o^*(\cl)].$ Therefore, by
(\ref{5.3}) \bfn\label{7.2}(\mathbb{UF})[\mathbf{E};\cl]\in
(\mathrm{op}-\mathrm{BAS})[\mathbb{F}_o^*(\cl)].\efn As a corollary,
we obtain (see Section 2) that
\begin{equation}\label{7.3}\begin{array}{c}\mathbf{T}_\cl^*[\mathbf{E}]\,\df\,\{\cup\}
\bigl((\mathbb{UF})[\mathbf{E};\cl]\bigl) =
\{\mathbb{G}:\,\mathbb{G}\in \cp\bigl(\mathbb{F}_o^*(\cl)\bigl)\,|\\
\fo \cu\in \mathbb{G}\ \ \exists L\in \cu:\,\Phi_\cl(L)\subset
\mathbb{G}\}\in
(\mathrm{top})[\mathbb{F}_o^*(\cl)].\end{array}\end{equation} We
recall the very known definition of Hausdorff topology; namely, we
introduce the set of such topologies: if $M$ is set, then
$$(\mathrm{top})_o[M]\,\df\,\{\tau\in (\mathrm{top})[M]\,|\,\fo
m_1\in M\ \ \fo m_2\in M\setminus \{m_1\}$$ $$\exists H_1\in N_\tau(m_1)\ \ \exists
H_2\in N_\tau(m_2):\, H_1 \cap H_2 = \emptyset\}.$$  For any set $M$ we suppose that
$$(c-\mathrm{top})_o[M]\,\df\,(c-\mathrm{top})[M] \cap
(\mathrm{top})_o[M].$$  If $\tau\in (c-\mathrm{top})_o[M],$ then TS
$(M,\tau)$ is called a compactum. Then, the obvious statement
follows from the ultrafilter properties (see (\ref{5.2}),
(\ref{7.0})):\bfn\label{7.4}\mathbf{T}_\cl^*[\mathbf{E}]\in
(\mathrm{top})_o[\mathbb{F}_o^*(\cl)].\efn So, by (\ref{7.4})
$(\mathbb{F}_o^*(\cl),\mathbf{T}_\cl^*[\mathbf{E}])$ is a Hausdorff
TS. Of course, we can use the previous statements of this section in
the case of $\cl\in (\mathrm{LAT})_o[\mathbf{E}],$ obtaining the
Hausdorff topology (\ref{7.4}).  But, in the above-mentioned case,
another construction of TS is very interesting. This construction is
similar to Wallman extension (see \cite[$\S\,$3.6]{36}). Moreover,
in this connection, we note the fundamental investigation \cite{37},
where topological representations in the class of ideals are
considered. We give the basic attention to the filter consideration
in connection with construction of Section 3 concerning with the
realization  of AS. In this connection, we note that
$\cp(\mathbf{E})\in \pi[\mathbf{E}]$ and the sets
$\mathbb{F}^*\bigl(\cp(\mathbf{E})\bigl)$ and
$\mathbb{F}_o^*\bigl(\cp(\mathbf{E})\bigl)$ are defined. From
(\ref{3.1}) and definitions of Section 5, we have the equality
$\mathbb{F}^*\bigl(\cp(\mathbf{E})\bigl) =
\mathfrak{F}[\mathbf{E}].$ Moreover, from (\ref{3.2}) and the
above-mentioned definitions of Section 5, the equality
\bfn\label{7.5}\mathbb{F}_o^*\bigl(\cp(\mathbf{E})\bigl) =
\mathfrak{F}_\mathbf{u}[\mathbf{E}]\efn follows. By these properties
(see (\ref{7.5})) the constructions of Section 3 obtain
interpretation in terms of filters and ultrafilters of measurable
spaces.

Now, we note one simple property; in addition, we use the inclusion
chain $\mathbb{F}_o^*(\cl)\subset \mathbb{F}^*(\cl) \subset
\beta_o[\mathbf{E}].$ So, by (\ref{3.3}) $$
(\mathbf{E}-\mathrm{\mathbf{fi}})[\cf] \in \mathfrak{F}[\mathbf{E}]\
\ \fo \cf\in \mathbb{F}^*(\cl).$$  In particular, we have the
following property: \bfn\label{7.6}(\mathbf{E}-\mathrm{\mathbf{fi}}
)[\cu] \in \mathfrak{F}[\mathbf{E}]\ \ \fo \cu\in
\mathbb{F}_o^*(\cl).\efn  We note one general simple property;
namely, in general case of $\cl\in \pi[\mathbf{E}]$
\bfn\label{7.7}\fo \cu\in \mathbb{F}_o^*(\cl)\ \ \exists
\widetilde{\cu}\in \mathfrak{F}_\mathbf{u}[\mathbf{E}]:\, \cu =
\widetilde{\cu} \cap \cl.\efn

{\bf Remark 6.1}. We note that (\ref{7.7}) is a variant of
Proposition 2.4.1 of monograph \cite{39}. Consider the corresponding
proof. Fix $\cu\in \mathbb{F}_o^*(\cl).$ Then by (\ref{7.6})
\bfn\label{7.8}\mathcal{V}\,\df\,(\mathbf{E}-\mathrm{\mathbf{fi}})[\cu]
= \{H\in \cp(\mathbf{E})\,|\,\exists B\in \cu:\,B\subset H\}\in
\mathfrak{F}[\mathbf{E}].\efn From (\ref{7.8}), we obtain (see
Section 3) that
$\mathfrak{F}_\mathbf{u}^o[\mathbf{E}|\,\mathcal{V}]\in
\cp^\prime(\mathfrak{F}_\mathbf{u}[E]).$ Let $$\mathcal{W}\in
\mathfrak{F}_\mathbf{u}^o[\mathbf{E}|\,\mathcal{V}].$$ Then,
$\mathcal{W}\in \mathfrak{F}_\mathbf{u}[\mathbf{E}]$ and
$\mathcal{V}\subset \mathcal{W}.$ In addition (see Section 5),
$\mathcal{W} \cap \cl \in \mathbb{F}^*(\cl).$ Let $U\in \cu.$ Then,
$U\in \cl$ and, in particular, $U\in \cp(\mathbf{E}).$ By
(\ref{7.8}) $U\in \mathcal{V}$ and, as a corollary, $U\in
\mathcal{W}.$ Then, $U\in \mathcal{W} \cap \cl.$ So, the inclusion
$\cu \subset \mathcal{W} \cap \cl$ is established; we obtain that
\bfn\label{7.9}\mathcal{W} \cap \cl \in \mathbb{F}^*(\cl):\, \cu
\subset \mathcal{W} \cap \cl.\efn From (\ref{5.0}) and (\ref{7.9}),
we have the equality $\cu = \mathcal{W} \cap \cl.$ So,
$$\mathcal{W}\in \mathfrak{F}_\mathbf{u}[E]:\,\cu = \mathcal{W} \cap
\cl.$$ Since the choice of $\cu$ was arbitrary, the property
(\ref{7.7}) is established.

\section{Topological properties, 2}
\setcounter{equation}{0}

In this and following sections, we fix a nonempty set $E$ and a
lattice $\cl\in (\mathrm{LAT})_o[E].$ We consider the question about
constructing a compact $T_1$-space with ``unit''
$\mathbb{F}_o^*(\cl).$ This space is similar to Wallman extension
for a $T_1$-space. But, we not use axioms of topology and operate
lattice constructions (here, a natural analogy with constructions of
\cite[ch.\,II]{37} takes place). Later we use the following simple
statement.

\begin{propos}\label{p8.1}{\TL}  $(\mathbb{UF})[E;\cl]\in
(\mathrm{cl}-\mathrm{BAS})[\mathbb{F}_o^*(\cl)].$ \end{propos}

\begin{proof} We use (\ref{5.11}). In particular, $\emptyset\in
(\mathbb{UF})[E;\cl].$ As a corollary,
\bfn\label{8.1}\bigcap\limits_{\zeta\in (\mathbb{UF})[E;\cl]} \zeta
= \emptyset.\efn  Moreover, $\mathbb{F}_o^*(\cl) = \Phi_\cl(E) \in
(\mathbb{UF})[E;\cl]$ (see (\ref{5.3})). So, $(\mathbb{UF})[E;\cl]$
is a family with ``zero'' and ``unit''. Moreover, by (\ref{5.11})
$$B_1 \cup B_2 \in (\mathbb{UF})[E;\cl]\ \ \fo B_1\in
(\mathbb{UF})[E;\cl]\ \ \fo B_2 \in (\mathbb{UF})[E;\cl].$$
Therefore, by (\ref{2.12}) the required statement is realized.
\end{proof}

By (\ref{2.14}) and Proposition~\ref{p8.1} we have the following construction:
\bfn\label{8.2}\{\cap\}\bigl((\mathbb{UF})[E;\cl]\bigl)\in
(\mathrm{clos})[\mathbb{F}_o^*(\cl)].\efn

\begin{propos}\label{p8.2}{\TL}  The following compactness property is valid:
\bfn\label{8.3}\mathbf{C}_{\mathbb{F}_o^*(\cl)}\Bigl(\{\cap\}\bigl((\mathbb{UF})[E;\cl]\bigl)
\Bigl)\in (c-\mathrm{top})[\mathbb{F}_o^*(\cl)].\efn \end{propos}

\begin{proof} For brevity, we suppose that
\bfn\label{8.4}\mathfrak{U}\,\df\,\{\cap\}\bigl((\mathbb{UF})[E;\cl]\bigl)\efn and
$\theta\,\df\,\mathbf{C}_{\mathbb{F}_o^*(\cl)}[\mathfrak{U}].$ Of course, by (\ref{2.8})
$\theta\in (\mathrm{top})[\mathbb{F}_o^*(\cl)].$ Moreover, under $S\in \mathfrak{U},$ the
family
$$\mathbf{u}[S]\,\df\,\{T\in (\mathbb{UF})[E;\cl]\,|\,S\subset T\}\in
\cp^\prime\bigl((\mathbb{UF})[E;\cl]\bigl)$$ has the following
obvious property \bfn\label{8.5}S = \bigcap\limits_{T\in
\mathbf{u}[S]}T.\efn  We have the equality $\mathfrak{U} =
\mathbf{C}_{\mathbb{F}_o^*(\cl)}[\theta].$ So, $\mathfrak{U}$ is the
family of all subsets of $\mathbb{F}_o^*(\cl)$ closed in the TS
\bfn\label{8.5`}\bigl(\mathbb{F}_o^*(\cl),\theta\bigl).\efn Let
$\eta$ be arbitrary nonempty centered subfamily of $\mathfrak{U}$
(for any $m\in \mathbb{N}$ and $(T_i)_{i\in\ov{1,m}}\in \eta^m$ the
intersection of all sets $T_i,\,i\in \ov{1,m},$ is not empty). If
$H\in \eta,$ then the family
\bfn\label{8.6}\mathbb{D}_H\,\df\,\{L\in \cl\,|\,\Phi_\cl(L)\in
\mathbf{u}[H]\}\in \cp^\prime(\cl)\efn has the property:
$\mathbb{D}_H\subset \cu\ \ \fo \cu\in H.$ Of course,
$$\mathfrak{L}\,\df\, \bigcup\limits_{H\in \eta}\mathbb{D}_H\in
\cp^\prime(\cl)$$ is centered. Indeed, choose $n\in \mathbb{N}$ and
$(\Lambda_i)_{i\in \ov{1,n}}\in \mathfrak{L}^n.$ Let
$(\widetilde{H}_i)_{i\in\ov{1,n}}\in \eta^n$ be a procession with
the property:
$$\Lambda_j \in \mathbb{D}_{\widetilde{H}_j}\ \ \fo j\in \ov{1,n}.$$ Then, in particular,
$(\La_i)_{i\in\ov{1,n}}\in \cl^n.$ In addition, by (\ref{8.6})
$\Phi_\cl(\Lambda_j) \in \mathbf{u}[\widetilde{H}_j]\ \ \fo j\in
\ov{1,n}.$ Of course,
$$\bigcap\limits_{i=1}^n\widetilde{H}_i\subset\bigcap\limits_{i=1}^n
\Phi_\cl(\Lambda_i).$$ Since the intersection of all sets
$\widetilde{H}_i,\,i\in\ov{1,n},$ is not empty (we use the
centrality of $\eta),$ we choose an ultrafilter $$\widetilde{\cu}\in
\bigcap\limits_{i=1}^n \widetilde{H}_i.$$ Then, $\La_j\in
\widetilde{\cu}$ under $j\in \ov{1,n}.$ By axioms of a filter (see
Section 5) we obtain that $$\bigcap\limits_{i=1}^n \La_i \neq
\emptyset.$$ Since $\cl$ is closed with respect to finite
intersections, we obtain that
\bfn\label{8.7}\{\cap\}_\mathbf{f}(\mathfrak{L})\in
\cp^\prime(\cl):\,\bigl(\emptyset\notin
\{\cap\}_\mathbf{f}(\mathfrak{L})\bigl)\,\&\,\bigl(\mathfrak{L}\subset
\{\cap\}_\mathbf{f}(\mathfrak{L})\bigl).\efn Moreover, (\ref{8.7})
is supplemented by the following obvious property; namely, $$\fo
B_1\in \{\cap\}_\mathbf{f}(\mathfrak{L})\ \ \fo B_2\in
\{\cap\}_\mathbf{f}(\mathfrak{L})\ \ \exists B_3\in
\{\cap\}_\mathbf{f}(\mathfrak{L}):\,B_3\subset B_1 \cap B_2.$$  From
(\ref{5.6}), we obtain that $\{\cap\}_\mathbf{f}(\mathfrak{L})\in
\beta_\cl^o[E].$ As a corollary,
$$\mathcal{V}\,\df\,(E-\mathrm{\mathbf{fi}})[\{\cap\}_\mathbf{f}(\mathfrak{L})\,|\,\cl]
\in \mathbb{F}^*(\cl);$$ in addition, by (\ref{8.7})
$\mathfrak{L}\subset \{\cap\}_\mathbf{f}(\mathfrak{L})\subset
\mathcal{V}.$ Finally, we use (\ref{5.1}). Let $\mathcal{W}\in
\mathbb{F}_o^*(\cl)$ be an ultrafilter for which $\mathcal{V}\subset
\mathcal{W}.$ Then, $\mathfrak{L}\subset \mathcal{W}.$ So,
\bfn\label{8.8}\mathcal{W}\in
\mathbb{F}_o^*(\cl):\,\mathfrak{L}\subset \mathcal{W}.\efn  Let
$\mathbb{M}\in \eta.$ Then, $\mathbb{D}_\mathbb{M}\in
\cp^\prime(\mathfrak{L})$ and the equality
\bfn\label{8.9}\mathbb{M}= \bigcap\limits_{T\in
\mathbf{u}[\mathbb{M}]} T \efn is valid (see (\ref{8.5})). Choose
arbitrary $\Sigma \in \mathbf{u}[\mathbb{M}].$ Then, $\Sigma \in
(\mathbb{UF})[E;\cl]$ and $\mathbb{M}\subset \Sigma.$ Using
(\ref{5.3}), we choose $D\in \cl$ for which $\Sigma = \Phi_\cl(D).$
Then $$D\in \cl:\,\Phi_\cl(D) \in \mathbf{u}[\mathbb{M}].$$ By
(\ref{8.6}) $D\in \mathbb{D}_\mathbb{M}$ and, in particular, $D\in
\mathfrak{L}.$ By (\ref{8.8}) $D\in \mathcal{W}$ and, as a
corollary, $\mathcal{W}\in \Phi_\cl(D);$ see (\ref{5.2}). So,
$\mathcal{W}\in \Sigma.$ Since the choice of $\Sigma$ was arbitrary,
we obtain that $\mathcal{W}\in B\ \ \fo B\in
\mathbf{u}[\mathbb{M}].$ By (\ref{8.9}) $\mathcal{W}\in \mathbb{M}.$
So, we have the property:
$$\mathcal{W}\in H\ \ \fo H\in \eta.$$ Then, the intersection of all
sets of $\eta$ is not empty. Since the choice of $\eta$ was
arbitrary, it is established that any nonempty centered family of
closed (in TS (\ref{8.5`})) sets has the nonempty intersection. So,
TS (\ref{8.5`}) is compact (see \cite{34, 35, 36}). \end{proof}

Using Proposition~\ref{p8.2}, by $\mathbf{T}_\cl^o[E]$ we denote the
topology (\ref{8.3}); so,
\bfn\label{8.10}\mathbf{T}_\cl^o[E]\,\df\,\mathbf{C}_{\mathbb{F}_o^*(\cl)}\Bigl(\{\cap\}
\bigl((\mathbb{UF})[E;\cl]\bigl)\Bigl)\in
(c-\mathrm{top})[\mathbb{F}_o^*(\cl)].\efn We have the nonempty
compact TS
\bfn\label{8.11}\bigl(\mathbb{F}_o^*(\cl),\mathbf{T}_\cl^o[E]\bigl).\efn

\begin{propos}\label{p8.3}{\TL} If $\cu\in \mathbb{F}_o^*(\cl),$ then
$\{\cu\}\in \{\cap\}\bigl((\mathbb{UF})[E;\cl]\bigl).$
\end{propos}

The corresponding proof follows from (\ref{7.1}); of course, we use (\ref{5.3}) also.
From (\ref{2.17}), (\ref{8.10}), and Proposition~\ref{8.3}, we obtain the following
property: \bfn\label{8.12}\mathbf{T}_\cl^o[E]\in (c-\mathrm{top})[\mathbb{F}_o^*(\cl)]
\cap (\mathcal{D}-\mathrm{top})[\mathbb{F}_o^*(\cl)].\efn So, by (\ref{8.12}) we obtain
that (\ref{8.11}) is a nonempty compact $T_1$-space.

In conclusion of the given section, we note several properties. First, we recall that
\bfn\label{8.12`}\mathbb{F}_o^*(\cl) \setminus \Phi_\cl(L) = \{\cu\in
\mathbb{F}_o^*(\cl)\,|\,L\notin \cu\}\ \ \fo L\in \cl.\efn  In addition, from
(\ref{8.10}), the obvious representation follows:
\begin{equation}\label{8.13}\begin{array}{c}\mathbf{T}_\cl^o[E] =
\bigl\{G\in \cp\bigl(\mathbb{F}_o^*(\cl)\bigl)\,|\,\fo \cu\in G\ \
\exists L\in \cl\setminus \cu\ \ \fo \mathcal{V}\in
\mathbb{F}_o^*(\cl)\\  \bigl((L\notin \mathcal{V}) \Longrightarrow
(\mathcal{V}\in G)\bigl)\bigl\}.\end{array}\end{equation} With the
employment of (\ref{8.13}) the following statement is established.

\begin{propos}\label{p8.4}{\TL} If $\cu\in \mathbb{F}_o^*(\cl),$ then the
family $$\varphi_\cu \,\df\,\{\mathbb{F}_o^*(\cl)\setminus \Phi_\cl(L):\, L\in
\cl\setminus \cu\}$$ is a local base of TS $(\ref{8.11})$ at $\cu:$
$$\bigl(\varphi_\cu\subset N_{\mathbf{T}_\cl^o[E]}(\cu)\bigl)\,\&\,\bigl(\fo H\in
N_{\mathbf{T}_\cl^o[E]}(\cu)\ \ \exists B\in \varphi_\cu:\,B\subset H\bigl).$$
\end{propos}

The proof is obvious. So, by (\ref{2.18}) and Proposition~\ref{8.4}
$$\{\mathbb{F}_o^*(\cl)\setminus \Phi_\cl(L):\,L\in \cl\setminus \cu\}\in
(\cu-\mathrm{bas})\bigl[\mathbf{T}_\cl^o[E]\bigl]\ \ \fo \cu\in
\mathbb{F}_o^*(\cl).$$ We note that, from definitions, the following
property is valid: \bfn\label{8.14}\mathbb{F}_o^*(\cl)\setminus
\Phi_\cl(L) \in \mathbf{T}_\cl^o[E]\  \fo L\in \cl.\efn

\section{The density properties}
\setcounter{equation}{0}

In this section, we continue the investigation of TS (\ref{8.11}).
Of course, we preserve the suppositions of Section 7 with respect to
$E$ and $\cl.$ But, in this section, we postulate that $\{x\}\in
\cl\ \ \fo x\in E.$ So, in this section \bfn\label{9.1}\cl\in
(\mathrm{LAT})^o[E]\efn unless otherwise stipulated. So, $\cl\in
(\mathrm{LAT})_o[E]$ and $\{x\}\in \cl\ \ \fo x\in E.$ Therefore,
with regard (\ref{3.8}) and (\ref{9.1}), we obtain that
\bfn\label{9.2}(E-\mathrm{ult})[x] \cap \cl = \{L\in \cl |\,x\in L\}
\in \mathbb{F}_o^*(\cl)\ \ \fo x\in E.\efn Of course, for any $x\in
E,$ the inclusion $\{x\}\in (E-\mathrm{ult})[x] \cap \cl$ is valid.

\begin{propos}\label{p9.1}{\TL}  $\mathbb{F}_o^*(\cl) =
\mathrm{cl}\bigl(\bigl\{(E-\mathrm{ult})[x] \cap \cl:\,x\in
E\bigl\},\mathbf{T}_\cl^o[E]\bigl).$
\end{propos}

\begin{proof} Let $\cu\in \mathbb{F}_o^*(\cl)$ and $\mathbb{H}\in
N_{\mathbf{T}_\cl^o[E]}(\cu).$ We use Proposition~\ref{p8.4}.
Namely, we choose a set $\mathbb{L}\in \cl\setminus \cu$ for which
\bfn\label{9.3}\mathbb{F}_o^*(\cl)\setminus
\Phi_\cl(\mathbb{L})\subset \mathbb{H}.\efn Since $E\in \cu$ by
axioms of a filter (see Section 5), we obtain that $\mathbb{L}\neq
E.$ In addition, $\mathbb{L}\in \cl$ and by (\ref{2.4}) and
(\ref{9.1}) $\mathbb{L}\subset E.$ So, $E\setminus \mathbb{L}\neq
\emptyset.$  Choose arbitrary point $e\in E\setminus \mathbb{L}$ and
consider the ultrafilter
\bfn\label{9.4}\mathcal{E}\,\df\,(E-\mathrm{ult})[e] \cap \cl\in
\mathbb{F}_o^*(\cl);\efn see (\ref{9.2}). In addition, $\{e\} \in
\mathcal{E}.$ As a corollary, by definitions of Section 5
\bfn\label{9.5}(\mathbb{L}\in \mathcal{E}) \Longrightarrow
(\mathbb{L} \cap \{e\} \neq \emptyset).\efn But, $\mathbb{L} \cap
\{e\} = \e$ by the choice of $e.$ Therefore, by (\ref{9.5})
$\mathbb{L}\notin \mathcal{E}.$ From (\ref{5.2}) we have the
property $\mathcal{E}\notin \Phi_\cl(\mathbb{L}).$ As a corollary,
by (\ref{9.4}) \bfn\label{9.6}\mathcal{E}\in \mathbb{F}_o^*(\cl)
\setminus \Phi_\cl(\mathbb{L}).\efn From (\ref{9.3}) and
(\ref{9.6}), we obtain that $\mathcal{E}\in \mathbb{H}.$ By
(\ref{9.4}) \bfn\label{9.7}\{(E-\mathrm{ult})[x] \cap \cl:\,x\in E\}
\cap \mathbb{H} \neq \e.\efn Since the choice of $\mathbb{H}$ was
arbitrary, $$\cu\in \mathrm{cl}\bigl(\bigl\{(E-\mathrm{ult})[x] \cap
\cl:\,x\in E\bigl\},\mathbf{T}_\cl^o[E]\bigl).$$  Since the choice
of $\cu$ was arbitrary, the inclusion $$\mathbb{F}_o^*(\cl) \subset
\mathrm{cl}\bigl(\{E-\mathrm{ult})[x] \cap \cl:\,x\in
E\},\,\mathbf{T}_\cl^o[E]\bigl)$$ is established. The inverse
inclusion is obvious (see (\ref{8.10})).
\end{proof}

So, we obtain that trivial ultrafilters (\ref{9.2}) realize an
everywhere dense set in the TS (\ref{8.11}).

Returning to (\ref{8.10}), we note one obvious property connected
with (\ref{8.14}). Namely, by (\ref{2.13}) and
Proposition~\ref{p8.1}, in general case of $\cl\in
(\mathrm{LAT})_o[E]$
$$\mathbf{C}_{\mathbb{F}_o^*(\cl)}\bigl[(\mathbb{UF})[E;\cl]\bigl]\in
(\mathrm{op}-\mathrm{BAS})_\e[\mathbb{F}_o^*(\cl)]$$ and, in
particular,
$\mathbf{C}_{\mathbb{F}_o^*(\cl)}\bigl[(\mathbb{UF})[E;\cl]\bigl]\in
(\mathrm{op}-\mathrm{BAS})[\mathbb{F}_o^*(\cl)];$ then, for $\cl\in
(\mathrm{LAT})_o[E]$
$$\{\cup\}\bigl(\mathbf{C}_{\mathbb{F}_o^*(\cl)}\bigl[(\mathbb{UF})[E;\cl]\bigl]\bigl)\in
(\mathrm{top})[\mathbb{F}_o^*(\cl)].$$ And what is more by
(\ref{2.16}), (\ref{8.10}), and Proposition~\ref{p8.1}, in general
case of $\cl\in \mathrm{LAT})_o[E]$
\bfn\label{9.8}\mathbf{T}_\cl^o[E] =
\mathbf{C}_{\mathbb{F}_o^*(\cl)}\bigl[\{\cap\}\bigl((\mathbb{UF})[E;\cl]\bigl)\bigl]
=
\{\cup\}\bigl(\mathbf{C}_{\mathbb{F}_o^*(\cl)}\bigl[(\mathbb{UF})[E;\cl]\bigl]\bigl);\efn
so, by (\ref{9.8})
$\mathbf{C}_{\mathbb{F}_o^*(\cl)}\bigl[(\mathbb{UF})[E;\cl]\bigl]$
is a base of topology (\ref{8.10}). We recall that by (\ref{2.7})
and (\ref{5.3}), for general case of $\cl\in (\mathrm{LAT})_o[E]$
\begin{equation}\label{9.9}\begin{array}{c}\mathbf{C}_{\mathbb{F}_o^*(\cl)}\bigl[(\mathbb{UF})[E;\cl]\bigl] =
\{\mathbb{F}_o^*(\cl)\setminus B:\,B\in (\mathbb{UF})[E;\cl]\} =\\
= \{\mathbb{F}_o^*(\cl)\setminus \Phi_\cl(L):\,L\in \cl\}.\end{array}\end{equation}

{\bf Connection with Wallman extension}. Let $\tau\in
(\mathcal{D}-\mathrm{top})[E].$ Then, $\mathbf{C}_E[\tau] \in
(\mathrm{clos})[E]$ and by (\ref{2.17}) $\{x\}\in
\mathbf{C}_E[\tau]\ \ \fo x\in E.$ Using (\ref{2.6`}), we obtain
that $\mathbf{C}_E[\tau]\in (\mathrm{LAT})_o[E].$ With the
employment of the above-mentioned closedness of singletons, by the
corresponding definition of Section 2 we obtain that
\bfn\label{9.10}\mathbf{C}_E[\tau] \in(\mathrm{LAT})^o[E].\efn Until
the end of the present section, we suppose that \bfn\label{9.11}\cl
= \mathbf{C}_E[\tau].\efn So, in our case, $(E,\cl)$ is the lattice
of closed sets in $T_1$-space. Then, (\ref{8.11}) is the
corresponding Wallman compact space (see \cite[ch.\,3]{36}). On the
other hand, by (\ref{9.10}) and (\ref{9.11}) we obtain that this
variant of $(E,\cl)$ corresponds to general statements of our
section (for example, see (\ref{9.2}) and Proposition~\ref{p9.1}).
In this connection, we consider the mapping
\bfn\label{9.12}x\longmapsto (E-\mathrm{ult})[x] \cap
\cl:\,E\longrightarrow \mathbb{F}_o^*(\cl);\efn we denote the
mapping (\ref{9.12}) by $\mathbf{f}.$ So, $\mathbf{f}\in
\mathbb{F}_o^*(\cl)^E$ and $$\mathbf{f}(x)\,\df\,(E-\mathrm{ult})[x]
\cap \cl\ \ \fo x\in E.$$ Consider some simple properties. First, we
note that $\mathbf{f}$ is injective: $\fo x_1\in E\ \ \fo x_2 \in E$
\bfn\label{9.13}\bigl(\mathbf{f}(x_1) =
\mathbf{f}(x_2)\bigl)\Longrightarrow (x_1 = x_2).\efn Indeed, for
$x_1\in E$ and $x_2\in E$ with the property $\mathbf{f}(x_1) =
\mathbf{f}(x_2),$ by (\ref{3.8}) we have that $\{x_1\}\in
\mathbf{f}(x_2)$ and, as a corollary, $x_2\in \{x_1\};$ so, $x_1 =
x_2.$

Of course, $\mathbf{f}$ is a bijection from $E$ onto the set
\bfn\label{9.14}\mathbf{f}^1(E) = \{(E-\mathrm{ult})[x] \cap
\cl:\,x\in E\}\in \cp^\prime\bigl(\mathbb{F}_o^*(\cl)\bigl).\efn If
$L\in \cl$ and $x\in E,$ then $\bigl(L\in \mathbf{f}(x)\bigl)
\Leftrightarrow (x\in L).$ As a corollary, we obtain that
\bfn\label{9.15}\mathbf{f}^{-1}\bigl(\Phi_\cl(L)\bigl) = L\ \ \fo
L\in \cl.\efn

{\bf Remark\,8.1}. Of course, in (\ref{9.15}), we use the
representation (\ref{9.12}). Fix $L\in\cl.$ Let $x_*\in
\mathbf{f}^{-1}\bigl(\Phi_\cl(L)\bigl).$ Then $x_*\in E$ and
$\mathbf{f}(x_*)\in \Phi_\cl(L).$ By (\ref{5.2}) $L\in
\mathbf{f}(x_*)$ and,  as a corollary, $x_*\in L.$ So,
\bfn\label{9.16}\mathbf{f}^{-1}\bigl(\Phi_\cl(L)\bigl)\subset L.\efn
If $x^*\in L,$ then $L\in \mathbf{f}(x^*);$ see (\ref{9.12}).
Therefore, by(\ref{5.2}) $\mathbf{f}(x^*) \in \Phi_\cl(L)$ and, as a
corollary, $x^*\in \mathbf{f}^{-1}\bigl(\Phi_\cl(L)\bigl).$ So,
$L\subset \mathbf{f}^{-1}\bigl(\Phi_\cl(L)\bigl).$ Therefore
(see(\ref{9.16}))  $L$ and $\mathbf{f}^{-1}\bigl(\Phi_\cl(L)\bigl)$
coincide.

From (\ref{5.3}) and (\ref{9.15}), we obtain that \bfn\label{9.16`}\mathbf{f}^{-1}(B)\in
\cl\ \ \fo B\in(\mathbb{UF})[E;\cl].\efn

\begin{propos}\label{p9.2}{\TL} $\mathbf{f}\in
C(E,\tau,\mathbb{F}_o^*(\cl),\mathbf{T}_\cl^o[E]).$
\end{propos}

\begin{proof} We use the construction dual with respect to (\ref{4.20}). Let $F\in
\mathbf{C}_{\mathbb{F}_o^*(\cl)}\bigl[\mathbf{T}_\cl^o[E]\bigl].$
Then, by (\ref{9.8}) $F\in
\{\cap\}\bigl((\mathbb{UF})[E;\cl]\bigl).$ Therefore, for some
$\cf\in \cp^\prime\bigl((\mathbb{UF})[E;\cl]\bigl)$ $$F =
\bigcap\limits_{B\in \cf}B.$$ As a result, we obtain that
\bfn\label{9.17}\mathbf{f}^{-1}(F) = \bigcap\limits_{B\in
\cf}\mathbf{f}^{-1}(B),\efn where $\mathbf{f}^{-1}(\widetilde{B})\in
\cl\ \ \fo \widetilde{B}\in \cf;$ see (\ref{9.16`}). By (\ref{2.6}),
(\ref{2.8}),  (\ref{9.11}), and (\ref{9.17}) we have the property:
$\la\,\df\,\{\mathbf{f}^{-1}(B):\,B\in \cf\}\in
\cp^\prime(\mathbf{C}_E[\tau])$ and
\bfn\label{9.18}\mathbf{f}^{-1}(F) = \bigcap\limits_{\La\in \la}
\La\in \mathbf{C}_E[\tau].\efn Since the choice of $F$ was
arbitrary, from (\ref{9.18}) we obtain the required continuity
property (see \cite[(2.5.2)]{39}). \end{proof}

\begin{cor}\label{c9.1}{\TL} $\mathbf{f}\in
C\bigl(E,\tau,\mathbf{f}^1(E),\mathbf{T}_\cl^o[E]\bigl|_{\mathbf{f}^1(E)}\bigl).$\end{cor}

\begin{proof} Recall that $\mathbf{f}(x) \in \mathbf{f}^1(E)\ \ \fo
x\in E.$ In addition, by (\ref{9.14}) $$\mathbf{f}^1(E) =
\{\mathbf{f}(x):\,x\in E\}\subset \mathbb{F}_o^* (\cl).$$ Let $G\in
\mathbf{T}_\cl^o[E]\bigl|_{\mathbf{f}^1(E)}$ and $\Gamma\in
\mathbf{T}_\cl^o[E]$ realizes the equality $G = \mathbf{f}^1(E) \cap
\Gamma.$ By Proposition~\ref{p9.2}
\bfn\label{9.19}\mathbf{f}^{-1}(\Gamma)\in \tau.\efn In addition,
$\mathbf{f}^{-1}(G) \subset \mathbf{f}^{-1}(\Gamma)$ (indeed,
$G\subset \Gamma).$ Let $x_*\in \mathbf{f}^{-1}(\Gamma).$ Then,
$x_*\in E$ and $\mathbf{f}(x_*)\in \Gamma.$ But, $\mathbf{f}(x_*)
\in \mathbf{f}^1(E)$ too. Then, $\mathbf{f}(x_*)\in \mathbf{f}^1(E)
\cap \Gamma.$ So, $\mathbf{f}(x_*)\in G.$ Therefore, $x_*\in
\mathbf{f}^{-1}(G).$ Since the choice of $x_*$ was arbitrary, the
inclusion
$$\mathbf{f}^{-1}(\Gamma)\subset \mathbf{f}^{-1}(G)$$ is established. So,
$\mathbf{f}^{-1}(G) = \mathbf{f}^{-1}(\Gamma).$ By (\ref{9.19})
$\mathbf{f}^{-1}(G) \in \tau.$ Since the choice of $G$ was
arbitrary, the inclusion $\mathbf{f}\in
C\bigl(E,\tau,\mathbf{f}^1(E),\mathbf{T}_\cl^o[E]\bigl|_{\mathbf{f}^1(E)}\bigl)$
is established.\end{proof}

Recall that $\mathbf{f}\in (\mathrm{bi})[E;\mathbf{f}^1(E)]$ (see
(\ref{4.21})).

\begin{propos}\label{p9.3}{\TL} $\mathbf{f}\in
C_\mathrm{op}\bigl(E,\tau,\mathbf{f}^1(E),\mathbf{T}_\cl^o[E]\bigl|_{\mathbf{f}^1(E)}\bigl).$
\end{propos}

\begin{proof} Let $G\in \tau.$ Then $\mathbf{f}^1(G) =
\{\mathbf{f}(x):\,x\in G\}$ and $F\,\df\,E\setminus G\in
\mathbf{C}_E[\tau].$ By (\ref{9.11}) $F\in \cl.$ In addition, by
(\ref{5.2}) \bfn\label{9.20}\Phi_\cl(F) = \{\cu\in
\mathbb{F}_o^*(\cl)\,|\,F\in \cu\}.\efn Of course, by (\ref{5.3})
$\Phi_\cl(F) \in (\mathbb{UF})[E;\cl].$ Then
$$\mathbb{F}_o^*(\cl)\setminus \Phi_\cl(F) \in
\mathbf{C}_{\mathbb{F}_o^*(\cl)}\bigl[(\mathbb{UF})[E;\cl]\bigl].$$ As a corollary,
$\mathbb{F}_o^*(\cl)\setminus \Phi_\cl(F) \in \mathbf{T}_\cl^o[E].$ Therefore,
\bfn\label{9.20}\mathbb{G}\,\df\,\mathbf{f}^1(E) \cap \bigl(\mathbb{F}_o^*(\cl)\setminus
\Phi_\cl(F)\bigl)\in \mathbf{T}_\cl^o[E]\bigl|_{\mathbf{f}^1(E)}.\efn Now, we compare
$\mathbf{f}^1(G)$ and $\mathbb{G}$ (\ref{9.20}). Let $\mathcal{V}\in \mathbf{f}^1(G).$
Then, for some $x_*\in G,$ \bfn\label{9.21}\mathcal{V}= \mathbf{f}(x_*) =
(E-\mathrm{ult})[x_*] \cap \cl.\efn Of course, $G\in (E-\mathrm{ult})[x_*].$ By
(\ref{3.8}) $F\notin (E-\mathrm{ult})[x_*] $ (indeed, $G \cap F = \e\notin
(E-\mathrm{ult})[x_*] ).$ By (\ref{9.21}) $F\notin \mathcal{V}$ and, as a corollary,
$\mathcal{V}\notin \Phi_\cl(F);$ see (\ref{9.20}). We obtain that
\bfn\label{9.22}\mathcal{V}\in \mathbb{F}_o^*(\cl) \setminus \Phi_\cl(F).\efn Since
$\mathbf{f}^1(G) \subset \mathbf{f}^1(E),$ we have the inclusion $\mathcal{V}\in
\mathbf{f}^1(E).$ Using (\ref{9.20}) and (\ref{9.22}), we obtain that $\mathcal{V}\in
\mathbb{G}.$ The inclusion \bfn\label{9.23}\mathbf{f}^1(G) \subset \mathbb{G}\efn is
established. Choose arbitrary $\mathcal{W}\in \mathbb{G}.$ Then, by (\ref{9.20}), for
some $x^*\in E,$ the equality $\mathcal{W} = \mathbf{f}(x^*)$ is valid. So,
\bfn\label{9.24}\mathcal{W} = (E-\mathrm{ult})[x^*] \cap \cl.\efn Moreover, $\mathcal{W}
\in \mathbb{F}_o^*(\cl)\setminus \Phi_\cl(F).$ So, $\mathcal{W}\notin \Phi_\cl(F).$ By
(\ref{9.20}) $F\notin \mathcal{W}.$ Since $F\in \cl,$ by (\ref{9.24}) $F \notin
(E-\mathrm{ult})[x^*].$ From (\ref{3.8}), the property $x^* \notin F$ follows. Then,
$x^*\in E\setminus F.$ Therefore, $x^*\in G.$ As a corollary, $\mathcal{W} =
\mathbf{f}(x^*)\in \mathbf{f}^1(G).$ The inclusion $\mathbb{G}\subset \mathbf{f}^1(G)$ is
established. Using (\ref{9.23}), we obtain that $\mathbf{f}^1(G) = \mathbb{G}.$ By
(\ref{9.20}) $$\mathbf{f}^1(G) \in \mathbf{T}_\cl^o[E]\bigl|_{\mathbf{f}^1(E)}.$$ Since
the choice of $G$ was arbitrary, by Corollary~\ref{c9.1} and (\ref{4.22}) we have the
inclusion $$\hspace{3.5cm}\mathbf{f}\in
C_\mathrm{op}\bigl(E,\tau,\mathbf{f}^1(E),\mathbf{T}_\cl^o[E]\bigl|_{\mathbf{f}^1(E)}\bigl).
\hspace{3.8cm}\Box$$
\smallskip

By (\ref{4.24}), (\ref{9.13}), and Proposition~\ref{p9.3} we obtain
that \bfn\label{9.25}\mathbf{f}\in
(\mathrm{Hom})\bigl[E;\tau;\mathbf{f}^1(E);\mathbf{T}_\cl^o[E]\bigl|_{\mathbf{f}^1(E)}\bigl].\efn
So, we construct the concrete homeomorphic inclusion of $T_1$-space
in the compact $T_1$-space (in this connection, we recall that by
Proposition~\ref{p9.1} $$\mathbb{F}_o^*(\cl) =
\mathrm{cl}\bigl(\mathbf{f}^1(E),\mathbf{T}_\cl^o[E]\bigl);$$
moreover, see (\ref{8.12})). So, we have the ``usual'' Wallman
extension.

\section{Ultrafilters of measurable space}
\setcounter{equation}{0}

In this Section, we fix a nonempty set $\mathbf{I}$ and an algebra
$\mathcal{A}$ of subsets of $\mathbf{I}.$ So, in this section,
$(\mathbf{I},\mathcal{A})$ is a measurable space with an algebra of
sets: $\mathcal{A}\in (\mathrm{alg})[\mathbf{I}].$ Of course, we can
to use constructions of Section 5; indeed, in particular, we have
the inclusion $\ca\in (\mathrm{LAT}_o)[\mathbf{I}];$ see
(\ref{2.9}). As a corollary, by (\ref{2.4}) $\ca \in
\pi[\mathbf{I}].$ So, we use the sets $\mathbb{F}^*(\ca)$ and
$\mathbb{F}_o^*(\ca)$ of Section 5; we use properties of these sets
also. We note the known representation  (see \cite[ch. I]{38}):
\bfn\label{10.1}\mathbb{F}_o^*(\ca) = \{\cf\in
\mathbb{F}^*(\ca)\,|\,\fo A\in \ca\ \ (A\in
\cf)\,\vee\,(\mathbf{I}\setminus A\in \cf)\}.\efn Now, we use
(\ref{10.1}) for investigation of TS (\ref{8.11}) in the case $\cl =
\ca.$ First, we note the obvious corollary of (\ref{10.1}):
\bfn\label{10.2}\mathbb{F}_o^*(\ca)\setminus \Phi_\ca(A) =
\Phi_\ca(\mathbf{I}\setminus A)\ \ \fo A\in \ca.\efn

{\bf Remark\,9.1}. Let $A\in \ca$ is fixed. Choose arbitrary
$\cu_1\in \mathbb{F}_o^*(\ca)\setminus \Phi_\ca(A).$ Then by
(\ref{8.12`}) $A\notin \cu_1.$ By (\ref{10.1}) $\mathbf{I}\setminus
A\in \cu_1,$ where $\mathbf{I}\setminus A\in \ca $ by axioms of an
algebra of sets. So, by (\ref{5.2}) $\cu_1\in
\Phi_\ca(\mathbf{I}\setminus A).$ The inclusion
\bfn\label{10.3}\mathbb{F}_o^*(\ca)\setminus \Phi_\ca(A) \subset
\Phi_\ca(\mathbf{I}\setminus A)\efn is established. Let $\cu_2\in
\Phi_\ca(\mathbf{I}\setminus A).$ Then, by (\ref{5.2}) $\cu_2\in
\mathbb{F}_o^*(\ca)$ and $\mathbf{I}\setminus A\in \cu_2.$ By axioms
of a filter $$(A\in \cu_2) \Longrightarrow \bigl(A \cap
(\mathbf{I}\setminus A) \neq \e\bigl).$$ So, $A\notin \cu_2$ and
$\cu_2\notin \Phi_\ca(A).$ As a corollary, $\cu_2\in
\mathbb{F}_o^*(\ca)\setminus \Phi_\ca(A).$ So, the inclusion
$$\Phi_\ca(\mathbf{I}\setminus A) \subset
\mathbb{F}_o^*(\ca)\setminus \Phi_\ca(A)$$ is established. Using
(\ref{10.3}), we obtain the required coincidence
$\mathbb{F}_o^*(\ca)\setminus \Phi_\ca(A)$ and
$\Phi_\ca(\mathbf{I}\setminus A).$ \end{proof}

Returning to (\ref{10.2}) in general case, we note the following
obvious

\begin{propos}\label{p10.1}{\TL} $(\mathbb{UF})[\mathbf{I};\ca] =
\mathbf{C}_{\mathbb{F}_o^*(\ca)}\bigl[(\mathbb{UF})[\mathbf{I};\ca]\bigl].$
\end{propos}

\begin{proof} Let $B_o\in (\mathbb{UF})[\mathbf{I};\ca].$ Using
(\ref{5.3}), we choose $L_o\in \ca$ such that $B_o = \Phi_\ca(L_o).$
Then $\mathbf{I}\setminus L_o\in \ca$ and by (\ref{10.2})
\bfn\label{10.4}\mathbb{F}_o^*(\ca)\setminus B_o =
\mathbb{F}_o^*(\ca)\setminus \Phi_\ca(L_o) =
\Phi_\ca(\mathbf{I}\setminus L_o).\efn  From (\ref{5.3}), we have
the obvious inclusion $\Phi_\ca(\mathbf{I}\setminus L_o)\in
(\mathbb{UF})[\mathbf{I};\ca].$ By (\ref{10.4})
$$\mathbb{F}_o^*(\ca)\setminus B_o\in (\mathbb{UF})[\mathbf{I};\ca].$$
 Therefore, we obtain the following property:
$$B_o = \mathbb{F}_o^*(\ca)\setminus (\mathbb{F}_o^*(\ca)\setminus
B_o) = \mathbb{F}_o^*(\ca)\setminus \Phi_\ca(\mathbf{I}\setminus
L_o)\in
\mathbf{C}_{\mathbb{F}_o^*(\ca)}\bigl[(\mathbb{UF})[\mathbf{I};\ca]\bigl].$$
The inclusion $(\mathbb{UF})[\mathbf{I};\ca]\subset
\mathbf{C}_{\mathbb{F}_o^*(\ca)}\bigl[(\mathbb{UF})[\mathbf{I};\ca]\bigl]$
is established. Choose arbitrary \bfn\label{10.5}\La\in
\mathbf{C}_{\mathbb{F}_o^*(\ca)}\bigl[(\mathbb{UF})[\mathbf{I};\ca]\bigl].\efn
Using (\ref{2.7}), we choose $B^o\in (\mathbb{UF})[\mathbf{I};\ca]$
such that $\La = \mathbb{F}_o^*(\ca)\setminus B^o.$ Let $L^o\in \ca$
be the set for which $B^o = \Phi_\ca(L^o);$ see (\ref{5.3}). Then,
by (\ref{10.2}) \bfn\label{10.6}\La = \mathbb{F}_o^*(\ca)\setminus
\Phi_\ca(L^o) = \Phi_\ca(\mathbf{I}\setminus L^o),\efn  where
$\mathbf{I}\setminus L^o\in \ca.$ Since by (\ref{5.3})
$\Phi_\ca(\mathbb{I}\setminus L^o)\in
(\mathbb{UF})[\mathbf{I};\ca],$ from (\ref{10.6}), we obtain that
$$\La\in (\mathbb{UF})[\mathbf{I};\ca].$$ Since the choice of $\La$
(\ref{10.5}) was arbitrary, the inclusion
$$\mathbf{C}_{\mathbb{F}_o^*(\ca)}\bigl[(\mathbb{UF})[\mathbf{I};\ca]\bigl]
\subset (\mathbb{UF})[\mathbf{I};\ca]$$ is established. So, we
obtain the required equality.\end{proof}

From (\ref{7.3}), (\ref{9.8}), and Proposition~\ref{p10.2}, the
simple (but useful) statement follows.

\begin{propos}\label{p10.2}{\TL} $\mathbf{T}_\ca^*[\mathbf{I}] =
\mathbf{T}_\ca^o[\mathbf{I}].$
\end{propos}

So, for measurable spaces with algebras of sets, the topological
representations of Sections 6 and 7, 8 realize  the same topology.
By (\ref{7.4}), (\ref{8.12}), and Proposition~\ref{p10.2}
\bfn\label{10.7}\mathbf{T}_\ca^*[\mathbf{I}] \in
(c-\mathrm{top})_o[\mathbb{F}_o^*(\ca)]. \efn So, we obtain a
nonempty compactum. Recall that (see (\ref{8.10}),
Proposition~\ref{p10.2})
\bfn\label{10.8}\{\cap\}\bigl((\mathbb{UF})[\mathbf{I};\ca]\bigl) =
\mathbf{C}_{\mathbb{F}_o^*(\ca)}\bigl[\mathbf{T}_\ca^*[\mathbf{I}]\bigl]\efn
is the family of all sets closed in the sense of topology
(\ref{10.7}).  We note the following obvious  property (see
\cite[ch.\,I]{38}) \bfn\label{10.100}(\mathbb{UF})[\mathbf{I};\ca]
\in (\mathrm{alg})[\mathbb{F}_o^*(\ca)].\efn

{\bf Remark\,9.2}. We recall (\ref{5.3}). Let $\Gamma\in
(\mathbb{UF})[\mathbf{I};\ca].$ Using (\ref{5.3}), we choose $\La\in
\ca$ such that $\Gamma = \Phi_\ca(\La).$ Then, $\mathbf{I}\setminus
\La\in \ca$ and by (\ref{10.2})
\bfn\label{10.101}\mathbb{F}_o^*(\ca)\setminus \Gamma =
\mathbb{F}_o^*(\ca)\setminus \Phi_\ca(\La) =
\Phi_\ca(\mathbf{I}\setminus \La).\efn By (\ref{5.3}) and
(\ref{10.101}) $\mathbb{F}_o^*(\ca)\setminus \Gamma\in
(\mathbb{UF})[\mathbf{I};\ca].$ So, we establish that
\bfn\label{10.102}\mathbb{F}_o^*(\ca)\setminus H\in
(\mathbb{UF})[\mathbf{I};\ca]\ \ \fo H\in
(\mathbb{UF})[\mathbf{I};\ca].\efn From (\ref{2.9}), (\ref{5.3}),
and (\ref{10.102}), the property (\ref{10.100}) follows.

\begin{propos}\label{p10.3}{\TL} $(\mathbb{UF})[\mathbf{I};\ca] =
\mathbf{T}_\ca^*[\mathbf{I}] \cap
\mathbf{C}_{\mathbb{F}_o^*(\ca)}[\mathbf{T}_\ca^*\bigl[\mathbf{I]\bigl]}.$
\end{propos}

\begin{proof} Recall that by statements of Section 2 and (\ref{10.8})
the inclusion \bfn\label{10.9}(\mathbb{UF})[\mathbf{I};\ca] \subset
\mathbf{C}_{\mathbb{F}_o^*(\ca)}\bigl[\mathbf{T}_\ca^*[\mathbf{I}]\bigl].\efn
From (\ref{7.3}), the inclusion
$(\mathbb{UF})[\mathbf{I};\ca]\subset \mathbf{T}_\ca^*[E]$ follows
too. So, by (\ref{10.9})
\bfn\label{10.10}(\mathbb{UF})[\mathbf{I};\ca]\subset
\mathbf{T}_\ca^*[\mathbf{I}] \cap
\mathbf{C}_{\mathbb{F}_o^*(\ca)}\bigl[\mathbf{T}_\ca^*[\mathbf{I}]\bigl].\efn
Let $\Om\in \mathbf{T}_\ca^*[\mathbf{I}] \cap
\mathbf{C}_{\mathbb{F}_o^*(\ca)}\bigl[\mathbf{T}_\ca^*[\mathbf{I}]\bigl].$
Since $\Om$ is open, then by (\ref{7.3}) we obtain that, for some
family \bfn\label{10.11}\mathfrak{W}\in
\cp\bigl((\mathbb{UF})[\mathbf{I};\ca]\bigl),\efn the following
equality is realized: \bfn\label{10.12}\Om =
\bigcup\limits_{W\in\,\mathfrak{W}} W.\efn If $\mathfrak{W} = \e,$
then by (\ref{10.12}) $\Om = \e$ and, as a corollary, $\Om =
\Phi_\ca (\e),$ where $\e\in \ca.$ So, by (\ref{5.3}) we obtain the
implication \bfn\label{10.13}(\mathfrak{W} = \e) \Longrightarrow
\bigl(\Om\in (\mathbb{UF})[\mathbf{I};\ca]\bigl).\efn Let
$\mathfrak{W}\neq \e.$ Then, $\mathfrak{W}\in
\cp^\prime\bigl((\mathbb{UF})[\mathbf{I};\ca]\bigl).$ Since $\Om$ is
a closed subset of a compactum, we have the compactess property of
$\Om;$ then, by (\ref{10.11}), for some $\mathbb{K}\in
\mathrm{Fin}(\mathfrak{W})$ \bfn\label{10.14}\Om =
\bigcup\limits_{W\in\,\mathbb{K}}W.\efn In particular,
$\mathbb{K}\in
\mathrm{Fin}\bigl((\mathbb{UF})[\mathbf{I};\ca]\bigl).$ We note that
$(\mathbb{UF})[\mathbf{I};\ca]$ is closed with respect to finite
unions (indeed, by (\ref{10.100}) $(\mathbb{UF})[\mathbf{I};\ca]$ is
an algebra of sets). Therefore, by (\ref{10.14}) $\Om\in
(\mathbb{UF})[\mathbf{I};\ca]$ in the case $\mathfrak{W}\neq \e.$
So, \bfn\label{10.15}(\mathfrak{W}\neq \e) \Longrightarrow
\bigl(\Om\in (\mathbb{UF})[\mathbf{I};\ca]\bigl).\efn  Using
(\ref{10.13}) and (\ref{10.15}), we obtain that $\Om\in
(\mathbb{UF})[\mathbf{I};\ca]$ in any possible cases. Since the
choice of $\Om$ was arbitrary, the inclusion
\bfn\label{10.15`}\mathbf{T}_\ca ^*[\mathbf{I}] \cap
\mathbf{C}_{\mathbb{F}_o^*(\ca)}\bigl[\mathbf{T}_\ca^*[\mathbf{I}]\bigl]
\subset (\mathbb{UF})[\mathbf{I};\ca]\efn in established. From
(\ref{10.10}) and (\ref{10.15}), the required statement follows.
\end{proof}

So, $(\mathbb{UF})[\mathbf{I};\ca]$ is the family of all open-closed
sets in the nonempty compactum \bfn\label{10.16}
\bigl(\mathbb{F}_o^*(\ca), \mathbf{T}_\ca^*[\mathbf{I}]\bigl) =
\bigl(\mathbb{F}_o^*(\ca), \mathbf{T}_\ca^o[\mathbf{I}]\bigl).\efn
In connection with the above-mentioned property of nonempty
compactum (\ref{10.16}), we recall \cite[ch.\,I]{38}. With the
employment of (\ref{10.1}), the following obvious property is
established: in our case of measurable space with an algebra of sets
\bfn\label{10.17}(\mathbf{I}-\mathrm{ult})[x] \cap \ca \in
\mathbb{F}_o^*(\ca)\ \ \fo x\in \mathbf{I}.\efn

{\bf Remark\,9.3} For a completeness, we consider the scheme  of the
proof of (\ref{10.17}). For this, we note that by (\ref{3.8}) and
the corresponding definition of Section 5
\bfn\label{10.18}(\mathbf{I}-\mathrm{ult})[x] \cap \cl\in
\mathbb{F}^*(\cl)\ \ \fo \cl\in \pi[\mathbf{I}]\ \ \fo x\in
\mathbf{I}.\efn In particular, by (\ref{10.18})
$(\mathbf{I}-\mathrm{ult})[x] \cap \ca\in \mathbb{F}^*(\ca)\ \ \fo
x\in \mathbf{I}.$ Fix $x_*\in \mathbf{I}$ and suppose that
$$\cf_*\,\df\,(\mathbf{I}-\mathrm{ult})[x_*] \cap \ca;$$
of course, $\cf_*\in \mathbb{F}^*(\ca).$ In addition, $\ca\subset
\cp(\mathbf{I}).$  Then, $\fo A\in \ca$ \bfn\label{10.19}(x_*\in
A)\,\vee\,(x_*\in \mathbf{I}\setminus A).\efn Of course, by
(\ref{3.8}), for $A\in \ca$, we have the following obvious
implications: $$\bigl((x_*\in A) \Longrightarrow (A\in
\cf_*)\bigl)\,\&\,\bigl((x_*\in \mathbf{I}\setminus A)
\Longrightarrow (\mathbf{I}\setminus A\in \cf_*)\bigl).$$ Then, by
(\ref{10.19}) $(A\in \cf_*)\,\vee\,(\mathbf{I}\setminus A\in
\cf_*).$ Since the choice of $A$ was arbitrary, by (\ref{10.1})
$\cf_*\in \mathbb{F}_o^*(\ca).$ So, (\ref{10.17}) is established.

Using (\ref{10.17}), we introduce the mapping
\bfn\label{10.20}(\ca-\mathrm{ult})[\mathbf{I}]\,\df\,\bigl((\mathbf{I}-\mathrm{ult})[x]
\cap \ca\bigl)_{x\in \mathbf{I}}\,\in
\mathbb{F}_o^*(\ca)^\mathbf{I}.\efn Of course, in (\ref{10.20}) we
have analog of the mapping $\mathbf{f}$ (\ref{9.12}). But, in the
given case, we realize the immersion of points of the initial set in
the ultrafilter space under other conditions. We will use the
specific character of measurable space with an algebra of sets. Now,
we note the obvious property:
\begin{equation}\label{10.21}\begin{array}{c}\bigl(\fo x\in \mathbf{I}\ \
\fo y\in \mathbf{I}\setminus \{x\}\ \ \exists A\in \ca:(x\in
A)\,\&\,(y\notin A)\bigl) \Longrightarrow\\ \Longrightarrow
\bigl((\ca-\mathrm{ult})\bigl[\mathbf{I}]\in
(\mathrm{bi})[\mathbf{I};(\ca-\mathrm{ult})[\mathbf{I}]^1(\mathbf{I})\bigl]\bigl).
\end{array}\end{equation} In (\ref{10.21}), the statement of the premise has the
following sense: algebra $\ca$ is distinguishing for points of $\mathbf{I}.$

If $\mathcal{J}\in \cp^\prime(\ca),$ then by analogy with Section 4 we suppose that

\begin{equation}\label{10.21`}\begin{array}{c}\bigl(\mathbb{F}^*(\ca
|\,\mathcal{J})\,\df\,\{\cf\in \mathbb{F}^*(\ca)\,|\,\mathcal{J}\subset \cf\}\bigl)\,\&\\
\&\,\bigl(\mathbb{F}_o^*(\ca\,|\,\mathcal{J})\,\df\,\{\cu\in
\mathbb{F}_o^*(\ca)\,|\,\mathcal{J}\subset \cu\}\bigl);
\end{array}\end{equation} of course, $\mathbb{F}_o^*(\ca\,|\,\mathcal{J})\subset
\mathbb{F}^*(\ca\,|\,\mathcal{J})$ and moreover the following
property is valid:\bfn\label{10.21``}\fo \cf\in
\mathbb{F}^*(\ca\,|\,\mathcal{J})\ \ \exists \cu\in
\mathbb{F}_o^*(\ca\,|\,\mathcal{J}): \cf\subset \cu.\efn Returning
to (\ref{10.21}), we note that
\bfn\label{10.22}(\ca-\mathrm{ult})[\mathbf{I}]^{-1}\bigl(\mathbb{F}_o^*(\ca\,|\,
\mathcal{I})\bigl) = \bigcap\limits_{A\in\,\mathcal{I}}A\ \ \fo
\mathcal{I}\in \cp^\prime(\ca).\efn In (\ref{10.22}), we can use
$\mathcal{I}$ as constraints of asymptotic character. Of course,
$\mathbf{F}_o^*(\ca) \subset \mathbf{F}^*(\ca) \subset
\beta_\ca^o[\mathbf{I}]\subset \beta_o[\mathbf{I}]$ (see Section 5).
Then, by (\ref{3.3})
\begin{equation}\label{10.23}(\mathbf{I}-\mathrm{\mathbf{fi}})[\cu]
= \{H\in \cp(\mathbf{I})\,|\,\exists
B\in \cu:\,B\subset H\}\in
\mathfrak{F}[\mathbf{I}]\ \ \fo \cu\in \mathbb{F}_o^*(\ca).
\end{equation} By analogy with (\ref{10.23}) we note that $\mathbb{F}^*(\ca)\subset
\beta_o[\mathbf{I}]$ and $(\mathbf{I}-\mathrm{\mathbf{fi}})[\cf]\in
\mathfrak{F}[\mathbf{I}]\ \ \fo \cf\in \mathbb{F}^*(\ca).$ These
properties permit realize an asymptotic analogs of solutions of the
set (\ref{10.22}). In this capacity, we can use elements of the sets
$\mathbb{F}^*(\ca |\,\mathcal{I})$ and  $\mathbb{F}_o^*(\ca
|\,\mathcal{I}),$ where $\mathcal{I}\in \cp^\prime(\ca)$ is used as
``asymptotic constraints''.  Of course, $\ca$ bounds our
possibilities: we can use only subfamilies of $\ca.$

\begin{propos}\label{p10.4}{\TL} $\mathbb{F}_o^*(\ca) =
\mathrm{cl}\bigl((\ca-\mathrm{ult})[\mathbf{I}]^1(\mathbf{I}),\mathbf{T}_\ca^*[\mathbf{I}]\bigl).$
\end{propos}

\begin{proof} Fix $\cf\in \mathbb{F}_o^*(\ca).$ Let $A\in \cf.$ Then
$A\in \cp^\prime(\mathbf{I}).$ So, $A\neq \e$ and $A\subset
\mathbf{I}.$ Choose arbitrary $a\in A.$ Then, by (\ref{10.20})
\bfn\label{10.24}(\mathbf{I}-\mathrm{ult})[a] \cap \ca =
(\ca-\mathrm{ult})[\mathbf{I}](a) \in
(\ca-\mathrm{ult})[\mathbf{I}]^1(\mathbf{I}).\efn By the choice of
$a$ we have the inclusion $A\in (\mathbf{I}-\mathrm{ult})[a].$ Since
$\cf\subset \ca,$ we obtain that $A\in \ca.$ Then, by (\ref{10.24})
$A\in (\mathbf{I}-\mathrm{ult})[a] \cap \ca.$ Since
$(\mathbf{I}-\mathrm{ult})[a] \cap \ca\in \mathbb{F}_o^*(\ca),$ by
(\ref{5.2}) \bfn\label{10.25}(\mathbf{I}-\mathrm{ult})[a] \cap
\ca\in \Phi_\ca (A).\efn  By (\ref{10.24}) and (\ref{10.25}) we
obtain the following property $$\Phi_\ca(A) \cap
(\ca-\mathrm{ult})[\mathbf{I}]^1(\mathbf{I}) \neq \e.$$ Since the
choice of $A$ was arbitrary, we have (see (\ref{9.3})) the statement
\bfn\label{10.26}\Phi_\ca(L) \cap
(\ca-\mathrm{ult})[\mathbf{I}]^1(\mathbf{I}) \neq \e\ \ \fo L\in
\cf.\efn Choose arbitrary $\Om\in N_{\mathbf{T}_\ca^*[\mathbf{I}]}
(\cf).$ Then, for some $\Om^o\in N_{\mathbf{T}_\ca^*[\mathbf{I}]}^o
(\cf),$ the inclusion $\Om^o\subset \Om$ is valid. Therefore,
$\Om^o\in \mathbf{T}_\ca^*[\mathbf{I}]$ and $\cf\in \Om^o.$ By
(\ref{7.3}), there exists $\La\in \cf$ such that
\bfn\label{10.27}\Phi_\ca(\La) \subset \Om^o.\efn From
(\ref{10.26}), the property $\Phi_\ca(\La) \cap
(\ca-\mathrm{ult})[\mathbf{I}]^1(\mathbf{I})\neq \e$ is valid. By
(\ref{10.27}) we obtain that $$\Om \cap
(\ca-\mathrm{ult})[\mathbf{I}]^1(\mathbf{I})\neq \e$$ (indeed,
$\Phi_\ca(\La) \subset \Om).$ Since the choice of $\Om$ was
arbitrary, $$S \cap (\ca-\mathrm{ult})[\mathbf{I}]^1(\mathbf{I})
\neq \e\ \ \fo S\in N_{\mathbf{T}_\ca^*[\mathbf{I}]}(\cf).$$ Then,
$\cf \in
\mathrm{cl}\bigl((\ca-\mathrm{ult})[\mathbf{I}]^1(\mathbf{I}),
\mathbf{T}_\ca^*[\mathbf{I}]\bigl).$ So, the inclusion
$$\mathbb{F}_o^*(\ca) \subset
\mathrm{cl}\bigl((\ca-\mathrm{ult})[\mathbf{I}]^1(\mathbf{I}),
\mathbf{T}_\ca^*[\mathbf{I}]\bigl)$$ is established. The opposite
inclusion is  obvious. \end{proof}

We note that Proposition~\ref{p10.4} is similar to
Proposition~\ref{p9.1}. But, in the given section, the condition
\bfn\label{10.28}\{x\}\in \ca\ \ \fo x\in \ca\efn was supposed not.
In construtions of Section 8 (in particular, in
Proposition~\ref{p9.1}), the condition similar to (\ref{10.28}) is
essential. So, Proposition~\ref{p10.4} has the independent meaning.

\section{Attraction sets under the restriction in the form of algebra of sets}
\setcounter{equation}{0}

In the following, we fix a nonempty set $E,$ a TS
$(\mathbf{H},\tau),$ where $\mathbf{H}\neq \e,$ and a mapping
$\mathbf{h}\in \mathbf{H}^E.$ Elements $e\in E$ are considered as
usual solutions and elements $y\in \mathbf{H}$ play the role of some
estimates. The natural variant of an obtaining of $y$ is realized in
the form $y = \mathbf{h}(e),$ where $e\in E.$ But, we admit the
possibility of the limit realization of $y.$ This is natural in
questions of asymptotic analysis. In the last case, it is natural to
use ``asymptotic constraints'' in the form of a nonempty subfamilies
of $\cp(E).$ Then, we obtain constructions of Section 4 under $X =
E,\,Y = \mathbf{H},$ and $f =\mathbf{h}.$ But, we admit yet one
possibility: along with ``usual'' AS, we use the sets

\begin{equation}\label{11.1}\begin{array}{c}(\tau-\mathbb{AS})[\mathcal{E}\,|\,\ca]\,\df\,
\{y\in \mathbf{H}\,|\,\exists \cf\in \mathbb{F}^*(\ca\,
|\,\mathcal{E}):\,\mathbf{h}^1[\cf]\,{\stackrel{\tau}{\Longrightarrow}}\,y\}\\  \fo
\ca\in (\mathrm{alg})[E]\ \ \fo \mathcal{E}\in \cp^\prime(\ca).
\end{array}\end{equation}
Of  course, we use remarks of the conclusion of the previous section.

\begin{propos}\label{p11.1}{\TL} If $\ca\in (\mathrm{alg})[E]$ and $\mathcal{E}\in
\cp^\prime(\ca),$ then \bfn\label{11.2}(\tau-\mathbb{AS})[\mathcal{E}\,|\,\ca] = \{y\in
\mathbf{H}\,|\,\exists \cu\in \mathbb{F}_o^*(\ca\,
|\,\mathcal{E}):\,\mathbf{h}^1[\cu]{\stackrel{\tau}{\Longrightarrow}}\,y\}.\efn
\end{propos}

\begin{proof} We use reasoning analogous to the proof of
Proposition~\ref{p4.2}. We denote by $\Om$ the set on the right side
of (\ref{11.2}). Since $\mathbb{F}_o^*(\ca |\,\mathcal{E})\subset
\mathbb{F}^*(\ca |\,\mathcal{E})$ (see Section 9), by (\ref{11.1})
\bfn\label{11.3}\Om\subset
(\tau-\mathbb{AS})[\mathcal{E}|\,\ca].\efn Let $y_o\in
(\tau-\mathbb{AS})[\mathcal{E}|\,\ca].$ Then, by (\ref{11.1})
$y_o\in \mathbf{H}$ and, for some $\cf\in \mathbb{F}^*(\ca
|\,\mathcal{E}),$
\bfn\label{11.4}\mathbf{h}^1[\cf]\,{\stackrel{\tau}{\Longrightarrow}}\,y_o.\efn
Recall that $\cf\in \beta_o[E]$ (see Section 9). Therefore, by
(\ref{4.1}) $\mathbf{h}^1[\cf]\in \beta_o[\mathbf{H}].$ Then,
(\ref{11.4}) denotes that \bfn\label{11.5}N_\tau(y_o) \subset
(\mathbf{H}-\mathrm{\mathbf{fi}})\bigl[\mathbf{h}^1[\cf]\bigl]\efn
(see (\ref{3.4})). In addition, by the choice of $\cf$ we have the
inclusion $\mathcal{E}\subset \cf;$ see (\ref{10.21`}). By
(\ref{10.21``}), for some $\mathfrak{U}\in \mathbb{F}_o^*(\ca
|\,\mathcal{E}),$ the inclusion $\cf\subset \mathfrak{U}$ is valid.
Then,
$$\mathbf{h}^1[\cf]\subset \mathbf{h}^1[\mathfrak{U}].$$ As a corollary, by(\ref{3.3})
and (\ref{11.5}) $$N_\tau(y_o) \subset
(\mathbf{H}-\mathrm{\mathbf{fi}})\bigl[\mathbf{h}^1[\cf]\bigl]
\subset
(\mathbf{H}-\mathrm{\mathbf{fi}})\bigl[\mathbf{h}^1[\mathfrak{U}]\bigl],$$
where $\mathbf{h}^1[\mathfrak{U}] \in \beta_o[\mathbf{H}]$ (see
Section 9). Then, by (\ref{3.4})
\bfn\label{11.6}\mathbf{h}^1[\mathfrak{U}]\,{\stackrel{\tau}{\Longrightarrow}}\,y_o.\efn
By definition of $\Om$ we obtain that $y_o\in \Om.$ Since the choice
of $y_o$ was arbitrary, the inclusion
\bfn\label{11.7}(\tau-\mathbb{AS})[\mathcal{E} |\,\ca]\subset
\Om\efn is established. Using (\ref{11.3}) and (\ref{11.7}), we
obtain the required equality
\bfn\label{11.7`}(\tau-\mathbb{AS})[\mathcal{E} |\,\ca] = \Om.\efn
From the definition of $\Om$ and (\ref{11.7`}), we obtain
(\ref{11.2}). \end{proof}

Recall that $\cp(E)\in (\mathrm{alg})[E]$ and therefore
$$(\tau-\mathbb{AS})[\mathcal{E} |\,\cp(E)]\in \cp(\mathbf{H})\ \ \fo\mathcal{E}
\in \cp^\prime\bigl(\cp(E)\bigl).$$ By definitions of Section 3, (\ref{7.5}), and
(\ref{10.21`}) we obtain that \bfn\label{11.9}\mathfrak{F}_\mathbf{u}^o[E |\,\mathcal{E}]
= \mathbb{F}_o^*\bigl(\cp(E) |\,\mathcal{E}\bigl)\ \ \fo \mathcal{E}\in
\cp^\prime\bigl(\cp(E)\bigl).\efn From Propositions~\ref{p4.2} and \ref{p11.1}, we have
(see (\ref{11.9})) the property:
$$(\mathrm{\mathbf{as}})[E;\mathbf{H};\tau;\mathbf{h};\mathcal{E}] =
(\tau-\mathbb{AS})[\mathcal{E} |\,\cp(E)]\ \ \fo \mathcal{E}\in
\cp^\prime\bigl(\cp(E)\bigl).$$ So, our new construction is
coordinated with AS of Section 4. Moreover, under $\ca\in
(\mathrm{alg})[E],$ we can consider AS
$(\mathrm{\mathbf{as}})[E;\mathbf{H};\tau;\mathbf{h};\mathcal{E}]$
for $\mathcal{E}\in \cp^\prime(\ca).$

\begin{propos}\label{p11.2}{\TL} If $\ca\in (\mathrm{alg})[E]$ and $\mathcal{E}\in
\cp^\prime(\ca),$ then
\bfn\label{11.9`}(\tau-\mathbb{AS})[\mathcal{E}\,|\,\ca]\subset
(\mathrm{\mathbf{as}})
[E;\mathbf{H};\tau;\mathbf{h};\mathcal{E}].\efn \end{propos}

\begin{proof} We use (\ref{7.7}). Choose $y_*\in (\tau-\mathbb{AS})[\mathcal{E}\,|\,\ca].$
Then, $y_*\in \mathbf{H}$ and, for some $\cu_*\in \mathbb{F}_o^*(\ca\,|\,\mathcal{E}),$
the convergence
\bfn\label{11.10}\mathbf{h}^1[\cu_*]\,{\stackrel{\tau}{\Longrightarrow}}\,y_*\efn is
valid. Then, $\cu_*\in \mathbb{F}_o^*(\ca)$ and $\mathcal{E}\subset \cu_*;$ see
(\ref{10.21`}).  By (\ref{7.7}) for some $\cu^*\in \mathfrak{F}_\mathbf{u}[E],$ the
equality $\cu_* = \cu^* \cap \ca$ is valid. Then, $\mathcal{E}\subset \cu^*.$ As a
corollary, $\cu^*\in \mathfrak{F}_\mathbf{u}^o[E\,|\,\mathcal{E}].$ Now, we return to
(\ref{11.10}). In addition, $\cu_*\in \beta_o[E].$ Therefore, $\mathbf{h}^1[\cu_*]\in
\beta_o[\mathbf{H}]$ and by (\ref{3.3})
$$(\mathbf{H}-\mathrm{\mathbf{fi}})\bigl[\mathbf{h}^1[\cu_*]\bigl]\in \mathfrak{F}[\mathbf{H}].$$
From (\ref{3.4}) and (\ref{11.10}), we have the obvious inclusion
\bfn\label{11.11}N_\tau(y_*)\subset
(\mathbf{H}-\mathrm{\mathbf{fi}})\bigl[\mathbf{h}^1[\cu_*]\bigl].\efn
In addition, $\cu^*\in \beta_o[E]$ and $\mathbf{h}^1[\cu^*]\in
\beta_o[\mathbf{H}];$ see (\ref{4.1}). Since $\cu_*\subset \cu^*,$
the inclusion $\mathbf{h}^1[\cu_*] \subset \mathbf{h}^1[\cu^*]$ is
valid. As a corollary, by (\ref{3.3})
$$(\mathbf{H}-\mathrm{\mathbf{fi}})\bigl[\mathbf{h}^1[\cu_*]\bigl]\subset
(\mathbf{H}-\mathrm{\mathbf{fi}})\bigl[\mathbf{h}^1[\cu^*]\bigl].$$
Using (\ref{11.11}), we obtain the basic inclusion
\bfn\label{11.12}N_\tau(y_*) \subset
(\mathbf{H}-\mathrm{\mathbf{fi}})\bigl[\mathbf{h}^1[\cu^*]\bigl].\efn
From (\ref{3.4}) and (\ref{11.12}), we obtain the following
convergence
\bfn\label{11.13}\mathbf{h}^1[\cu^*]\,{\stackrel{\tau}{\Longrightarrow}}\,y_*.\efn
So, $\cu^*\in \mathfrak{F}_\mathbf{u}^o[E\,|\,\mathcal{E}]$ has the
property (\ref{11.13}). Then, by Proposition~\ref{p4.2} $$y_*\in
(\mathrm{\mathbf{as}})[E;\mathbf{H};\tau;\mathbf{h};\mathcal{E}].$$
Since the choice of $y_*$ was arbitrary, the required inclusion
(\ref{11.9`}) is established. \end{proof}

So, by (\ref{11.1}) and (\ref{11.2}) some ``partial'' AS are
defined. Of course, the case for which (\ref{11.9`}) is converted in
a equality is very interesting. For investigation of this case, we
consider auxiliary constructions. In the following, in this section,
we fix $\ca\in (\mathrm{alg})[E].$ So, $(E,\ca)$ is a measurable
space with an algebra of sets. In this case, we can supplement the
property (\ref{7.7}). Namely, \bfn\label{11.14}\cu \cap \ca\in
\mathbb{F}_o^*(\ca)\ \ \fo \cu\in \mathfrak{F}_\mathbf{u}[E].\efn

{\bf Remark\,10.1}. We omit the sufficiently simple proof of
(\ref{11.14}). Now, we are restricted to brief remarks. Namely, by
ultrafilter $\cu\in \mathfrak{F}_\mathbf{u}[E]$ we can realize a
finitely additive (0,1)-measure $\mu$ on the family $\cp(E)$
supposing that $\mu(L)\,\df\,1$ under $L\in \cu$ and
$\mu(\La)\,\df\,0$ under $\La\in \cp(E)\setminus \cu.$ In connection
with such possibility, we use \cite[(7.6.17)]{32} (moreover, see
\cite[(7.6.7)]{32}). The natural narrowing $\nu$ of $\mu$ on our
algebra $\ca$ is finitely additive (0,1)-measure on $\ca$ (of
course, $\nu = (\mu |\,\ca)).$ Therefore, for some $\mathcal{V}\in
\mathbb{F}_o^*(\ca),$ by \cite[(7.6.17)]{32} $\nu$ is defined by the
rule \bfn\label{11.15}\bigl(\nu(A) = 1\ \ \fo A\in
\mathcal{V}\bigl)\,\&\,\bigl(\nu(\widetilde{A}) = 0\ \ \fo
\widetilde{A}\in \ca \setminus \mathcal{V}\bigl).\efn On the other
hand, the family $\cu \cap \ca$ realizes $\nu$ by the obvious rule:
\bfn\label{11.16}\bigl(\nu(A) = 1\ \ \fo A\in \cu \cap
\ca\bigl)\,\&\, \bigl(\nu(\widetilde{A}) = 0\ \ \fo \widetilde{A}\in
\ca \setminus (\cu \cap \ca)\bigl).\efn From (\ref{11.15}) and
(\ref{11.16}), the required equality $\cu \cap \ca = \mathcal{V}$
follows. Then, by the choice of $\mathcal{V}$ we have the inclusion
$\cu \cap \ca \in \mathbb{F}_o^*(\ca).$

Using (\ref{7.7}) and (\ref{11.14}), we obtain that
\bfn\label{11.17} \mathbb{F}_o^*(\ca) = \{\cu \cap \ca:\,\cu\in
\mathfrak{F}_\mathbf{u} [E]\}.\efn By (\ref{11.17}) we establish the
natural connection of $\mathfrak{F}_\mathbf{u}[E]$ and
$\mathbb{F}_o^*(\ca).$ Now, we consider some other auxiliary
properties.

If $\mathcal{B}\in \beta_o[\mathbf{H}]$ and $z\in \mathbf{H},$ then we have the following
equivalence \bfn\label{11.18}(\mathcal{B}\,{\stackrel{\tau}{\Longrightarrow}}\,z)
\Longleftrightarrow \bigl(N_\tau^o(z) \subset
(\mathbf{H}-\mathrm{\mathbf{fi}})[\mathcal{B}]\bigl).\efn Of course, we can use instead
of $\mathcal{B}$ the corresponding image of a filter base in $E.$ Indeed, by (\ref{4.1})
and (\ref{11.18}) $\fo \mathcal{B}\in \beta_o[E]\ \ \fo z\in \mathbf{H}$
\bfn\label{11.19}(\mathbf{h}^1[\mathcal{B}]\,{\stackrel{\tau}{\Longrightarrow}}\,z)
\Longleftrightarrow \bigl(N_\tau^o(z) \subset
(\mathbf{H}-\mathrm{\mathbf{fi}})\bigl[\mathbf{h}^1[\mathcal{B}]\bigl]\bigl).\efn

Moreover, in connection with (\ref{11.19}), we note that $\fo
\mathcal{B}\in \beta_o[E]\ \ \fo z\in \mathbf{H}$
\bfn\label{11.20}(\mathbf{h}^1[\mathcal{B}]\,{\stackrel{\tau}{\Longrightarrow}}\,z)
\Longleftrightarrow \bigl(\mathbf{h}^{-1}[N_\tau^o(z)] \subset
(E-\mathrm{\mathbf{fi}})[\mathcal{B}]\bigl).\efn

{\bf  Remark\,10.2}. Consider the proof of (\ref{11.20}). Fix
$\mathcal{B}\in \beta_o[E]$ and $z\in \mathbf{H}.$ Let
$\mathbf{h}^1[\mathcal{B}]\,{\stackrel{\tau}{\Rightarrow}}\,z.$
Then, by (\ref{11.19})
$$N_\tau^o(z) \subset
(\mathbf{H}-\mathrm{\mathbf{fi}})\bigl[\mathbf{h}^1[\mathcal{B}]\bigl].$$
Therefore, for any $G_*\in N_\tau^o(z)$ there exists $B_*\in
\mathcal{B}$ such that $\mathbf{h}^1(B_*)\subset G_*.$ As a
corollary, $$B_*\subset
\mathbf{h}^{-1}\bigl(\mathbf{h}^1(B_*)\bigl)\subset
\mathbf{h}^{-1}(G_*).$$ Then, $\mathbf{h}^{-1}(G_*)\in
(E-\mathrm{\mathbf{fi}})[\mathcal{B}].$ Since the choice of $G_*$
was arbitrary, $$\mathbf{h}^{-1}[N_\tau^o(z)]\subset
(E-\mathrm{\mathbf{fi}})[\mathcal{B}].$$ So,
$(\mathbf{h}^1[\mathcal{B}]\,{\stackrel{\tau}{\Longrightarrow}}\,z)\Longrightarrow
\bigl(\mathbf{h}^{-1}[N_\tau^o(z)]\subset
(E-\mathrm{\mathbf{fi}})[\mathcal{B}]\bigl).$ Let
\bfn\label{11.21}\mathbf{h}^{-1}[N_\tau^o(z)]\subset
(E-\mathrm{\mathbf{fi}})[\mathcal{B}].\efn Choose arbitrary
neighborhood $G^*\in N_\tau^o(z).$ Then, by (\ref{11.21})
$\mathbf{h}^{-1}(G^*)\in (E-\mathrm{\mathbf{fi}})[\mathcal{B}].$
Therefore, for some $B^*\in \mathcal{B},$ the inclusion $B^*\subset
\mathbf{h}^{-1}(G^*)$ is valid. In addition, $\mathbf{h}^1(B^*)\in
\mathbf{h}^1[\mathcal{B}]$ and $$\mathbf{h}^1(B^*)\subset
\mathbf{h}^1\bigl(\mathbf{h}^{-1}(G^*)\bigl)\subset G^*.$$ Then,
$G^*\in
(\mathbf{H}-\mathrm{\mathbf{fi}})\bigl[\mathbf{h}^1[\mathcal{B}]\bigl].$
Therefore, $N_\tau^o(z) \subset
(\mathbf{H}-\mathrm{\mathbf{fi}})\bigl[\mathbf{h}^1[\mathcal{B}]\bigl]$
and by (\ref{11.19})
$\mathbf{h}^1[\mathcal{B}]\,{\stackrel{\tau}{\Rightarrow}}\,z.$ So,
$$\bigl(\mathbf{h}^{-1}[N_\tau^o(z)]\subset (E-\mathrm{\mathbf{fi}})
[\mathcal{B}]\bigl)
\Longrightarrow(\mathbf{h}^1[\mathcal{B}]\,{\stackrel{\tau}{\Longrightarrow}}\,z\bigl).$$
The proof of (\ref{11.20}) is completed.

We note that, in (\ref{11.20}), we can use  instead of $\mathcal{B}$
arbitrary filter of $(E,\ca).$ In this connection, we recall that by
constructions of Section 5, for any $\cf\in \mathbb{F}^*(\ca),$ we
obtain (in particular) that $\cf\in \beta_o[E]$ and
\bfn\label{11.22}(E-\mathrm{\mathbf{fi}})[\cf] \cap \ca =
(E-\mathrm{\mathbf{fi}})[\cf\,|\,\ca] = \cf.\efn Then, from
(\ref{11.20}) and (\ref{11.22}), we have the following property:
$\fo \cf\in \mathbb{F}^*(\ca)\ \ \fo z\in \mathbf{H}$
\bfn\label{11.23}(\mathbf{h}^1[\cf]\,{\stackrel{\tau}{\Longrightarrow}}\,z)
\Longleftrightarrow \bigl(N_\tau^o(z)\subset
(\mathbf{H}-\mathrm{\mathbf{fi}})\bigl[\mathbf{h}^1[\cf]\bigl]\bigl).\efn
Of course, (\ref{11.23}) is the particular case of (\ref{11.20}); in
(\ref{11.22}), we have the useful addition. We note that $\fo
\mathcal{B}\in \beta_o[\mathbf{H}]\ \ \fo z\in \mathbf{H}\ \ \fo
\mathcal{Z}\in (z-\mathrm{bas})[\tau]$
\bfn\label{11.24}(\mathcal{B}\,{\stackrel{\tau}{\Longrightarrow}}\,z)
\Longleftrightarrow\bigl(\mathcal{Z}\subset (\mathbf{H}-\mathbf{fi}
)[\mathcal{B}]\bigl).\efn

{\bf Remark\,10.3}. Fix $\mathcal{B}\in \beta_o[\mathbf{H}], z\in
\mathbf{H},$ and $\mathcal{Z}\in (z-\mathrm{bas})[\tau].$ Consider
the proof of (\ref{11.24}). By (\ref{2.18}) and (\ref{3.4}) we have
the following  implication
\bfn\label{11.25}(\mathcal{B}\,{\stackrel{\tau}{\Longrightarrow}}\,z)
\Longrightarrow\bigl(\mathcal{Z}\subset (\mathbf{H}-\mathbf{fi}
)[\mathcal{B}]\bigl).\efn Let $\mathcal{Z}\subset
(\mathbf{H}-\mathrm{\mathbf{fi}})[\mathcal{B}].$ Choose arbitrary
$S\in N_\tau(z).$ Then, by (\ref{2.17}), for some $Z\in
\mathcal{Z},$ the inclusion $Z\subset S$ is valid. Since $Z\in
(\mathbf{H}-\mathrm{\mathbf{fi}})[\mathcal{B}],$ by filter axioms
(see \ref{3.1})) $S\in
(\mathbf{H}-\mathrm{\mathbf{fi}})[\mathcal{B}].$ So, the inclusion
$N_\tau(z) \subset (\mathbf{H}-\mathrm{\mathbf{fi}})[\mathcal{B}]$
is established. By (\ref{3.4}) we have the convergence
$\mathcal{B}\,{\stackrel{\tau}{\Rightarrow}}\,z.$ So, $$
\bigl(\mathcal{Z}\subset
(\mathbf{H}-\mathrm{\mathbf{fi}})[\mathcal{B}]\bigl) \Longrightarrow
(\mathcal{B}\,{\stackrel{\tau}{\Longrightarrow}}\,z).$$ Now, with
the employment of (\ref{11.25}), we obtain (\ref{11.24}).

We note the following obvious corollary of (\ref{11.24}) (in this
connection, we recall (\ref{11.20})): $\fo \mathcal{B}\in
\beta_o[E]\ \ \fo z\in \mathbf{H}$
\bfn\label{11.26}(\mathbf{h}^1[\mathcal{B}]\,{\stackrel{\tau}{\Longrightarrow}}\,z)
\Longleftrightarrow \bigl(\exists \mathcal{Z}\in
(z-\mathrm{bas})[\tau]:\,\mathbf{h}^{-1}[\mathcal{Z}]\subset
(E-\mathrm{\mathbf{fi}})[\mathcal{B}]\bigl).\efn

{\bf Remark\,10.4}. Consider the proof of (\ref{11.26}). We fix
$\mathcal{B}\in \beta_o[E]$ and $z\in \mathbf{H}.$ Since
$N_\tau^o(z) \in (z-\mathrm{bas})[\tau]$ (see (\ref{2.18}) and
definitions of Section 3), by (\ref{11.20}) \bfn\label{11.27}
(\mathbf{h}^1[\mathcal{B}]\,{\stackrel{\tau}{\Longrightarrow}}\,z)
\Longrightarrow \bigl(\exists \mathcal{Z}\in
(z-\mathrm{bas})[\tau]:\,\mathbf{h}^{-1}[\mathcal{Z}]\subset
(E-\mathrm{\mathbf{fi}})[\mathcal{B}]\bigl).\efn Let the corollary
of (\ref{11.27}) is valid. Fix $\mathfrak{Z}\in
(z-\mathrm{bas})[\tau]$ with the property
\bfn\label{11.28}\mathbf{h}^{-1}[\mathfrak{Z}]\subset
(E-\mathrm{\mathbf{fi}})[\mathcal{B}].\efn Let $\mathbb{G}\in
N_\tau^o(z).$ Then, by (\ref{2.18}),  for some $\mathbb{B}\in
\mathfrak{Z},$ the inclusion $\mathbb{B}\subset \mathbb{G}$ is
valid, where $\mathbf{h}^{-1}(\mathbb{B})\in
\mathbf{h}^{-1}[\mathfrak{Z}].$ By (\ref{11.28})
$\mathbf{h}^{-1}(\mathbb{B})\in (E-\mathrm{\mathbf{fi}})[\cb]$ and
$\mathbf{h}^{-1}(\mathbb{B})\subset \mathbf{h}^{-1}(\mathbb{G}).$
From (\ref{3.1}) and (\ref{3.3}), the inclusion
$\mathbf{h}^{-1}(\mathbb{G})\in (E-\mathrm{\mathbf{fi}})[\cb]$
follows. Since the choice of $\mathbb{G}$ was arbitrary, the
inclusion
$$\mathbf{h}^{-1}[N_\tau^o(z)]\subset (E-\mathrm{\mathbf{fi}})[\cb]$$
is established. By (\ref{11.20})
$\mathbf{h}^1[\cb]\,{\stackrel{\tau}{\Rightarrow}}\,z.$ So, we
obtain that $$\bigl(\exists \mathcal{Z}\in
(z-\mathrm{bas})[\tau]:\,\mathbf{h}^{-1}[\mathcal{Z}]\subset
(E-\mathbf{fi} )[\mathcal{B}]\bigl)\Longrightarrow
(\mathbf{h}^1[\cb]\,{\stackrel{\tau}{\Rightarrow}}\,z).$$ Using the
last implication and (\ref{11.27}), we obtain the required property
(\ref{11.26}).

Using (\ref{11.14}), we obtain the obvious corollary of
(\ref{11.26}): $\fo \cu\in \mathfrak{F}_\mathbf{u}[E]\ \ \fo z\in
\mathbf{H}$ \bfn\label{11.29}(\mathbf{h}^1[\cu \cap
\ca]\,{\stackrel{\tau}{\Rightarrow}}\,z) \Leftrightarrow
\bigl(\exists \mathcal{Z}\in (z-\mathrm{bas})[\tau]:\,
\mathbf{h}^{-1}[\mathcal{Z}]\subset (E-\mathrm{\mathbf{fi}})[\cu
\cap \ca]\bigl).\efn

{\bf Remark\,10.5}. Consider the proof of (\ref{11.29}) fixing
$\cu\in \mathfrak{F}_\mathbf{u}[E]$ and $z\in \mathbf{H}.$ Then, by
(\ref{11.14}) $\cu \cap \ca \in \mathbb{F}_o^*(\ca).$ In particular
(see Section 9), $\cu \cap \ca \in \beta_o[E].$ Now, (\ref{11.29})
follows from (\ref{11.26}).

\begin{con}\label{co11.1}{\TL} $\fo z\in \mathbf{H}\ \ \exists
\mathcal{Z}\in
(z-\mathrm{bas})[\tau]:\,\mathbf{h}^{-1}[\mathcal{Z}]\subset \ca.$
\end{con}

{\bf Remark\,10.6}. It is possible to consider
Condition~\ref{co11.1} as a weakened variant of the measurability of
$\mathbf{h}.$ The usual measurability of $\mathbf{h}$ is not natural
since $\ca$ is only algebra of sets.

Until the end of the present section, we suppose that
Condition~\ref{co11.1} is valid.

\begin{propos}\label{p11.3}{\TL} If $Condition~\ref{co11.1}$ is
fulfilled, then $(\tau-\mathbb{AS})[\mathcal{E}] =
(\mathrm{\mathbf{as}})[E;\mathbf{H};\tau;\mathbf{h};\mathcal{E}]\ \
\fo \mathcal{E}\in \cp^\prime(\ca).$
\end{propos}

\begin{proof} Let Condition~\ref{co11.1} be fulfilled. Fix
$\mathcal{E}\in \cp^\prime(\ca)$ and $z\in
(\mathrm{\mathbf{as}})[E;\mathbf{H};\tau;\mathbf{h};\mathcal{E}].$
Then, $z\in \mathbf{H}$ and, for some $\cu\in
\mathfrak{F}_\mathbf{u}[E |\,\mathcal{E}]$
\bfn\label{11.30}\mathbf{h}^1[\cu]\,{\stackrel{\tau}{\Longrightarrow}}\,z\efn
(see Proposition~\ref{p4.2}). Then, by (\ref{11.20}) and
(\ref{11.30}) we have the inclusion
$\mathbf{h}^{-1}[N_\tau^o(z)]\subset \cu,$ since $(E-\mathbf{fi}
)[\cu] = \cu$ by (\ref{3.1}). As a corollary,
$$\mathbf{h}^{-1}[N_\tau(z)]\subset \cu$$ (indeed, for $H\in
N_\tau(z),$ we can choose $G\in N_\tau^o (z)$ such that $G\subset
H;$ therefore, by (\ref{2.1}) $\mathbf{h}^{-1}(G) \in \cu,\,
\mathbf{h}^{-1}(G)\subset \mathbf{h}^{-1}(H),$ and by (\ref{3.1})
$\mathbf{h}^{-1}(H) \in \cu).$ By Condition~\ref{co11.1} there
exists $\widetilde{\mathcal{Z}}\in (z-\mathrm{bas})[\tau]$ such that
$\mathbf{h}^{-1}[\widetilde{\mathcal{Z}}]\subset \ca.$ In addition,
\bfn\label{11.31}\mathbf{h}^{-1}[\widetilde{\mathcal{Z}}]\subset
\mathbf{h}^{-1}[N_\tau(z)]\subset \cu;\efn therefore,
$\mathbf{h}^{-1}[\widetilde{\mathcal{Z}}]\subset \cu \cap \ca,$
where by (\ref{11.14}) $\cu \cap \ca \in \mathbb{F}_o^*(\ca).$ We
recall that $\cu \cap \ca \in \beta_o[E]$ (see Section 9) and
\bfn\label{11.32}\mathbf{h}^{-1}[\widetilde{\mathcal{Z}}]\subset \cu
\cap \ca\subset (E-\mathbf{fi} )[\cu \cap \ca].\efn Of course, by
(\ref{4.1}) $\mathbf{h}^1[\cu \cap \ca]\in \beta_o[\mathbf{H}].$ In
addition, by (\ref{11.32}) $$\exists \mathcal{Z}\in
(z-\mathrm{bas})[\tau]:\,\mathbf{h}^{-1}[\mathcal{Z}]\subset
(E-\mathbf{fi} )[\cu \cap \ca].$$ By (\ref{11.26}) $\mathbf{h}^1[\cu
\cap \ca]\,{\stackrel{\tau}{\Rightarrow}}\,z.$ Recall that
$\mathcal{E}\subset \cu.$ Since $\mathcal{E}\in \cp^\prime(\ca),$ we
obtain that $\mathcal{E}\subset \cu \cap \ca.$ Therefore (see
(\ref{10.21`})), $$\cu \cap \ca \in \mathbb{F}_o^*(\ca
|\,\mathcal{E}):\, \mathbf{h}^1[\cu \cap
\ca]\,{\stackrel{\tau}{\Longrightarrow}}\,z.$$ By
Proposition~\ref{p11.1} $z\in (\tau-\mathbb{AS})[\mathcal{E}
|\,\ca].$ Since the choice of $z$ was arbitrary, we have the
inclusion
\bfn\label{11.33}(\mathrm{\mathbf{as}})[E;\mathbf{H};\tau;\mathbf{h};\mathcal{E}]
\subset (\tau-\mathbb{AS})[\mathcal{E}|\,\ca].\efn Using
(\ref{11.33}) and Proposition~\ref{p11.2}, we obtain the equality
$$\hspace{3.5cm}(\tau-\mathbb{AS})[\mathcal{E}|\,\ca] = (\mathrm{\mathbf{as}})
[E;\mathbf{H};\tau;\mathbf{h};\mathcal{E}]. \hspace{3.6cm} $$
\end{proof}

So, we can use (see Proposition~\ref{p11.1} and
Condition~\ref{co11.1}) ultrafilters of the space $(E,\ca)$ as
nonsequential approximate solutions in the case, when a nonempty
subfamily of $\ca$ is used as  the constraint of asymptotic
character. This property is very useful in the cases of spaces
$(E,\ca)$ for which the set $\mathbb{F}_o^*(\ca)$ is realized
effectively. In addition, for a semialgebra $\cl\in \Pi[E]$ with the
property $\ca = a_E^o(\cl)$ (see Section 2), we consider the passage
$$\mathbb{F}_o^*(\cl)\longrightarrow \mathbb{F}_o^*(\ca)$$ as an
unessential transformation (see \cite[$\S\,7.6$]{32} and
\cite[$\S\,2.4$]{39}; here it is appropriate to use the natural
connection of ultrafilters and finitely additive (0,1)-measures).
Then, after unessential transformations, the examples of
\cite[$\S\,7.6$]{32} can be used in our scheme sufficiently
constructively.

\section{Ultrasolutions}
\setcounter{equation}{0}

First, we recall some statements of \cite{40}. In addition, we fix a
nonempty set $E$ and a TS $(\mathbf{H},\tau),$ where $\mathbf{H}\neq
\e.$ We consider the nonempty set $\mathfrak{F}_\mathbf{u}[E].$
Suppose that $\mathbf{h}\in \mathbf{H}^E.$ Then, we suppose that
\bfn\label{12.1}(\mathbf{h}-\mathrm{LIM})[\cu |\,\tau]\,\df\,\{z\in
\mathbf{H}
|\,\mathbf{h}^1[\cu]\,{\stackrel{\tau}{\Longrightarrow}}\,z\}\ \ \fo
\cu\in \mathfrak{F}_\mathbf{u}[E].\efn So, we introduce the limit
sets corresponding to ultrafilters of $E.$ By analogy with
Proposition\,5.4 of \cite{40} the following statement is
established.

\begin{propos}\label{p12.1}{\TL} If $\tau\in (c-\mathrm{top})[\mathbf{H}],$ then
$(\mathbf{h}-\mathrm{LIM})[\cu |\,\tau]\in
\cp^\prime\Bigl(\mathrm{cl}\bigl(\mathbf{h}^1(E),\tau\bigl)\Bigl)\ \
\fo \cu \in \mathfrak{F}_\mathbf{u}[E].$\end{propos}

\begin{proof} Fix $\cu\in \mathfrak{F}_\mathbf{u}[E].$ Then
$\mathbf{h}^1[\cu]\in \beta_o[\mathbf{H}]$ and by (\ref{4.2})
\bfn\label{12.2}\mathcal{K}\,\df\,(\mathbf{H}-\mathrm{\mathbf{fi}})\bigl[\mathbf{h}^1[\cu]\bigl]\in
\mathfrak{F}_\mathbf{u}[\mathbf{H}]\efn (recall that
$(E-\mathrm{\mathbf{fi}})[\cu] = \cu).$ Since $(\mathbf{H},\tau)$ is
a compact TS, there exists $y\in \mathbf{H}$ such that
$\mathcal{K}{\stackrel{\tau}{\Rightarrow}}\,y;$ see
\cite[ch.\,I]{32}. Then, by (\ref{3.5}) $N_\tau(y) \subset
\mathcal{K}$ or \bfn\label{12.3}N_\tau(y)\subset
(\mathbf{H}-\mathrm{\mathbf{fi}})\bigl[\mathbf{h}^1[\cu]\bigl]\efn
(see (\ref{12.2})). By (\ref{3.4}) and (\ref{12.3})
$\mathbf{h}^1[\cu]{\stackrel{\tau}{\Rightarrow}}\,y.$ Then, by
(\ref{12.1}) $y\in (\mathbf{h}-\mathrm{LIM})[\cu |\,\tau].$ So,
\bfn\label{12.3`}(\mathbf{h}-\mathrm{LIM})[\cu |\,\tau] \neq \e.\efn
Let $z\in (\mathbf{h}-\mathrm{LIM})[\cu |\,\tau].$ Then $z\in
\mathbf{H}$ and
$\mathbf{h}^1[\cu]{\stackrel{\tau}{\Rightarrow}}\,z.$ By (\ref{3.4})
and (\ref{12.2}) $N_\tau(z) \subset \mathcal{K}.$ In addition, by
(\ref{12.2}) $\mathbf{h}^1[\cu]\subset \mathcal{K}.$ Then, by
(\ref{3.1})
$$A \cap B \neq \e\ \ \fo A\in N_\tau (z)\ \ \fo B\in
\mathbf{h}^1[\cu].$$ By (\ref{2.1}) we obtain that
$$\mathbf{h}^1(U) \cap H \neq \e\ \ \fo U\in \cu\ \ \fo H\in
N_\tau (z).$$ Since $\cu\neq \e$ and $\mathbf{h}^1(U) \subset
\mathbf{h}^1(E)$ for $U\in \cu,$ we have the property:
$$\mathbf{h}^1(E) \cap H \neq \e\ \ \fo H\in N_\tau(z).$$ So, $z\in
\mathrm{cl}\bigl(\mathbf{h}^1(E),\tau\bigl).$ The inclusion
$$(\mathbf{h}-\mathrm{LIM})[\cu |\,\tau]\subset
\mathrm{cl}\bigl(\mathbf{h}^1(E),\tau\bigl)$$ is established. Using
(\ref{12.3`}), we obtain that $(\mathbf{h}-\mathrm{LIM})[\cu\,
|\,\tau]
\in\cp^\prime\Bigl(\mathrm{cl}\bigl(\mathbf{h}^1(E),\tau\bigl)\Bigl).$\end{proof}

We note the following obvious property too: if $\tau\in (\mathrm{top})_o[\mathbf{H}],\,
\cf\in \mathfrak{F}[\mathbf{H}],\,y_1\in \mathbf{H},$ and $y_2\in \mathbf{H},$ then
\bfn\label{12.4}\bigl((\cf {\stackrel{\tau}{\Longrightarrow}}\,y_1)\,\&\,(\cf
{\stackrel{\tau}{\Longrightarrow}}\,y_2)\bigl) \Longrightarrow (y_1 = y_2).\efn

{\bf Remark\,11.1}. Let the premise of (\ref{12.4}) be fulfilled.
Then, by (\ref{3.5}) \bfn\label{12.5}\bigl(N_\tau(y_1)\subset
\cf\bigl)\,\&\,\bigl(N_\tau(y_2)\subset \cf\bigl).\efn Then, $y_1 =
y_2.$ Indeed, suppose the contrary: $y_1 \neq y_2.$ Then, by
(\ref{7.0}), for some $H_1\in N_\tau(y_1)$ and $H_2\in N_\tau(y_2),$
the equality $H_1 \cap H_2 = \e$ is valid. But, by (\ref{12.5}) $H_1
\in \cf$ and $H_2\in \cf.$ Then, by (\ref{3.1}) $H_1 \cap H_2 \neq
\e.$ The obtained contradiction means that $y_1 \neq y_2$ is
impossible. So, $y_1 = y_2.$

\begin{propos}\label{p12.2}{\TL} If $\tau\in (c-\mathrm{top})_o[\mathbf{H}]$ and
$\cu\in \mathfrak{F}_\mathbf{u}[E],$ then $$\exists ! \tilde{z}\in
\mathbf{H}:\,(\mathbf{h}-\mathrm{LIM})[\cu |\,\tau] = \{\tilde{z}\}.$$
\end{propos}
\begin{proof}
The corresponding  proof is the obvious combination of (\ref{12.1}), (\ref{12.4}), and
Proposition~\ref{p12.1}. Indeed, by Proposition~\ref{p12.1}
$(\mathbf{h}-\mathrm{LIM})[\cu |\,\tau]\neq \e$ and $(\mathbf{h}-\mathrm{LIM})[\cu
|\,\tau]\subset \mathbf{H}.$ Let $y\in (\mathbf{h}-\mathrm{LIM})[\cu |\,\tau].$ Then,
$y\in \mathbf{H}$ and \bfn\label{12.6}\mathbf{h}^1[\cu]
{\stackrel{\tau}{\Longrightarrow}}\,y.\efn Let $z\in (\mathbf{h}-\mathrm{LIM})[\cu
|\,\tau].$ Then, $z\in \mathbf{H}$ and \bfn\label{12.7}\mathbf{h}^1[\cu]
{\stackrel{\tau}{\Longrightarrow}}\,z.\efn For
$\mathcal{H}\,\df\,(\mathbf{H}-\mathrm{\mathbf{fi}})\bigl[\mathbf{h}^1[\cu]\bigl]\in
\mathfrak{F}_\mathbf{u}[\mathbf{H}]$ by (\ref{12.6}) and (\ref{12.7})
$$\bigl(N_\tau(y)\subset \mathcal{H}\bigl)\,\&\,\bigl(N_\tau(z)\subset
\mathcal{H}\bigl).$$ So, by (\ref{3.5})
$\mathcal{H}{\stackrel{\tau}{\Rightarrow}}\,y$ and
$\mathcal{H}{\stackrel{\tau}{\Rightarrow}}\,z.$ From (\ref{12.4})
the equality $y = z$ is valid. Then, $z\in \{y\}.$ The inclusion
\bfn\label{12.8}(\mathbf{h}-\mathrm{LIM})[\cu |\,\tau]\subset
\{y\}\efn is established. But, by the choice of $y$ we have the
inclusion $\{y\}\subset (\mathbf{h}-\mathrm{LIM})[\cu |\,\tau].$
Using (\ref{12.8}), we obtain that
$$(\mathbf{h}-\mathrm{LIM})[\cu |\,\tau] = \{y\}.$$ The uniqueness of $y$ is obvious.
\end{proof}

From Proposition~\ref{p12.2} the natural corollary follows: if
$\tau\in (c-\mathrm{top})_o[\mathbf{H}],$ then
\bfn\label{12.9}\exists ! g\in
\mathbf{H}^{\mathfrak{F}_\mathbf{u}[E]}:\,
(\mathbf{h}-\mathrm{LIM})[\cu |\,\tau] = \{g(\cu)\}\ \ \fo \cu\in
\mathfrak{F}_\mathbf{u}[E].\efn In the following, we postulate that
\bfn\label{12.10}\tau\in (c-\mathrm{top})_o[\mathbf{H}].\efn  Then
(see (\ref{12.9}) and (\ref{12.10})), we suppose  that
\bfn\label{12.11}\mathfrak{H}[\tau]\in
\mathbf{H}^{\mathfrak{F}_\mathbf{u}[E]}\efn is defined by the
following rule: if $\cu\in \mathfrak{F}_\mathbf{u}[E],$ then
$\mathfrak{H}[\tau](\cu)\in \mathbf{H}$ has the property:
\bfn\label{12.12}(\mathbf{h}-\mathrm{LIM})[\cu |\,\tau] =
\{\mathfrak{H}[\tau](\cu)\}.\efn We note that by (\ref{12.12}) and
Proposition~\ref{p12.1} \bfn\label{12.13}\mathfrak{H}[\tau](\cu)\in
\mathrm{cl}\bigl(\mathbf{h}^1(E),\tau\bigl)\ \ \fo \cu\in
\mathfrak{F}_\mathbf{u}[E]\efn So, by (\ref{12.11}) and
(\ref{12.13}) $\mathfrak{H}[\tau]:\,\mathfrak{F}_\mathbf{u}[E]
\rightarrow \mathrm{cl}\bigl(\mathbf{h}^1(E),\tau\bigl).$ In this
connection, we note the following typical situation: under condition
(\ref{12.10}), $\mathbf{h}^1(E) \neq \mathbf{H}$ and
$\mathbf{h}^1(E) \notin (\tau-\mathrm{comp})[\mathbf{H}].$ Of
course, by (\ref{12.10})
$$\mathrm{cl}\bigl(\mathbf{h}^1(E),\tau\bigl)\in
(\tau-\mathrm{comp})[\mathbf{H}].$$ Indeed, any closed set in a
compact TS is compact too. Recall that
\bfn\label{12.14}\mathfrak{H}[\tau] \in
\mathrm{cl}\bigl(\mathbf{h}^1(E),\tau\bigl)^{\mathfrak{F}_\mathbf{u}[E]}.\efn
Returning to (\ref{12.1}) and (\ref{12.12}) we note that
\bfn\label{12.15}\mathbf{h}^1[\cu]\,{\stackrel{\tau}{\Longrightarrow}}\,\mathfrak{H}[\tau](\cu)\
\ \fo \cu\in \mathfrak{F}_\mathbf{u}[E].\efn With the employment of
(\ref{3.8}), we introduce the natural immersion of $E$ in
$\mathfrak{F}_\mathbf{u}[E]$ supposing that
\bfn\label{12.16}(E-\mathrm{ult})[\cdot]\,\df\,\bigl((E-\mathrm{ult})[x]\bigl)_{x\in
E}\in \mathfrak{F}_\mathbf{u}[E]^E.\efn In connection with
(\ref{12.16}), we note the following obvious
equality:\bfn\label{12.17}\mathbf{h} = \mathfrak{H}[\tau] \circ
(E-\mathrm{ult})[\cdot].\efn

{\bf  Remark\,11.2}. Consider the proof of (\ref{12.17}). Fix $x\in
E.$ By (\ref{3.8}) we obtain that \bfn\label{12.18}\mathbf{h}(x) \in
\mathbf{h}^1(S)\ \ \fo S\in (E-\mathrm{ult})[x].\efn So, by
(\ref{2.1}) and (\ref{12.18}) we obtain the following
property:\bfn\label{12.19}\mathbf{h}(x)\in T\ \ \fo T\in
\mathbf{h}^1\bigl[(E-\mathrm{ult})[x]\bigl].\efn In addition, by
(\ref{3.4}), (\ref{3.8}), and (\ref{12.15})
$$N_\tau\Bigl(\mathfrak{H}[\tau]\bigl((E-\mathrm{ult})[x]\bigl)\Bigl)\subset
(\mathbb{H}-\mathrm{\mathbf{fi}})\Bigl[\mathbf{h}^1\bigl[(E-\mathrm{ult})[x]\bigl]\Bigl].$$
So, $\fo \Gamma\in
N_\tau\Bigl(\mathfrak{H}[\tau]\bigl((E-\mathrm{ult})[x]\bigl)\Bigl)\
\ \exists T\in \mathbf{h}^1\bigl[(E-\mathrm{ult})[x]\bigl]:\,
T\subset \Gamma.$ Then, by (\ref{12.19}) $\mathbf{h}(x) \in \Gamma\
\ \fo \Gamma \in
N_\tau\Bigl(\mathfrak{H}[\tau]\bigl((E-\mathrm{ult})[x]\bigl)\Bigl).$
Using the separability of $\tau$ (\ref{12.10}), we obtain the
equality chain $$\mathbf{h}(x) =
\mathfrak{H}[\tau]\bigl((E-\mathrm{ult})[x]\bigl) =
\bigl(\mathfrak{H}[\tau] \circ (E-\mathrm{ult})[\cdot]\bigl)(x).$$
Since the choice of $x$ was arbitrary, we obtain that (\ref{12.17})
is fulfilled.

\begin{propos}\label{p12.3}{\TL} If $\cu\in \mathfrak{F}_\mathbf{u}[E],$ then the following
equality is valid:
$$\bigcap\limits_{A\in\,\cu}\mathrm{cl}\bigl(\mathbf{h}^1(A),\tau\bigl) =
\{\mathfrak{H}[\tau](\cu)\}.$$
\end{propos}

\begin{proof} Let $u\,\df\,\mathfrak{H}[\tau](\cu).$ Then $u\in \mathbf{H}$ and by
(\ref{12.15}) $$\mathbf{h}^1[\cu]\,{\stackrel{\tau}{\Longrightarrow}}\,u.$$ Then, by
(\ref{3.4}) $N_\tau(u) \subset
(\mathbf{H}-\mathrm{\mathbf{fi}})\bigl[\mathbf{h}^1[\cu]\bigl].$ Therefore, for any $T\in
N_\tau(u),$ there exists $U\in \cu$ such that $\mathbf{h}^1(U) \subset T$ (see
(\ref{3.3}) and (\ref{4.1})).

Let $A_*\in \cu.$ Then, $\mathbf{h}^1(A_*)\in \cp(\mathbf{H}).$ If
$S\in N_\tau(u),$ then, for some $U_S\in \cu,$ the inclusion
$\mathbf{h}^1(U_S)\subset S$ is valid; moreover, $A_* \cap U_S \neq
\e$ and \bfn\label{12.19``}\mathbf{h}^1(A_* \cap U_S) \subset
\mathbf{h}^1(A_*) \cap \mathbf{h}^1(U_S) \subset \mathbf{h}^1(A_*)
\cap S,\efn where $\mathbf{h}^1(A_* \cap U_S)\neq \e.$ So,
$\mathbf{h}^1(A_*) \cap S \neq \e.$ Since the choice of $S$ was
arbitrary, we obtain that $$\mathbf{h}^1(A_*) \cap H \neq \e\ \ \fo
H\in N_\tau(u).$$ Therefore, $u\in
\mathrm{cl}\bigl(\mathbf{h}^1(A_*),\tau\bigl).$ Since the choice of
$A_*$ was arbitrary too, we have the inclusion $u\in
\bigcap\limits_{A\in
\cu}\mathrm{cl}\bigl(\mathbf{h}^1(A),\tau\bigl).$ Therefore,
$$\{u\}\subset \bigcap_{A\in
\cu}\mathrm{cl}\bigl(\mathbf{h}^1(A),\tau\bigl).$$ Choose arbitrary
$q \in \bigcap\limits_{A\in
\cu}\mathrm{cl}\bigl(\mathbf{h}^1(A),\tau\bigl).$ Then $q \in
\mathbf{H}$ and \bfn\label{12.20}\widetilde{W} \cap \mathbf{h}^1(A)
\neq \e\ \ \fo A\in \cu\ \ \fo\widetilde{W}\in N_\tau(q).\efn Then,
for $W\in N_\tau(q),$ we obtain the property $W \cap T \neq \e\ \
\fo T\in \mathbf{h}^1[\cu].$ Using (\ref{3.3}), we have the
following statement: $$W \cap M \neq \e\ \ \fo M\in
(\mathbf{H}-\mathrm{\mathbf{fi}})\bigl[\mathbf{h}^1[\cu]\bigl].$$
Therefore, $W\in
\bigl((\mathbf{H}-\mathrm{\mathbf{fi}})\bigl[\mathbf{h}^1[\cu]\bigl]-\mathrm{set}\bigl)[\mathbf{H}]$
(see Section 5), where
\bfn\label{12.21}(\mathbf{H}-\mathrm{\mathbf{fi}})\bigl[\mathbf{h}^1[\cu]\bigl]\in
\mathfrak{F}_\mathbf{u}[\mathbf{H}]\efn (we use (\ref{4.2})). In
addition, $\cp(\mathbf{H})\in \pi[\mathbf{H}].$ Using (\ref{7.5}),
(\ref{12.21}), and statements of Section 5, we obtain that
$$\bigl((\mathbf{H}-\mathrm{\mathbf{fi}})\bigl[\mathbf{h}^1[\cu]\bigl]-\mathrm{set}\bigl)
[\mathbf{H}] =
\bigl((\mathbf{H}-\mathrm{\mathbf{fi}})\bigl[\mathbf{h}^1[\cu]\bigl]-\mathrm{set}\bigl)
[\mathbf{H}] \cap \cp(\mathbf{H}) =
(\mathbf{H}-\mathrm{\mathbf{fi}})\bigl[\mathbf{h}^1[\cu]\bigl].$$ Therefore, $W \in
(\mathbf{H}-\mathrm{\mathbf{fi}})\bigl[\mathbf{h}^1[\cu]\bigl].$ Since the choice of $W$
was arbitrary, we obtain that $$N_\tau(q) \subset
(\mathbf{H}-\mathrm{\mathbf{fi}})\bigl[\mathbf{h}^1[\cu]\bigl].$$ So, by (\ref{3.4})
$\mathbf{h}^1[\cu]\,{\stackrel{\tau}{\Rightarrow}}\,q.$ Then, we have the following
properties:
$$(\mathbf{h}^1[\cu]\,{\stackrel{\tau}{\Longrightarrow}}\,u)\,\&\,
(\mathbf{h}^1[\cu]\,{\stackrel{\tau}{\Longrightarrow}}\,q).$$ By
(\ref{12.4}) $u = q.$ Then $q\in \{u\}.$ Since the choice of $q$ was
arbitrary, we obtain that
$$\bigcap\limits_{A\in\,\cu}\mathrm{cl}\bigl(\mathbf{h}^1(A),\tau\bigl)\subset \{u\}.$$
The opposite inclusion was established previously. Therefore,
$\{u\}$ and the intersection of all sets
$\mathrm{cl}\bigl(\mathbf{h}^1(A),\tau\bigl),\, A\in \cu,$
coincide.\end{proof}

From (\ref{4.15}) and Proposition~\ref{p12.3} we obtain that
\bfn\label{12.22}(\mathrm{\mathbf{as}})[E;\mathbf{H};\tau;\mathbf{h};\cu]
= \{\mathfrak{H}[\tau](\cu)\}\ \ \fo \cu\in
\mathfrak{F}_\mathbf{u}[E].\efn So, ultrafilters of $E$ realize very
perfect constraints of asymptotic character.

\section{Ultrafilters of measurable space with algebra of sets}
\setcounter{equation}{0}

In this section, we fix a nonempty set $E,$ TS
$(\mathbf{H},\tau),\,\mathbf{H}\neq \e,$ and $\mathbf{h}\in
\mathbf{H}^E.$ Moreover, we fix $\ca\in (\mathrm{alg})[E].$ Finally,
we suppose that Condition~\ref{co11.1} is fullfiled. Then, we have
the statement of Proposition~\ref{p11.3} and other statements of
Section 10. We suppose that (\ref{12.10}) is valid also. So, we have
the mapping (\ref{12.11}). In addition, we have the natural
uniqueness of the filter limit: (\ref{12.4}) is fulfilled. Now, we
supplement (\ref{12.4}). Namely, $\fo \cb\in \beta_o[\mathbf{H}]\ \
\fo y_1\in \mathbf{H}\ \ \fo y_2\in \mathbf{H}$
\bfn\label{13.1}\bigl((\mathcal{B}\,{\stackrel{\tau}{\Longrightarrow}}\,y_1
)\,\&\,(\mathcal{B}\,{\stackrel{\tau}{\Longrightarrow}}\,y_2)\bigl)
\Longrightarrow (y_1 = y_2).\efn

{\bf Remark\,12.1}. For the proof of (\ref{13.1}), we fix $\cb\in
\beta_o[\mathbf{H}],\,y_1\in \mathbf{H},$ and $y_2\in \mathbf{H}.$
Let the premise statement of (\ref{13.1}) is valid: $\cb$ converges
to $y_1$ and $y_2.$ Then, for
$\cf\,\df\,(\mathbf{H}-\mathrm{\mathbf{fi}})[\cb]\in
\mathfrak{F}_\mathbf{u}[\mathbf{H}]$ (see (\ref{3.3})), the
inclusions $$\bigl(N_\tau(y_1)\subset
\cf\bigl)\,\&\,\bigl(N_\tau(y_2)\subset \cf\bigl)$$ are fulfilled.
Therefore, by (\ref{3.5}) the following two properties are valid:
$$(\cf\,{\stackrel{\tau}{\Longrightarrow}}\,y_1)\,\&\,(\cf\,
{\stackrel{\tau}{\Longrightarrow}}\,y_2).$$ By (\ref{12.4}) $y_1 =
y_2.$ So, (\ref{13.1}) is established.

We recall (\ref{11.22}) and (\ref{11.23}): if $\ \cf\in
\mathbb{F}^*(\ca),$ then $\ \cf\in \beta_o[E]$ and
$\mathbf{h}^1[\cf]\in \beta_o[\mathbf{H}]$ (see (\ref{4.1})). We use
(\ref{11.14}).

\begin{propos}\label{p13.1}{\TL} $\mathbf{h}^1[\cu \cap
\ca]\,{\stackrel{\tau}{\Longrightarrow}}\,\mathfrak{H}[\tau](\cu)\ \
\fo \cu\in \mathfrak{F}_\mathbf{u}[E].$ \end{propos}

\begin{proof} Let $\cu\in \mathfrak{F}_\mathbf{u}[E]$ and $z\,\df\,
\mathfrak{H}[\tau](\cu).$ Then $z\in \mathbf{H}$ and by
Condition~\ref{co11.1}, for some $\mathcal{Z}\in
(z-\mathrm{bas})[\tau]$, the inclusion
$\mathbf{h}^{-1}[\mathcal{Z}]\subset \ca$ is fulfilled. In addition,
by (\ref{12.15})
$\mathbf{h}^1[\cu]\,{\stackrel{\tau}{\Rightarrow}}\,z.$ Since
$\cu\in \beta_o[E]$ and $\cu = (E-\mathrm{\mathbf{fi}})[\cu],$ then
\bfn\label{13.2}\mathbf{h}^{-1}[N_\tau^o(z)]\subset \cu.\efn From
(\ref{13.2}) the inclusion $\mathbf{h}^{-1}[N_\tau(z)]\subset \cu$
follows (namely, for any $S\in \mathbf{h}^{-1}[N_\tau(z)],$ there
exists $T\in \mathbf{h}^{-1}[N_\tau^o(z)]$ such that $T\subset S;$
then, $T\in \cu$ by (\ref{13.2}) and $S\in \cu$ by axioms of a
filter). In addition, $\mathcal{Z}\subset N_\tau(z).$ Then,
$$\mathbf{h}^{-1}[\mathcal{Z}]\subset
\mathbf{h}^{-1}[N_\tau(z)]\subset \cu$$ and (by the choice of $\
\mathcal{Z})$ $\ \mathbf{h}^{-1}[\mathcal{Z}]\subset \cu \cap \ca
\subset (E-\mathrm{\mathbf{fi}})[\cu \cap \ca].$ \hspace{0.5cm}Since
$\mathcal{Z}\in (z-\mathrm{bas})[\tau],$ by (\ref{11.26})
$$\mathbf{h}^1[\cu \cap \ca]\,{\stackrel{\tau}{\Longrightarrow}}\,z.$$ By definition of
$z\ \ \mathbf{h}^1[\cu \cap
\ca]\,{\stackrel{\tau}{\Longrightarrow}}\,\mathfrak{H}[\tau](\cu).$
\end{proof}

\begin{propos}\label{p13.2}{\TL} $\fo \cf\in \mathbb{F}_o^*(\ca)\ \ \exists ! z\in
\mathbf{H}:\,\mathbf{h}^1[\cf]\,{\stackrel{\tau}{\Longrightarrow}}\,z.$
\end{propos}

\begin{proof} Fix $\cf\in \mathbb{F}_o^*(\ca).$ Using (\ref{11.17}), we choose
$\cu\in\mathfrak{F}_ \mathbf{u}[E]$ such that \bfn\label{13.3}\cf = \cu \cap \ca.\efn
Then, by (\ref{12.11}) $\mathfrak{H}[\tau](\cu)\in \mathbf{H}$ and by (\ref{13.3}) and
Proposition~\ref{p13.1}
\bfn\label{13.4}\mathbf{h}^1[\cf]\,{\stackrel{\tau}{\Longrightarrow}}\,
\mathfrak{H}[\tau](\cu).\efn In addition, $\cf\in \beta_o[E]$ and $\mathbf{h}^1[\cf]\in
\beta_o[\mathbf{H}].$ Therefore, by (\ref{13.1}) and (\ref{13.4}) $\fo y\in \mathbf{H}$
$$\hspace{3.6cm}(\mathbf{h}^1[\cf]\,{\stackrel{\tau}{\Longrightarrow}}\,y)
\Longrightarrow \bigl(y = \mathfrak{H}[\tau](\cu)\bigl).
\hspace{3.7cm}$$\end{proof}

From Proposition~\ref{p13.2}, the obvious corollary follows; namely
$\exists ! g\in~\mathbf{H}^{\mathbb{F}_o^*(\ca)}:$
$$\mathbf{h}^1[\cf]\,{\stackrel{\tau}{\Longrightarrow}}\,g(\cf)\ \ \fo \cf\in
\mathbb{F}_o^*(\ca).$$  Now, we suppose  that the mapping
\bfn\label{13.5}\mathfrak{H}_\ca[\tau]:\,\mathbb{F}_o^*(\ca)
\longrightarrow \mathbf{H}\efn is defined by the following rule: if
$\cf\in \mathbb{F}_o^*(\ca),$ then
\bfn\label{13.6}\mathbf{h}^1[\cf]\,{\stackrel{\tau}{\Longrightarrow}}\,
\mathfrak{H}_\ca[\tau](\cf).\efn From (\ref{11.14}) and
(\ref{13.5}), the obvious property follows; namely,
$\mathfrak{H}_\ca[\tau](\cu~\cap~\ca)\in \mathbf{H}\ \ \fo \cu\in
\mathfrak{F}_\mathbf{u}[E].$

\begin{propos}\label{p13.3}{\TL} $\mathfrak{H}_\ca[\tau](\cu \cap \ca) =
\mathfrak{H}[\tau](\cu)\ \ \fo \cu\in \mathfrak{F}_\mathbf{u}[E].$
\end{propos}

\begin{proof} Fix $\cu\in \mathfrak{F}_\mathbf{u}[E].$  Then, by
(\ref{12.11}) $\mathfrak{H}[\tau](\cu)\in \mathbf{H}.$ By
(\ref{11.14}) we obtain that $\cu \cap \ca \in \mathbb{F}_o^*(\ca).$
In particular, $\cu \cap \ca\in \beta_o[E]$ and by (\ref{4.1})
$\mathbf{h}^1[\cu \cap \ca]\in \beta_o[\mathbf{H}].$ From
Proposition~\ref{p13.1}, we have the following convergence
\bfn\label{13.7}\mathbf{h}^1[\cu \cap
\ca]\,{\stackrel{\tau}{\Longrightarrow}}\,\mathfrak{H}[\tau](\cu).\efn
Using Proposition~\ref{p13.2}, (\ref{13.1}), (\ref{13.6}), and
(\ref{13.7}), we obtain that $ \mathfrak{H}_\ca[\tau](\cu~ \cap~
\ca) = \mathfrak{H}[\tau](\cu).$ \end{proof}

\begin{propos}\label{p13.4}{\TL} If $\cu\in \mathbb{F}_o^*(\ca)$ and $U\in \cu,$ then
 $\mathfrak{H}_\ca[\tau](\cu)\in \mathrm{cl}\bigl(\mathbf{h}^1(U),\tau\bigl).$
 \end{propos}

\begin{proof} Using (\ref{11.17}), we choose $\mathcal{V}\in \mathfrak{F}_\mathbf{u}[E]$
such that $\cu = \mathcal{V}\cap \ca.$ Then, by Proposition~\ref{p12.3} we have the
inclusion \bfn\label{13.8}\mathfrak{H}[\tau](\mathcal{V})\in
\mathrm{cl}\bigl(\mathbf{h}^1(U),\tau\bigl)\efn (we use the obvious inclusion $U\in
\mathcal{V}$ realized by the choice of $U).$ By Proposition~\ref{p13.3}
$$\mathfrak{H}_\ca[\tau](\cu) = \mathfrak{H}_\ca[\tau](\mathcal{V} \cap \ca) =
\mathfrak{H}[\tau](\mathcal{V}).$$ From (\ref{13.8}), the inclusion
$\mathfrak{H}_\ca[\tau](\cu) \in
\mathrm{cl}\bigl(\mathbf{h}^1(U),\tau\bigl)$ follows.\end{proof}

We note that (see \cite{34, 35, 36}) by (\ref{12.10}) the space
$(\mathbf{H},\tau)$ is regular: if $x\in \mathbf{H},$ then
\bfn\label{13.9}\exists \mathfrak{X}\in
(x-\mathrm{bas})[\tau]:\,\mathfrak{X}\subset
\mathbf{C}_\mathbf{H}(\tau).\efn

\begin{propos}\label{p13.5}{\TL} The mapping $(\ref{13.5})$ is continuous:
\bfn\label{13.9`}\mathfrak{H}_\ca[\tau]\in
C\bigl(\mathbb{F}_o^*(\ca),\mathbf{T}_\ca^*[E],\mathbf{H},\tau\bigl).\efn
\end{propos}

\begin{proof} Fix $\cu\in \mathbb{F}_o^*(\ca).$ Then, by (\ref{13.5})
$z\,\df\,\mathfrak{H}_\ca[\tau](\cu)\in \mathbf{H}.$ In addition, by
(\ref{13.6})
\bfn\label{13.10}\mathbf{h}^1[\cu]\,{\stackrel{\tau}{\Longrightarrow}}\,z.\efn
Of course, $\cu\in \beta_o[E]$ and $\mathbf{h}^1[\cu]\in
\beta_o[\mathbf{H}]$ (see (\ref{4.1})). As a corollary, by
(\ref{3.3})
\bfn\label{13.11}\mathcal{H}\,\df\,(\mathbf{H}-\mathrm{\mathbf{fi}})\bigl[\mathbf{h}^1[\cu]\bigl]\in
\mathfrak{F}[\mathbf{H}].\efn From (\ref{3.4}), (\ref{13.10}), and
(\ref{13.11}), we obtain the following inclusion:
\bfn\label{13.12}N_\tau(z) \subset \mathcal{H}.\efn From
(\ref{2.1}), (\ref{3.3}), and (\ref{13.12}), we obtain that
\bfn\label{13.13}\fo S\in N_\tau(z)\ \ \exists U\in
\cu:\,\mathbf{h}^1(U)\subset S.\efn Fix $\mathbf{N}\in N_\tau(z).$
Using (\ref{13.9}), we choose $\mathcal{Z}\in (z-\mathrm{bas})[\tau]
$ such that $\mathcal{Z}\subset \mathbf{C}_\mathbf{H}(\tau).$ Then,
by (\ref{2.18}), for some $\mathbf{F}\in \mathcal{Z},$ the inclusion
\bfn\label{13.14}\mathbf{F}\subset \mathbf{N}\efn is valid.
Therefore, $\mathbf{F}\in \mathbf{C}_\mathbf{H}(\tau).$ Of course,
$\mathbf{F}\in N_\tau(z).$ Therefore, by (\ref{13.13}), for some
$\mathbf{U}\in \cu$ \bfn\label{13.15}\mathbf{h}^1(\mathbf{U})\subset
\mathbf{F}.\efn In addition, $\Phi_\ca(\mathbf{U})\in
(\mathbb{UF})[E;\ca]$ (see (\ref{5.3})). By (\ref{7.3})
$\Phi_\ca(\mathbf{U})\in \mathbf{T}_\ca^*[E].$ In addition, by
(\ref{5.2}) $\cu\in \Phi_\ca(\mathbf{U}).$ Therefore,
\bfn\label{13.16}\Phi_\ca(\mathbf{U})\in
N_{\mathbf{T}_\ca^*[E]}^o(\cu).\efn Choose arbitrary ultrafilter
$\mathcal{V}\in \Phi_\ca(\mathbf{U}).$ Then, $\mathcal{V}\in
\mathbb{F}_o^*(\ca)$ and $\mathbf{U}\in \mathcal{V};$ see
(\ref{5.2}). By Proposition~\ref{p13.4}
\bfn\label{13.17}\mathfrak{H}_\ca[\tau](\mathcal{V})\in
\mathrm{cl}\bigl(\mathbf{h}^1(\mathbf{U}),\tau\bigl).\efn  By the
closedness of $\mathbf{F}$ and (\ref{13.15})
$\mathrm{cl}\bigl(\mathbf{h}^1(\mathbf{U}),\tau\bigl)\subset
\mathbf{F}.$ So, from (\ref{13.17}), we have the inclusion
$$\mathfrak{H}_\ca[\tau](\mathcal{V})\in \mathbf{F}.$$ Using
(\ref{13.14}), we obtain that
$\mathfrak{H}_\ca[\tau](\mathcal{V})\in \mathbf{N}.$ Since the
choice of $\mathcal{V}$ was arbitrary, the inclusion
\bfn\label{13.18}\mathfrak{H}_\ca[\tau]^1\bigl(\Phi_\ca(\mathbf{U})\bigl)
\subset \mathbf{N}\efn is established. Since the choice of
$\mathbf{N}$ was arbitrary too, from (\ref{13.16}), we obtain that
$$\fo S\in N_\tau\bigl(\mathfrak{H}_\ca[\tau](\cu)\bigl)\ \ \exists
T\in N_{\mathbf{T}_\ca^*[E]}(\cu):\,\mathfrak{H}_\ca[\tau]^1(T)
\subset S.$$ So, the mapping $\mathfrak{H}_\ca[\tau]$ is continuous
at the point $\cu.$ Since the choice of $\cu$ was arbitrary, the
required inclusion (\ref{13.9`}) is established (see
\cite[(2.5.4)]{39}) \end{proof}

In connection with Proposition~\ref{p13.5}, we recall
Proposition~\ref{p10.2} and known statement about the possibility of
an extension of continuous  functions defined on the initial space;
in this connection, see, for example, Theorem 3.6.21 of monograph
\cite{36}. For this approach, constructions of Section 8 are
essential. Of course, under corresponding conditions, we can use the
natural connection with the Wallman extension (see (\ref{9.25}) and
Proposition~\ref{p10.2}).

In this case, Proposition~\ref{p13.5} can be ``replaced'' (in some
sense) by statements similar to the above-mentioned Theorem 3.6.21
of \cite{36} (of course, this approach requires a correction, since
we consider ultrafilters of the measurable space). But, we use the
``more straight'' way with point of view of asymptotic analysis: we
construct the required continuous mapping by the limit passage (see
Proposition~\ref{p13.5}). We recall (\ref{10.20}). Then, by
(\ref{10.20}) and (\ref{13.5}) the mapping
\bfn\label{13.19}\mathfrak{H}_\ca[\tau] \circ (\ca-\mathrm{ult})[E]
= \Bigl(\mathfrak{H}_\ca[\tau] \bigl((E-\mathrm{ult})[x] \cap
\ca\bigl)\Bigl)_{x\in E}\in \mathbf{H}^E\efn is defined; moreover,
$\mathbf{h}\in \mathbf{H}^E.$

\begin{propos}\label{p13.6}{\TL} The equality $\mathbf{h}= \mathfrak{H}_\ca[\tau] \circ
(\ca-\mathrm{ult})[E]$ is valid.
\end{propos}

\begin{proof} Fix $x\in E.$ Then by (\ref{12.16})
$(E-\mathrm{ult})[x]\in \mathfrak{F}_\mathbf{u}[E].$ In addition, by
(\ref{12.17}) the obvious equality follows:
\bfn\label{13.20}\mathbf{h}(x) =
\mathfrak{H}[\tau]\bigl((E-\mathrm{ult})[x]\bigl).\efn Moreover, by
(\ref{10.20}) we obtain that
\bfn\label{13.20`}(\ca-\mathrm{ult})[E](x) = (E-\mathrm{ult})[x]
\cap \ca \in \mathbb{F}_o^*(\ca).\efn Then, by
Proposition~\ref{p13.3} and (\ref{13.20`}) we have the equality
chain
$$\mathfrak{H}_\ca[\tau]\bigl((\ca-\mathrm{ult})[E](x)\bigl) =
\mathfrak{H}[\tau]\bigl((E-\mathrm{ult})[x]\bigl) = \mathbf{h}(x).$$
So, $\bigl(\mathfrak{H}_\ca[\tau] \circ
(\ca-\mathrm{ult}[E]\bigl)(x) =
\mathfrak{H}_\ca[\tau]\bigl((\ca-\mathrm{ult})[E](x)\bigl) =
\mathbf{h}(x)$. Since the choice of $x$ was arbitrary, $\mathbf{h}=
\mathfrak{H}_\ca[\tau] \circ (\ca-\mathrm{ult})[E].$
\end{proof}

Since (\ref{10.7}) is valid, from Propositions~\ref{p13.5} and
\ref{p13.6}, we have the important corollary connected with
Proposition 5.2.1 of \cite{32}:
\bfn\label{13.21}\bigl(\mathbb{F}_o^*(\ca),\mathbf{T}_\ca^*[E],
(\ca-\mathrm{ult})[E],\mathfrak{H}_\ca[\tau]\bigl)\efn is a
compactificator, for which (in the considered case)
\begin{equation}\label{13.22}\begin{array}{c}(\mathbf{as})
[E;\mathbf{H};\tau;\mathbf{h};\mathcal{E}] =
\mathfrak{H}_\ca[\tau]^1\bigl((\mathrm{\mathbf{as}})
[E;\mathbb{F}_o^*(\ca);\mathbf{T}_\ca^*[E];\\(\ca-\mathrm{ult})[E];\mathcal{E}]\bigl)
\ \ \fo \mathcal{E}\in
\cp^\prime\bigl(\cp(E)\bigl);\end{array}\end{equation} in
(\ref{13.22}) we use Proposition~3.1 and Corollary~3.1 of \cite{41}.
In addition, $\cp^\prime(\ca) \subset \cp^\prime\bigl(\cp(E)\bigl).$
Therefore, by (\ref{13.22})
\begin{equation}\label{13.23}\begin{array}{c}(\mathbf{as})
[E;\mathbf{H};\tau;\mathbf{h};\mathcal{E}] =
\mathfrak{H}_\ca[\tau]^1\bigl((\mathrm{\mathbf{as}})
[E;\mathbb{F}_o^*(\ca);\mathbf{T}_\ca^*[E];\\(\ca-\mathrm{ult})[E];\mathcal{E}]\bigl)
\ \ \fo \mathcal{E}\in \cp^\prime(\ca).\end{array}\end{equation} In
(\ref{13.23}), we have the important particular case. We consider
this case in the following section.

\section{Ultrafilters as generalized solutions}
\setcounter{equation}{0}

We suppose that $E,\,(\mathbf{H},\tau),\,\mathbf{h},$ and $\ca$
satisfy  to the conditions of Section 12. We postulate
(\ref{12.10}). Finally, we postulate Condition~\ref{co11.1}.
Therefore, we can use constructions of the previous section. In
particular, (\ref{13.23}) is fulfilled (the more general property
(\ref{13.22}) is fulfilled too). In connection with (\ref{13.23}),
the obtaining of more simple representations of AS
\bfn\label{14.1}(\mathrm{\mathbf{as}})\bigl[E;\mathbb{F}_o^*(\ca);\mathbf{T}_\ca^*[E];
(\ca-\mathrm{ult})[E];\mathcal{E}\bigl],\ \ \mathcal{E}\in
\cp^\prime(\ca),\efn is important. For this goal, we use the natural
construction of Theorem 8.1 in \cite{40}. Namely, we have the
following

\begin{propos}\label{p14.1}{\TL} If $\mathcal{E}\in \cp^\prime(\ca),$
then
$(\mathrm{\mathbf{as}})\bigl[E;\mathbb{F}_o^*(\ca);\mathbf{T}_\ca^*[E];
(\ca-\mathrm{ult})[E];\mathcal{E}\bigl] = \mathbb{F}_o^*(\ca
|\,\mathcal{E}).$
\end{propos}

\begin{proof} Let $\cf \in
(\mathrm{\mathbf{as}})\bigl[E;\mathbb{F}_o^*(\ca);\mathbf{T}_\ca^*[E];
(\ca-\mathrm{ult})[E];\mathcal{E}\bigl].$ Then, by the corresponding
definition of Section 4 (see (\ref{4.3})) $\cf\in
\mathbb{F}_o^*(\ca)$ and, for some net $(D,\preceq,f) $ in the set
$E,$ \bfn\label{14.2}\bigl(\mathcal{E}\subset
(E-\mathrm{ass})[D;\preceq;f]\bigl)\,\&\,\bigl((D,\preceq,(\ca-\mathrm{ult})[E]\circ
f\bigl)\,{\stackrel{\mathbf{T}_\ca^*[E]}{\longrightarrow}}\,\cf\bigl).\efn
Fix $A\in \mathcal{E}.$ Then, by (\ref{14.2}) $A\in
(E-\mathrm{ass})[D;\preceq;f].$ Using (\ref{3.6}), we choose $d_1\in
D$ such that $\fo \delta\in D$ \bfn\label{14.3}(d_1 \preceq \delta)
\Longrightarrow \bigl(f(\delta)\in A\bigl).\efn Of course, $A\in
\cp(E).$ And what is more, $A\in \cf.$ Indeed, let us assume the
contrary: \bfn\label{14.4}A\in \mathcal{E}\setminus \cf.\efn Recall
that $\mathcal{E}\subset \ca.$ Therefore, $A\in \ca.$ By
(\ref{10.1}) and (\ref{14.4}) we have the inclusion $E\setminus A
\in \cf.$ Then, by (\ref{5.3}) $\Phi_\ca(E\setminus A) \in
(\mathbb{UF})[E;\ca].$ In particular (see (\ref{7.3})),
\bfn\label{14.5}\Phi_\ca(E\setminus A)\in \mathbf{T}_\ca^*[E].\efn
Moreover, by (\ref{5.2}) $\cf\in \Phi_\ca(E\setminus A).$ Using
(\ref{14.5}), we obtain that \bfn\label{14.6}\Phi_\ca(E\setminus
A)\in N_{\mathbf{T}_\ca^*[E]}^o(\cf),\efn where
$N_{\mathbf{T}_\ca^*[E]}^o(\cf)\subset
N_{\mathbf{T}_\ca^*[E]}(\cf).$ From (\ref{14.6}) and the second
statement of (\ref{14.2}) we have the following property: there
exists $d_2\in D$ such that $\fo \delta \in D$ \bfn\label{14.7}(d_2
\preceq \delta) \Longrightarrow
\Bigl(\bigl((\ca-\mathrm{ult})[E]\circ f\bigl)(\delta)\in
\Phi_\ca(E\setminus A)\Bigl).\efn By axioms of DS there exists
$d_3\in D$ for which $d_1\preceq d_3$ and $d_2 \preceq d_3.$ By
(\ref{14.3}) $f(d_3) \in A.$ Moreover, by (\ref{14.7})
$$\bigl((\ca-\mathrm{ult})[E]\circ f\bigl)(d_3)\in
\Phi_\ca(E\setminus A).$$ By (\ref{10.20})
$\bigl((\ca-\mathrm{ult})[E]\circ f\bigl)(d_3) =
(E-\mathrm{ult})[f(d_3)] \cap \ca.$ Therefore,
$$(E-\mathrm{ult})[f(d_3)] \cap \ca \in \Phi_\ca(E\setminus A).$$
From (\ref{5.2}), the inclusion $E\setminus A\in
(E-\mathrm{ult})[f(d_3)] \cap \ca$ follows; in particular,
$E\setminus A \in (E-\mathrm{ult})[f(d_3)].$ By (\ref{3.8})
$f(d_3)\in E\setminus A.$ So, $$\bigl(f(d_3)\in
A\bigl)\,\&\,\bigl(f(d_3)\in E\setminus A\bigl).$$ We have the
obvious contradiction. This contradiction means that (\ref{14.4}) is
impossible. So, $A\in \cf.$ Since the choice of $A$ was arbitrary,
the inclusion $\mathcal{E}\subset \cf$ is established. Then (see
(\ref{10.21`})), $\cf \in \mathbb{F}_o^*(\ca |\,\mathcal{E}).$ So,
we obtain the inclusion
\bfn\label{14.8}(\mathrm{\mathbf{as}})[E;\mathbb{F}_o^*(\ca);\mathbf{T}_\ca^*[E];
(\ca-\mathrm{ult})[E];\mathcal{E}]\subset \mathbb{F}_o^*(\ca
|\,\mathcal{E}).\efn Choose arbitrary $\mathcal{V}\in
\mathbb{F}_o^*(\ca |\,\mathcal{E}).$ Then, by (\ref{10.21`})
$\mathcal{V}\in \mathbb{F}_o^*(\ca)$ and $\mathcal{E}\subset
\mathcal{V}.$ By Proposition~\ref{p10.4} and \cite[(3.3.7)]{32}, for
some net $(\mathbb{D},\sqsubseteq, g)$ in the set $E,$ the
convergence
\bfn\label{14.9}\bigl(\mathbb{D},\sqsubseteq,(\ca-\mathrm{ult})[E]
\circ
g\bigl)\,{\stackrel{\mathbf{T}_\ca^*[E]}{\longrightarrow}}\,\mathcal{V}\efn
is fulfilled. Now, we use axiom of choice. Fix $\Om\in \mathcal{E}.$
Then, by the choice of $\mathcal{V}$ the inclusion $\Om\in
\mathcal{V}$ is fulfilled. Of course, by (\ref{5.3})
\bfn\label{14.10}\Phi_\ca(\Om)\in (\mathbb{UF})[E;\ca];\efn in
addition, by (\ref{5.2}) $\mathcal{V}\in \Phi_\ca(\Om).$ Since by
(\ref{7.3}) and (\ref{14.10}) $\Phi_\ca(\Om)\in
\mathbf{T}_\ca^*[E],$ we have the inclusion
\bfn\label{14.11}\Phi_\ca(\Om)\in
N_{\mathbf{T}_\ca^*[E]}^o(\mathcal{V}).\efn From (\ref{14.9}) and
(\ref{14.11}), we have the property: for some $\mathbf{d}\in
\mathbb{D},$ we obtain that $\fo \delta\in \mathbb{D}$
\bfn\label{14.12}(\mathbf{d}\sqsubseteq \delta)
\Longrightarrow\Bigl(\bigl((\ca-\mathrm{ult})[E] \circ
g\bigl)(\delta) \in \Phi_\ca(\Om)\Bigl).\efn  From (\ref{10.20}) and
(\ref{14.12}), we have the following property: $\fo \delta\in
\mathbb{D}$ \bfn\label{14.13}(\mathbf{d}\sqsubseteq \delta)
\Longrightarrow\bigl((E-\mathrm{ult})[g(\delta)] \cap \ca \in
\Phi_\ca(\Om)\bigl).\efn By (\ref{5.2}) and (\ref{14.13}) we obtain
that, for $\delta\in \mathbb{D}$ with the property
$\mathbf{d}\sqsubseteq \delta,$ the inclusion $\Om\in
(E-\mathrm{ult})[g(\delta)] \cap \ca$ is valid and, as a corollary,
by (\ref{3.8}) $g(\delta)\in \Om.$ So, $\Om\in \cp(E)$ and $$\exists
d_1\in \mathbb{D}\ \ \fo d_2\in \mathbb{D}\ \ (d_1\sqsubseteq  d_2)
\Longrightarrow \bigl(g(d_2)\in \Om\bigl).$$ Then, by (\ref{3.6})
$\Om\in (E-\mathrm{ass})[\mathbb{D};\sqsubseteq ;g].$ Since the
choice of $\Om$ was arbitrary, the inclusion
\bfn\label{14.14}\mathcal{E}\subset
(E-\mathrm{ass})[\mathbb{D};\sqsubseteq ;g]\efn is established. So,
by (\ref{14.9}) and (\ref{14.14}) we obtain that the net
$(\mathbb{D},\sqsubseteq,g)$ in the set $E$ has the following
properties: $$\bigl(\mathcal{E}\subset
(E-\mathrm{ass})[\mathbb{D};\sqsubseteq
;g]\bigl)\,\&\,\Bigl(\bigl(\mathbb{D},\sqsubseteq,(\ca-\mathrm{ult})[E]
\circ
g\bigl)\,{\stackrel{\mathbf{T}_\ca^*[E]}{\longrightarrow}}\,\mathcal{V}\Bigl).$$
By definition of Section 4 (see (\ref{4.3})) $\ \ \mathcal{V}\in
(\mathrm{\mathbf{as}})\bigl[E;\mathbb{F}_o^*(\ca);\mathbf{T}_\ca^*[E];
(\ca-\mathrm{ult})[E];\mathcal{E}\bigl].$ So, the inclusion
$$\mathbb{F}_o^*(\ca |\,\mathcal{E})\subset
(\mathrm{\mathbf{as}})\bigl[E;\mathbb{F}_o^*(\ca);\mathbf{T}_\ca^*[E];
(\ca-\mathrm{ult})[E];\mathcal{E}\bigl]$$ is established. Using
(\ref{14.8}), we have the required equality
$$\hspace{2.3cm}(\mathrm{\mathbf{as}})\bigl[E;\mathbb{F}_o^*(\ca);\mathbf{T}_\ca^*[E];
(\ca-\mathrm{ult})[E];\mathcal{E}\bigl] = \mathbb{F}_o^*(\ca
|\,\mathcal{E}).\hspace{2.4cm}$$ \end{proof}

From (\ref{13.23}) and Proposition~\ref{p14.1}, we have the following

\begin{theor}\label{t14.1}{\TL} If $\mathcal{E}\in \cp^\prime(\ca),$ then, AS in
$(\mathbf{H},\tau)$ with constraints of the asymptotic character
defined by $\mathcal{E}$ is realized by the rule
$$(\mathrm{\mathbf{as}})[E;\mathbf{H};\tau;\mathbf{h};\mathcal{E}] =
\mathfrak{H}_\ca[\tau]^1\bigl(\mathbb{F}_o^*(\ca\,
|\,\mathcal{E})\bigl).$$
\end{theor}

We note that, in Theorem~\ref{t14.1}, the set $\mathbb{F}_o^*(\ca  |\,\mathcal{E})$ plays
the role of the set of admissible generalized solutions.

\section{Some remarks}
\setcounter{equation}{0}

In our investigation, one approach to the representation of AS and
approximate solutions is considered. This very general approach
requires the employment of constructions of nonsequential asymptotic
analysis. This is connected both with the necessity of validity of
``asymptotic constraints'' and with the general type of the
convergence in TS.  We fix a nonempty set of usual solutions (the
solution space), the estimate space, and an operator from the
solution space into the estimate space. In the estimate space, a
topology is given. Then, under very different constraints, we can
realize in this space both usual attainable elements and AE. But, if
usual attainable elements are defined comparatively simply (in the
logical relation), then AE are constructed very difficult. For last
goal, extensions of the initial space are used. In addition, the
corresponding spaces of GE are constructed. Ultrafilters of the
initial space can be used as GE. But, the realizability problem
arise: free ultrafilters are ``invisible''. In addition, free
ultrafilters realize limit attainable elements which nonrealizable
in the usual sense. In this connection, we propose to use
ultrafilters of (nonstandard) measurable space; we keep in mind
spaces with an algebra of sets. But, it is possible to consider the
more general constructions with the employment of ultrafilters. In
our investigation, ultrafilters of lattices of sets are used. On
this basis, the interesting connection with the Wallman extension in
general topology arises.

It is possible that the proposed approach motivated by problems of
asymptotic analysis can be useful in other constructions of
contemporary mathematics.

{
\renewcommand{\refname}{{\rm\centerline{The bibliography}}}

\end{document}